\documentclass[12pt]{amsart}
\usepackage[T2A]{fontenc}
\usepackage[russian, english, shorthands=off]{babel}
\usepackage{amsfonts}
\usepackage{amssymb}
\usepackage{quiver}
\usepackage{epigraph}
\usepackage{euscript}
\usepackage{tikz-cd}
\tikzcdset{scale cd/.style={every label/.append style={scale=#1},
    cells={nodes={scale=#1}}}}
\usepackage{mathabx,epsfig}
\usepackage{amsmath}
\usepackage{amsthm}
\usepackage{comment}
\usepackage{color}
\definecolor{darkspringgreen}{rgb}{0.09, 0.45, 0.27}

\renewcommand{\leq}{\leqslant}

\newcommand{\BC}{{\mathbb{C}}} \newcommand{\CN}{{\mathcal{N}}}

\newcommand{\F}{{\mathcal F}}

\tikzcdset{scale cd/.style={every label/.append style={scale=#1},
    cells={nodes={scale=#1}}}}
\usepackage[pdftex,bookmarks=false,colorlinks=true,citecolor=darkspringgreen,debug=true,
  naturalnames=true,pdfnewwindow=true]{hyperref}
\newtheorem{thm}{Theorem}[subsection]

\newtheorem{cor}[thm]{Corollary}
\newtheorem{remark}[thm]{Remark}
\newtheorem{lemma}[thm]{Lemma}
\newtheorem{prop}[thm]{Proposition}

\newtheorem{warning}[thm]{Warning}

\newtheorem{thmx}{Theorem}

\DeclareMathAlphabet\mathbfcal{OMS}{cmsy}{b}{n}

\newcommand{\nc}{\newcommand}
\newcommand\iso{\,\vphantom{j^{X^2}}\smash{\overset{\sim}{\vphantom{\rule{0pt}{0.20em}}\smash{\longrightarrow}}}\,}
\nc{\CH}{\mathcal{H}}
\nc{\CB}{\mathcal{B}}
\nc{\on}{\operatorname}
\nc{\CF}{\mathcal{F}}
\nc{\BZ}{\mathbb{Z}}
\nc{\La}{\Lambda}
\nc{\J}{{\bf{J}}}
\nc{\K}{{\bf{K}}}
\nc{\He}{{\mathbfcal{H}}}
\nc{\R}{{\bf{R}}}
\nc{\CO}{\mathcal{O}}
\nc{\la}{\lambda}

\newtheorem{defn}[thm]{Definition}

\theoremstyle{proof}
\makeatletter
\renewcommand{\subsection}{\@startsection{subsection}{2}{0pt}{-3ex
plus -1ex minus -0.2ex}{-2mm plus -0pt minus
-2pt}{\normalfont\bfseries}} \makeatother

\numberwithin{equation}{subsection}
\allowdisplaybreaks
\usepackage{mathdots}

\begin{document}

\title{A geometric realization of the asymptotic affine Hecke algebra}

\author{Roman Bezrukavnikov}
\address{Department of Mathematics
Massachusetts Institute of Technology
\newline
77 Massachusetts Avenue,
Cambridge, MA 02139,
USA
}
\email{bezrukav@math.mit.edu}

\author{Ivan Karpov}
\address{Department of Mathematics
Massachusetts Institute of Technology
\newline
77 Massachusetts Avenue,
Cambridge, MA 02139,
USA
}
\email{karpov57@mit.edu}

\author{Vasily Krylov}
\address{Department of Mathematics
Harvard University and CMSA
\newline
1 Oxford Street,
Cambridge, MA 02138,
USA}
\email{vkrylov@math.harvard.edu, krylovasya@gmail.com}

\maketitle

\thispagestyle{empty}

\begin{abstract}
A key tool for the study of an affine Hecke algebra $\mathcal{H}$ is provided
by Springer theory of the Langlands dual group via the realization of
$\mathcal{H}$ as equivariant $K$-theory of the Steinberg variety.
We prove a similar geometric description for Lusztig's asymptotic affine
Hecke algebra $J$ identifying it with the sum of equivariant $K$-groups of the
squares of ${\mathbb C}^*$-fixed points in the Springer fibers,
as conjectured by Qiu and Xi (the same result was also obtained by Oron Popp using different methods).
As an application, we give a new geometric proof of Lusztig's parametrization of irreducible representations of $J$. We also reprove Braverman-Kazhdan's spectral description of $J$. As another application, we prove a description of the cocenters of $\mathcal{H}$ and $J$ conjectured by the first author with Braverman, Kazhdan and
Varshavsky.  
The proof is based on a new algebraic description of $J$, which
may be of independent interest.
\end{abstract}

\setcounter{page}{1}
\section{Introduction}


\subsection{Affine Hecke algebra and asymptotic affine Hecke algebra} Let $G$ be a reductive algebraic group over $\mathbb{C}$. Let $W_f$ be the Weyl group of $G$ and let $\Lambda$ be the coweight lattice of $G$. The (extended) affine Weyl group  $W$ of $G$ is defined as  $W:=W_f \ltimes \Lambda$. Let $\mathcal{H}=\mathcal{H}_G=\mathcal{H}(W)$ be the affine Hecke algebra of $G$, recall that $\mathcal{H}$ is a deformation of the group algebra of $W$ over $\mathbb{Z}[{\bf{v}},{\bf{v}}^{-1}]$, where ${\bf{v}}={\bf{q}}^{1/2}$ and ${\bf{v}}, {\bf{q}}$ are formal variables. Let $\He$ be the complexification of $\mathcal{H}$. For $q \in \BC^*$ we will denote by $\He_q := \He \otimes_{\BC[{\bf{v}},{\bf{v}}^{-1}]} \BC_q$ the corresponding specialization of $\He$.

Kazhdan-Lusztig and Ginzburg showed that $\mathcal{H}$ has a geometric realization: it can be identified with the convolution algebra $K_{G^\vee \times \mathbb{C}^*}(\widetilde{\mathcal N} \times_{\mathcal{N}} \widetilde{\mathcal N})$, where $\widetilde{\mathcal{N}} \rightarrow \mathcal{N}$ is the Springer resolution for the Langlands dual group $G^\vee$. Kazhdan and Lusztig used the relation between $\He$ and geometry of the Steinberg variety $\widetilde{\mathcal N} \times_{\mathcal{N}} \widetilde{\mathcal N}$ to
parametrize simple $\He_q$-modules for $q$ not a root of unity. It turns out that this parametrization extends to any $q \in \mathbb{C}^*$ that is not a root of the Poincar\'e polynomial $P_W$ of $W$ (see \cite{Xi3}), moreover, the parametrization {\emph{does not}} depend on $q$.

Lusztig introduced the (based) algebra $J$ called the asymptotic Hecke algebra, its basis is parametrized by $W$, and there is an injective homomorphism $\phi\colon {\mathcal{H}} \hookrightarrow J \otimes_{\mathbb{Z}} \mathbb{Z}[\bf{v},{\bf{v}}^{-1}]$. We set $\J:=J \otimes_{\mathbb{Z}} \mathbb{C}$.

Informally, one can think about the algebra $J$ as the ``limit'' of $\mathcal{H}$ as $q$ goes to $0$. 
As discovered by Lusztig, representations of $\J$ are closely related
to those of $\He_q$, more precisely 
pulling back an {\emph{irreducible}} $\J$-module $E$ under the homomorphism $\phi_q\colon \He_q \hookrightarrow \J$, $q \in \BC^*$ one gets a so-called {\emph{standard}} module $K_q$ over 
$\He_q$\footnote{see \cite{Lcells3}, \cite{Lcells4} for the case of generic $q$;
more generally,  from \cite[Theorem 3.2]{Xi3} and \cite[Corollary 2.6]{BK} it follows that this holds for $q$ such that $P_W(q) \neq 0$. We
reprove and strengthen this fact below.} This connection explains the observation above that the parametrization of modules over the Hecke algebra $\He_q$ remains independent of the parameter $q$, as long as $q$ is not a root of $P_W$.

The goal of this paper, achieved in Theorem \ref{thm_A}, is 
to prove a similar geometric description of the algebra $J$ conjectured in \cite{QX} and derive some applications (Theorems \ref{thm_B}, \ref{thm_C} etc.). We note that Theorem \ref{thm_A} was also proved by Oron Popp in his PhD thesis \cite{Pr} by a different method.

The coherent realization of $\mathcal H$ is a manifestation of local Langlands duality used in the proof of a special case
of local Langlands conjectures \cite{KL}. It would be very interesting
to find an interpretation and a generalization for Theorem \ref{thm_A} in that context. In particular, according to a recent insight of Braverman and Kazhdan \cite{BK}, $\J$ can be viewed as a ring of Iwahori bi-invariant distributions on the $p$-adic group intermediate 
between the algebra of compactly supported distributions (well-known
to be isomorphic to a specialization of $\He$) and
Harish-Chandra Schwartz space of tempered distributions.
In \cite{BK} one also finds a generalization of this definition to 
not necessarily Iwahori bi-invariant distribution. We hope that 
Theorem \ref{thm_A} admits an (at least conjectural) extension to that
generality.


\subsection{Filtrations on $\CH$ and two-sided cells} Let us now return to the identification $K_{G^\vee \times \mathbb{C}^*}(\widetilde{\CN} \times_{\CN} \widetilde{\CN}) \simeq \mathcal{H}$. The group $G^\vee$ acts on $\CN$ with a finite number of orbits ordered by the closure order. For every orbit $\mathbb{O}_e \subset \mathcal{N}$ consider its closure $\overline{\mathbb{O}}_e$ and take its preimage in $\widetilde{\CN} \times_{\CN} \widetilde{\CN}$. The $G^\vee \times \mathbb{C}^*$-equivariant $K$-theory of this preimage is a term in  the filtration of  $\mathcal{H}$ by the two-sided ideals $\mathcal{H}_{\leqslant e}$ indexed by nilpotent orbits ${\mathbb{O}}_e$. Each subquotient $\CH_{e}:=\mathcal{H}_{\leqslant e}/\mathcal{H}_{<e}$ is a bimodule over $\mathcal{H}$ that is clearly isomorphic to $K_{Z_e \times \BC^*}(\mathcal{B}_e \times \mathcal{B}_e)$, where $\mathcal{B}_e$ is the fiber of the Springer resolution over $e$ called the Springer fiber, $Z_e$ is the reductive part of the centralizer of $e$ in $G^\vee$ and the $\BC^*$-action on $\mathcal{B}_e$ is defined using the Jacobson-Morozov Theorem. Set $R_e:=K_{Z_e \times \BC^*}(\on{pt})$, it acts naturally on $\CH_{e}=K_{Z_e \times \BC^*}(\mathcal{B}_e \times \mathcal{B}_e)$.

\begin{remark}
Note that $\CH_e$ is also a ring (without a unit element). To see the ring structure geometrically, we need to identify $\CH_e$ with $K_{Z_e \times \BC^*}(\La_e \times \La_e;\mathcal{B}_e \times \mathcal{B}_e)$, where $\La_e$ is the Slodowy variety corresponding to $e$ and $K_{Z_e \times \BC^*}(\La_e \times \La_e;\mathcal{B}_e \times \mathcal{B}_e)$ is the $K$-group of the category of $Z_e \times \BC^*$-equivariant coherent sheaves on $\La_e \times \La_e$ supported on $\mathcal{B}_e \times \mathcal{B}_e \subset \La_e \times \La_e$, see \cite[Theorem B.2]{TX}.
\end{remark}

There is a way to describe the filtration $\CH_{\leqslant e} \subset \CH$ above algebraically. Let $\{C_w\,|\, w \in W\}$ be the {\emph{canonical}} basis of $\mathcal{H}$ introduced in \cite{KL0} and generalized by Lusztig to the case of extended Weyl groups in \cite{Lcells3}. 
Let $\CH_{\leqslant w}$ be the minimal based (i.e., spanned over $\mathbb{Z}$ by a subset of the canonical basis) 
two-sided ideal of $\mathcal{H}$ that contains $C_w$. In \cite{Lcells1} Lusztig introduced a notion of the two-sided cell in $W$ that can be characterized as follows: two elements $w,w' \in W$ lie in the same two-sided cell iff $\CH_{\leqslant w}=\CH_{\leqslant w'}$. We obtain a partition of $W$ into two-sided cells.  To a cell ${\bf{c}} \subset W$ there thus corresponds the two-sided ideal 
$\mathcal{H}_{\leqslant {\bf{c}}}$.
This is the so-called \textit{cell filtration}.  Let us denote by $\mathcal H_{\bf c}$ the corresponding subquotient.

We obtain a partial order on the set of two-sided cells defined as follows: ${\bf{c}}' \leqslant {\bf{c}}$ iff $\CH_{\leqslant {\bf{c}}'} {{\subset}} \CH_{\leqslant {\bf{c}}}$. By a result of Lusztig (see \cite{Lcells4}), the set of $G^\vee$-orbits in $\CN$ is in bijection with the set of two-sided cells in $W$. Let us denote this bijection by $L$. It then follows from 
results of  Xi \cite{Xi1} and of the first author
\cite[Theorem 55, \S 11.3]{B3}  that if ${\bf{c}}=L(e)$ then $\mathcal{H}_{\leqslant {\bf{c}}} = \mathcal{H}_{\leqslant e}$ 
and the $L$ bijection above is order-preserving. 
Thus the cell filtration will also be referred to as the \textit{geometric filtration}.

\subsubsection{The direct sum decomposition  for the ring $J$} The ring $J$ can be decomposed as the direct sum $J=\bigoplus_{{\bf{c}}} J_{\bf{c}}$, where ${\bf{c}} \subset W$ runs over the set of two-sided cells in $W$.  

For ${\bf{c}}=L(e)$ we have a natural homomorphism of algebras $K_{Z_e \times \BC^*}(\La_e \times \La_e;\mathcal{B}_e \times \mathcal{B}_e)=\mathcal{H}_{{\bf{c}}} \xrightarrow{\phi_{{\bf{c}}}} J_{{\bf{c}}} \otimes \mathbb{Z}[{\bf{v}},{\bf{v}}^{-1}]$. This homomorphism becomes an isomorphism after tensoring by $\BC({\bf{v}})$.
Let $\phi^{\bf{c}}\colon \CH \rightarrow J_{\bf{c}} \otimes \mathbb{Z}[{\bf{v}},{\bf{v}}^{-1}]$ be the composition of $\phi\colon \mathcal{H} \rightarrow J \otimes \mathbb{Z}[{\bf{v}},{\bf{v}}^{-1}]$ and the projection onto $J_{\bf{c}} \otimes \BZ[{\bf{v}},{\bf{v}}^{-1}]$.



\subsection{Main results}

\subsubsection{Geometric description of $J_e$}

One of the main results of this paper is a geometric description of $J_e$ conjectured by Qiu and Xi in \cite{QX}, see also Propp's thesis \cite[Theorem 1.5.2]{Pr} for another proof. 

\begin{thmx}\label{thm_A}
There exists an isomorphism of rings $J_{e} \otimes \mathbb{Z}[{\bf{v}},{\bf{v}}^{-1}] \simeq K_{Z_e \times \mathbb{C}^*}(\CB_e^{\mathbb{C}^*} \times \CB_e^{\mathbb{C}^*})$.
\end{thmx}

\begin{remark}
Note that $J$ does not depend on $q$, this corresponds to the fact that $\BC^*$ acts trivially on $\CB_e^{\BC^*} \times \CB_e^{\BC^*}$. 
 Specializing at 
 ${\bf{v}}=1$ we obtain an identification $J_e \simeq K_{Z_e}(\CB_e^{\mathbb{C}^*} \times \CB_e^{\mathbb{C}^*})$; however,
 specializing at ${\bf{v}}=-1$ yields another isomorphism. 
\end{remark}

\begin{remark} In view of Theorem \ref{thm_A}, the category $\operatorname{Coh}_{Z_{G^\vee}(e) }(\CB_e^{\mathbb C^*} \times \CB_e^{\mathbb C^*})$ (or $\operatorname{Coh}_{Z_e}(\CB_e^{\mathbb C^*} \times \CB_e^{\mathbb C^*})$) can be viewed as a categorification of the ring
$J_e$. 
Other categorifications of $J_e$ were proposed in \cite{LJcat}, \cite{BO}, \cite{BL} (these three are equivalent), and partially in \cite{D}. We do not know how these different categorifications are related.
\end{remark}

Let us briefly describe the idea of the proof of Theorem  \ref{thm_A}. Recall the (injective) homomorphisms $\phi\colon \CH \hookrightarrow J \otimes \BZ[{\bf{v}},{\bf{v}}^{-1}]$, $\phi_{\bf{c}}\colon \CH_{\bf{c}} \hookrightarrow J_{\bf{c}} \otimes \BZ[{\bf{v}},{\bf{v}}^{-1}]$. They induce $\CH-J_{{\bf{c}}} \otimes \BZ[{\bf{v}},{\bf{v}}^{-1}]$-bimodule and  $\CH_{\bf{c}}-J_{\bf{c}}$-bimodule structures on $J_{{\bf{c}}} \otimes \BZ[{\bf{v}},{\bf{v}}^{-1}]$. Moreover, considered as a left $\CH$ (resp. $\CH_{\bf{c}}$)-module, it is isomorphic to $\CH_{{\bf{c}}}$ (see Corollary \ref{cor_H_mod_is_iso}).   
We then prove the following theorem.

\begin{thmx}\label{thm_B}
The right action of $J_{\bf{c}} \otimes \BZ[{\bf{v}},{\bf{v}}^{-1}]$ induces isomorphisms of algebras 
\begin{equation*}
J_{{\bf{c}}} \otimes \BZ[{\bf{v}}, {\bf{v}}^{-1}] \iso \operatorname{End}_{\mathcal{H}}(J_{\mathbf{c}} \otimes \BZ[{\bf{v}}, {\bf{v}}^{-1}])^{\mathrm{opp}} \simeq \operatorname{End}_{\mathcal{H}}(\CH_{\bf{c}})^{\mathrm{opp}} \simeq \operatorname{End}^{R_e}_{\mathcal{H}_{\bf{c}}}(\CH_{\bf{c}})^{\mathrm{opp}},
\end{equation*}
where $\operatorname{End}^{R_e}_{\mathcal{H}_{\bf{c}}}(\CH_{\bf{c}})$ stands for the set of endomorphisms, which are both $R_e$-and $\CH_{{\bf{c}}}$-linear. 

\end{thmx}

\begin{warning}
Note that $\operatorname{End}_{\mathcal{H}_{\bf{c}}}(\CH_{\bf{c}})^{\mathrm{opp}}$ is {\emph{not}} isomorphic to $\mathcal{H}_{\bf{c}}$ because $\mathcal{H}_{\bf{c}}$ is a ring without a unit. We only have the natural embedding $\CH_{\bf{c}} \hookrightarrow \operatorname{End}_{\mathcal{H}_{\bf{c}}}(\CH_{\bf{c}})^{\mathrm{opp}}$ induced by the right multiplication.
\end{warning}

Thus, in order to establish the identification between $K_{e}:=K_{Z_e \times \BC^*}(\CB_e^{\BC^*} \times \CB_e^{\BC^*})$ and $J_{e} \otimes \BZ[{\bf{v}},{\bf{v}}^{-1}]$ we need to construct a homomorphism  $K_e \rightarrow \operatorname{End}_{\mathcal{H}}(\CH_{\bf{c}})^{\mathrm{opp}}$ and prove that it is an isomorphism. To construct the aforementioned homomorphism, we need to find a $\CH-K_e$-bimodule that is isomorphic to $\CH_{e}=K_{Z_e \times \BC^*}(\CB_e \times \CB_e)$ as a $\CH$-module. The natural candidate is $\mathcal{F}_e:=K_{Z_e \times \BC^*}(\mathcal{B}_e \times \mathcal{B}_e^{\BC^*})$. We prove that $\CF_e$ is indeed isomorphic to $\CH_e$.  To see that, we use the Bialynicki-Birula type decomposition of $\CB_e$ by the attractors via the $\BC^*$-action
studied by De Concini, Lusztig and Procesi in \cite{DLP}. In more detail, recall that if $X$ is a smooth projective variety with a $\BC^*$-action then by the Bialynicki-Birula theorem attractors to the connected components of $X^{\BC^*}$ are affine fibrations. It follows that the decomposition of $X$ by the attractors induces a {\emph{filtration}} on the (complexified) $K$-theory of $X$ with associated graded being isomorphic to the $K$-theory of $X^{\BC^*}$ if $H_*(X^{\BC^*},\mathbb{C})$ is generated by algebraic cycles. Variety $\CB_e$ is projective but not smooth, so the Bialynicki-Birula theorem can not be applied to it directly. On the other hand by the results of \cite{DLP}, $\CB_e^{\BC^*}$ is smooth and attractors in $\CB_e$ of the components of $\CB_e^{\BC^*}$ are indeed affine fibrations. So, we obtain a filtration on $K_{Z_e \times \BC^*}(\CB_e \times \CB_e^{\BC^*})$ with associated graded being isomorphic to $K_{Z_e \times \BC^*}(\CB_e^{\BC^*} \times \CB_e^{\BC^*})$. We then produce a splitting of this filtration. To this end we consider closures of the subvarieties $\on{Attr}_{F}:=\{(x,y) \in \CB_e \times F\,|\,\underset{t \rightarrow 0}{\on{lim}}\, t \cdot x = y\}$, where $F \subset \CB_e^{\BC^*}$ are connected components of $\CB_e^{\BC^*}$. 
We then consider natural morphisms $\overline{\pi}_F\colon \overline{\on{Attr}}_F \rightarrow F$, $\overline{\iota}_F\colon \overline{\on{Attr}}_{F} \rightarrow \CB_e$ and the splitting is given by $\bigoplus_{F} (\on{Id}_{\CB_e} \times \overline{\iota}_F)_*(\on{Id}_{\CB_e} \times \overline{\pi}_{F})^*$. In other words, the splitting is given by the natural correspondences $\CB_e \times \CB_e^{\BC^*} \xleftarrow{\on{Id}_{\CB_e} \times \overline{\pi}_F} 
\CB_e \times \overline{\on{Attr}}_{F} \xrightarrow{\on{Id}_{\CB_e} \times \overline{\iota}_F} \CB_e \times \CB_e$.



So, we obtain a homomorphism $\theta\colon K_e \rightarrow \on{End}_{\CH}(\CF_e)^{\mathrm{opp}}=J_e \otimes \BZ[{\bf{v}},{\bf{v}}^{-1}]$, and it remains to check that it is an isomorphism. Injectivity is easy since we have an identification $F\colon K_e \iso \CF_e$ (of right $K_e$-modules), and our homomorphism is then given by the action of $K_e$ on itself via right multiplication. To check surjectivity, it is enough to show that any element of $\on{End}_{\CH}(\CF_e)$ is uniquely determined by its value on $F(1)$. To see that we use the identifications 
\begin{equation*}
K_e\simeq_{K_e} \CF_e \simeq_{\CH_e} \CH_e \simeq_{\CH_e} J_e \otimes \BZ[{\bf{v}},{\bf{v}}^{-1}]
\end{equation*}
and the only thing to check is that the image of $1 \in K_e$ in $J_e \otimes \mathbb{Z}[{\bf{v}},{\bf{v}}^{-1}]$ after these identifications is a (left) invertible element (we use Theorem \ref{thm_B}). It is clearly right invertible, but $\J_e$ is left-Noetherian, so right invertible elements are left invertible. 


\subsubsection{Representation theory of $\J$ from the geometric perspective}

Realization of $J_e$ as the convolution algebra $K_{Z_e}(\CB_e^{\BC^*} \times \CB_e^{\BC^*})$ allows one to apply geometric methods to  representation theory of $\J$  (cf. \cite{KL} and \cite[Sections 7, 8]{CG}).
Our geometric realization has particularly favorable properties since  $\CB_e^{\BC^*}$ is smooth and projective, and its homology is generated by algebraic cycles (\cite{DLP}). 

In \cite{Lcells3} Lusztig classified irreducible modules over the algebra $\J$. He proved that these modules are parametrized by the triples $(s,e,\rho)$ (up to conjugation), where $e \in \mathcal{N}$ is a nilpotent element, $s$ is a semisimple element of $Z_e$ and $\rho$ is an irreducible representation of $\Gamma^s_e:=Z_{G^\vee}(s,e)/Z_{G^\vee}(s,e)^{0}$. Irreducible $\J$-module corresponding to $(s,e,\rho)$ will be denoted by $E(s,e,\rho)$. For a variety $X$ equipped with the action of an algebraic group $H$ we set $\K_{H}(X):=K_H(X) \otimes_{\mathbb{Z}} \mathbb{C}$. We reprove Lusztig's description using our geometric approach by proving the following theorem. 

\begin{thmx}\label{thm_C}
(a) Irreducible modules over $\J_e=\K_{Z_e}(\CB_e^{\BC^*} \times \CB_e^{\BC^*})$ are all of the form $\K(\CB_e^{\BC^*,s})_{\rho}$.

(b) We have $E(s,e,\rho)=\K(\CB_e^{\BC^*,s})_{\rho}$.   
\end{thmx}

Let us briefly outline the argument. Part $(a)$ of the Theorem follows from  general considerations about modules over convolution algebras (see Section \ref{conv} for the details). To prove part $(b)$, we recall that Lusztig's parametrization works as follows (see \cite{Lcells4}):  $E(s,e,\rho)$ is the unique irreducible module over $\J$ such that $\phi_q^*E(s,e,\rho)$ is isomorphic to $K(s,e,\rho,q):=\K(\CB_e^{sq})$ for a generic $q$. So, in order to prove part $(b)$ we need to check that $\phi^*_q(\K(\CB_e^{\BC^*,s})_{\rho}) \simeq \K(\CB_e^{qs})_{\rho}$ for a generic $q$. We prove that this holds for {\emph{any}} $q$. This implies part $(b)$ of the Theorem and also shows that $\phi_q^*E(s,e,\rho)=K(s,e,\rho,q)$ for every $q \in \BC^*$ (this result is new for $q$ being a root of $P_W$).


\subsubsection{Proof of Braverman-Kazhdan's theorem}
In \cite{BK} the authors described the algebra $\J_e$ in spectral terms by formulating a version of the matrix Paley-Wiener theorem for $\J_e$ (see \cite[Theorem 1.8 (3)]{BK}). We reprove their theorem using our geometric approach.\footnote{We were not able to follow some steps in the argument
of  \cite{BK}  (see footnotes $2$, $3$ below), we fill in the details  not found in {\em loc. cit} using
our present geometric methods.  Stefan Dawydiak \cite{rigid} developed an alternative algebraic approach
to completing the proof of that theorem.}
Let us first recall the content of \cite[Theorem 1.8 (3)]{BK}.

Pick a semisimple element $s \in Z_e$. Let $M^\vee \subset G^\vee$ be a Levi containing $s$ and such that $e \in \on{Lie}M^\vee$. In \cite[Section 1.2]{BK} the authors consider a certain family of $\He_q$-modules over $Z_M:=Z(M^\vee)^{0}$ that we will denote by $\mathcal{V}(Z_M,s,\rho,q)$. Its fiber over $1 \in Z_M$ is $\on{Ind}_{\He_M}^{\He}K(s,e,\rho,q)$.

\begin{thm}[Braverman-Kazhdan]\label{BK_main_thm_intro}
 Let $\mathcal S_e$ be a subalgebra of $\prod_{M, s, \rho} \on{End}^{\mathrm{rat}.}_{\mathcal O(Z_M)} \mathcal V(Z_M, s,\rho,q)$ (where the product is taken over all compact $s$) given by the following conditions:
     \\
     a) any $\varphi \in \mathcal S_e$ does not  have poles  at the points of families $\mathcal V(Z_M,s,\rho,q)$ which correspond to (cf. \textit{loc. cit.}) non-strictly positive characters of Levi subgroups;
     \\
     b) the endomorphisms $\varphi$ are compatible in some precise sense (see Section \ref{BK_formulation_thm} for details).
     
     Then $\mathcal S_e \simeq J_e$.
\end{thm}

Before sketching our proof of this theorem, let us discuss certain  families of modules over $\He$, $\He_e$, $\J_e \otimes \BC[{\bf{v}},{\bf{v}}^{-1}]$ that should be considered as geometric counterparts of the Braverman-Kazhdan's families $\mathcal{V}(Z_M,s,\rho,q)$.

Here by a ``geometric''  family we mean the result of the following construction. Let $C \subset Z_{G^\vee}(s,e)$ be a torus and let $M^\vee:=Z_{G^\vee}(C)^0$ be the corresponding Levi. Let $\rho$ be an irreducible representation of $\Gamma_M:=Z_{M^\vee}(s,e)/Z_{M^\vee}(s,e)^0$. Set $H:=\langle C,s\rangle$ (the smallest closed diagonalizable subgroup of $G^\vee$ containing $s$ and $C$). 
We have two families over $C \times \BC^*$:
\begin{equation*}
\mathcal{K}(C,s):=\K_{H \times \BC^*}(\CB_e)|_{Cs \times \BC^*},\, 
\mathcal{L}(C,s):=\K_{H \times \BC^*}(\CB_e^{\BC^*})|_{Cs \times \BC^*}.
\end{equation*} 
The algebras $\He$, $\He_e$ act naturally on $\mathcal{K}(C,s)$ and the algebra $\J_e=\K_{Z_e \times \BC^*}(\CB_e^{\BC^*} \times \CB_e^{\BC^*})$ acts naturally on $\mathcal{L}(C,s)$. We also have the natural action of $\Gamma_M$ on the families $\mathcal{K}(C,s)$, $\mathcal{K}(L,s)$. Taking $\rho$-multiplicity spaces we obtain the families $\mathcal{K}(C,s,\rho):=\mathcal{K}(C,s)_{\rho}$, $\mathcal{L}(C,s,\rho):=\mathcal{L}(C,s)_{\rho}$. 
One can show that for $(\chi,q) \in C \times \BC^*$ we have $\Gamma_M$-equivariant identifications:
\begin{equation*}
\mathcal{K}(C,s)|_{(\chi,q)}=K(s\chi,e,q),\,\mathcal{L}(C,s)|_{\{\chi\} \times \BC^*}=E(s\chi,e) \otimes \BC[{\bf{v}},{\bf{v}}^{-1}],
\end{equation*}
so these families are nothing but the families of (directs sums of) standard or irreducible modules over our algebras. We prove the following theorem that should be considered as a version  ``in families'' of the  identification  $\phi_q^*E(s,e,\rho) \simeq K(s,e,\rho,q)$ above\footnote{Existence of an action of $\J$ on $\mathcal{K}(C,s)$ that is algebraic in $s$ is
claimed and used in \cite{BK}, see the proof of Theorem 2.4 $(3)$ in {\textit{loc. cit.}}, 
%
but not checked there in detail.  Our Theorem \ref{thm_ident_families_pull_back} fills in the details.\label{foot2}} 

\begin{thmx}\label{thm_ident_families_pull_back}
We have a natural $\Gamma_M$-equivariant isomorphism of the families of $\He$ and $\He_{\bf{c}}$-modules respectively:
\begin{equation*}
\phi^{{\bf{c}}*}\mathcal{L}(C,s) \simeq_{\He} \mathcal{K}(C,s),\, \phi^{*}_{{\bf{c}}}\mathcal{L}(C,s) \simeq_{\He_{\bf{c}}} \mathcal{K}(C,s).
\end{equation*}
\end{thmx}

Let us now return to the proof of Theorem \ref{BK_main_thm_intro}. We show that the families $\mathcal{K}(Z_M,s,\rho,q)$, $\mathcal{V}(Z_M,s,\rho,q)$ can be identified after restricting to some open dense subset $\mathcal{U} \subset Z_M$ containing all non-strictly positive characters. Using Theorem \ref{thm_ident_families_pull_back}, it then follows that we are reduced to proving the following proposition.

\begin{prop}\label{prop_reform_BK_thm_intro}
    Let $\mathcal E$ be the subalgebra of $\prod_{M, s} \on{End}_{{\mathcal{O}}(Z_M)} \mathcal L(Z_M, s,q)=:\tilde{E}$ consisting of elements $\phi = (\phi(M, s))_{M, s}$ satisfying the following property:
    
    for any conjugate pair $s\chi \sim t\chi'$ (here $\chi' \in Z_L$ for a Levi subgroup $L^\vee$, so that $t \in L^\vee$, and $e \in \on{Lie}L^\vee$), the following 
    equality holds:
     \begin{equation*}
     \phi(M, s)_\chi = \phi(L, t)_{\chi'}.
     \end{equation*}
     Then $\mathcal{E} \simeq J_e$ via the  action map $\alpha\colon J_e \to \mathcal E$.
\end{prop}
To prove this proposition, we show that the homomorphism $\alpha\colon \J_e \rightarrow \mathcal{E}$ becomes an isomorphism after completion at every point of $\on{Spec}\K_{Z_e}(\on{pt})$ by describing completions of $\J_e$, $\mathcal{E}$ as explicit subalgebras of the completion of $\tilde{E}$.{\footnote{Our argument is largely parallel to the original proof of \cite{BK}, the difference is that we consider completion at the points of $\on{Spec}\K_{Z_e}(\on{pt})$ which is an exact functor, so, in particular, the natural embedding $\mathcal{E} \subset \tilde{E}$ remains embedding after completions. In general it does not induce an embedding of fibers (since the fibers of $\J_e$ are not semisimple in general, see Section \ref{ex_fiber_not_semisimple}), so
 the last paragraph of the proof of \cite[Theorem 1.8]{BK} does not hold as stated. We use our geometric description of $J_e$ and localization theorem in $K$-theory to analyze the completion of $\J_e$. \label{foot3}}} \label{footnote_paral_BK}

\subsubsection{Description of the cocenter of $\He$}
In the last part of the text we use our geometric approach to $J_e$ to obtain a description of the cocenter of $\J_e$ conjectured in \cite[Section 6.2]{BDD}.
Let us formulate the theorem.

Let $\operatorname{Comm}_{Z_e}$ be the commuting variety for $Z_e$ (with the reduced scheme structure).  Set $\mathcal C_{Z_e}:=\operatorname{Comm}_{Z_e}/\!/Z_e$.

\begin{thmx}\label{intro_final_cocenter} Let $\mathcal{O}^{a}(e)$ be the space consisting of  regular functions $f$ on $\mathcal C_{Z_e}$, subject to the following properties:

a) for any semisimple $s \in Z_e$, $f|_{\{s\} \times Z_{Z_e}(s)}$ is locally constant (and, hence, gives a well-defined function $f_s$ on the component group $\Gamma_e^s$ of $Z_{Z_e}(s)$;

b) $f_s$ is a sum of characters of the group $\Gamma_e^s$ arising in $\K(\CB_{e}^s)$.

Then $\mathcal{O}^{a}(e) \simeq \J_e/[\J_e,\J_e]$.
\end{thmx}

\begin{remark}
Recall that by \cite[Theorem 1]{BDD} the homomorphism $\phi_q\colon \He_q \rightarrow \J=\bigoplus_e \J_e$ induces an isomorphism on the level of cocenters for $q$ not a root of unity.  This means that Theorem \ref{intro_final_cocenter} gives a description of the cocenter $\He_q/[\He_q,\He_q]$ for $q$ not a root of unity.
\end{remark}

\subsection{Structure of the paper}
The paper is organized as follows. In Section \ref{sec:intro} we recall definitions and known properties of the algebras $\CH$, $J$. Section \ref{proof_of_prop_1} is devoted to the proof of Theorem \ref{thm_B}. In Section \ref{sec:main1} we prove Theorem \ref{thm_A}. In Section \ref{sec:fam} we study representations  of $\J$, construct families of (irreducible) modules over it and prove Theorems \ref{thm_C}, \ref{thm_ident_families_pull_back}. In Section \ref{end}, we prove \cite[Theorem 1.8 (3)]{BK} by our methods. In Section \ref{cocenter_J_section} we prove Theorem \ref{intro_final_cocenter} i.e. describe the cocenter of $\J_e$. Appendix \ref{appK} contains proofs of various general facts about equivariant $K$-theory that we need for our arguments. 


\subsection{Acknowledgements} We  gratefully acknowledge helpful input from Dan Ciubotaru, Stefan Dawydiak, Pavel Etingof, Michael Finkelberg,  Victor Ginzburg, Mikhail Goltsblat, 
William Graham, Do Kien Hoang, David Kazhdan, Ivan Losev, 
Jakub L\" owit, Victor Ostrik, Oron Propp, Vadim Vologodsky and Zhiwei Yun. We are grateful to George Lusztig for useful comments on the first version of the text.

I. K. especially thanks Dmitrii Zakharov for explaining the material of Proposition~\ref{appendix} and Michael Finkelberg for his \TeX-nical generosity.

R.B. was partly supported by the NSF grant DMS-2101507.

\section{Generalities on $\CH$ and $J$}\label{sec:intro}
\subsection{The affine Hecke algebra}
Let $W$ be as above and let $S \subset W$ be the set of simple reflections. Let $\ell\colon W \rightarrow \mathbb{Z}_{\geqslant 0}$ be the length function on $W$. The algebra $\CH=\CH_G$ is an algebra over $\BZ[{\bf{v}},{\bf{v}}^{-1}]$ with basis $\{T_w\}_{w \in W}$. Multiplication in this algebra is determined by the relations $T_{x}T_{y}=T_{xy}$ when $\ell(xy)=\ell(x)+\ell(y)$ and $(T_s-v^2)(T_s+1)=0$ for $s \in S$. Recall that $\He:=\mathcal{H} \otimes_{\BZ} \BC$ is the complexification of $\mathcal{H}$.

\subsubsection{Lusztig's cells} There is a partition of $W$ into the union of the so-called \textit{cells}. We start with a brief recollection of it. 

Let $\leq$ be the strong Bruhat order. First of all, two elements $x$ and $y$ of $W$ are said to be \textit{connected}, if either $x < y$ or $y < x$; and  $\operatorname{deg} P_{x, y} = \frac{|\ell(x)-\ell(y)|-1}{2}$: in particular, $\ell(x) - \ell(y)$ is odd. (Here, we denote by $P_{x, y}$ the corresponding Kazhdan-Lusztig polynomial, see \cite{KL0} and \cite[Section 1.2]{Lcells3}.)

To each $w \in W$ corresponds the so-called \textit{left descending set} $D^l(w)=\{s \in S \ | \ sw \leq w\}$. 

Now, for $x, y \in W$ we say that $x \leq_{L} y$, if there is a chain of elements $(x_1=x,  \ldots, x_l, \ldots, x_k=y)$, s.t. all neighbours in it are connected, and $D^l(x_i) \setminus D^l(x_{i+1}) \neq\emptyset$. 

It is well-known that $\leq_{L}$ is a partial preorder, and the corresponding equivalence classes are called \textit{left cells}: this is the work by Lusztig (\cite{Lcells1}). 

\textit{Right cells} are subsets of the form $\Psi^{-1}$ for $\Psi$ a left cell. Their definition can also be given in a similar fashion to the one above. We should only replace $D^l(w)$ by $D^r(w) := \{s \in S \ | \ ws \leq w\}$.

Finally, there are also \textit{two-sided cells} (we will denote the set of them by $\mathcal C$). They are defined in almost the same fashion as right or left ones; now the appropriate condition for neighbors in $(x_1,  \ldots, x_l, \ldots)$ is that either $D^l(x_i) \setminus D^l(x_{i+1}) \neq \emptyset$ or $D^r(x_i) \setminus D^r(x_{i+1}) \neq \emptyset$ for each $i$.

\subsubsection{Two-sided cells and nilpotent conjugacy classes}\label{nilp} The famous theorem by Lusztig (cf. \cite{Lcells4}) says that two-sided cells in $W$ are in bijection with the conjugacy classes of nilpotent elements in the Lie algebra $\mathfrak{g}^\vee$ of the Langlands dual group $G^\vee$. In particular, this gives an order on cells (via the closure order on the nilpotent orbits). This order coincides with the natural one (this was conjectured by Lusztig and proved by the first author in \cite[Theorem 4 (b)]{B2}). 

The so-called \textit{$a$-function} sends every two-sided cell $\bf c$ to the dimension of the corresponding Springer fiber $\CB_e$.

It has a combinatorial meaning as well. Namely, recall that $\mathcal{H}$ is the affine Hecke algebra of $G$ and let us denote by $C_x, \ x \in W,$ the Kazhdan-Lusztig basis elements in $\mathcal{H}$. Then, if $h_{x, y, z} \in \mathbb Z[{\bf{v}}, {\bf{v}}^{-1}]$ are structure constants of $\mathcal{H}$ with respect to $\{C_w\}$, $-a(\bf c)$ is the lowest possible degree of non-zero term in $h_{x, y, z}$, $x, y \in \bf c$, $z \in W$.

\subsection{Asymptotic affine Hecke algebra} Now we recall Lusztig's definition of the asymptotic affine Hecke algebra $J$ (cf. \textit{loc. cit.}). It has a basis $\{t_w\,|\,  w \in W\}$. In this basis the structure constant $\gamma_{x, y, z}$ is the constant term of the polynomial ${\bf{v}}^{a(z)}h_{x, y, z^{-1}}$: $t_xt_y=\sum_{z}\gamma_{x,y,z}t_{z^{-1}}$.

It is shown in {\em loc. cit.} that $J$ is actually an associative algebra. Moreover, denoting the span over ${\BZ[{\bf{v}},{\bf{v}}^{-1}]}$ of $\{t_x \ | \ x \in \bf c\}$ by $J_{\bf c}$, one gets a decomposition of $J$ as a direct sum (product) of rings: $$J = \bigoplus_{\bf c \in \mathcal C} J_{\bf c}.$$
We will sometimes denote $J_{\bf c}$  by $J_e$ for $\bf c$ corresponding to the nilpotent orbit ${\mathbb{O}}_e=G^\vee e$ (cf. Section \ref{nilp}). The same notation will be used for subscripts and superscripts of various morphisms in Section \ref{phi}.

\subsubsection{Distinguished involutions and the unit element in $J_{\bf{c}}$}\label{s} 
The set of \textit{distinguished involutions}: $\mathcal D \subset W$ consists of all elements $d \in W$, such that $a(w) - \ell(w) = 2\operatorname{deg}P_{1, w}.$ It is known that its elements are actually involutions, and each one-sided cell contains exactly one of them.

The element ${\bf 1} = \sum_{d \in \mathcal D} t_d$ is the unit element in the algebra $J$,
while for ${\bf{c}} \in \mathcal{C}$, 
the element ${\bf 1}_{\bf{c}}:=\sum_{d \in {\bf{c}} \cap \mathcal{D}} {{t_d}}$ is the unit element in
$J_{\bf{c}}$.

\subsection{Relation between $\mathcal{H}$ and $J$}\label{phi} 

Recall 
Lusztig's homomorphism (see \cite{Lcells3}) 
$\phi\colon \mathcal H \to J \otimes \mathbb Z[{\bf{v}}, {\bf{v}}^{-1}]$, $$C_w \mapsto \sum\limits_{d \in \mathcal D, z \in W, \ a(z) = a(d)} h_{w, d, z}t_z$$  
and the {\emph{additive}} isomorphism  $\psi\colon \mathcal H \iso J\otimes \mathbb \BZ[{\bf{v}}, {\bf{v}}^{-1}], \ C_w \mapsto t_w$.

For a two-sided cell $\bf c \in \mathcal C$ we set 
$$J_{\leq \bf c}:=\bigoplus\limits_{{\bf c'}\leq_{LR} \bf c} J_{ \bf c'}.$$
Obviously $\psi(J_{\leq \bf c}[{\bf{v}}, {\bf{v}}^{-1}])=\mathcal{H}_{\leq \bf c}$
is the corresponding term of the cell filtration.

The coincidence of the geometric and the cell filtration implies 
 that $\mathcal H_{\leq \bf c}$ can be identified with the $G^\vee \times \BC^*$-equivariant $K$-theory of the preimage of $\overline{\mathbb{O}}_e$ in $\widetilde{\CN} \times_{\CN} \widetilde{\CN}$ and $\mathcal H_{ \bf c} = K_{Z_e \times \BC^*}(\mathcal{B}_e \times \mathcal{B}_e)$, see \cite{TX}, \cite{Xi1} for details.

Finally,  note that homomorphism $\phi$ induces  homomorphisms 
\begin{equation*}
\phi_{\bf c}\colon \mathcal H_{\bf c} \to J_{\bf c} \otimes \mathbb Z[{\bf{v}}, {\bf{v}}^{-1}],\, \phi^{\bf c}\colon \mathcal{H} \to J_{\bf c} \otimes \BZ[{\bf{v}},{\bf{v}}^{-1}]
\end{equation*}
($\phi^{\bf c}$ is the composition of $\phi$ and the projection to the direct summand).

The isomorphism $\psi$ restricts to the isomorphisms $\psi_{\bf c}\colon \mathcal{H}_{\bf c} \iso J_{\bf c} \otimes \BZ[{\bf{v}}, {\bf{v}}^{-1}]$.

\subsubsection{Bimodule structure on $J \otimes \mathbb Z[{\bf{v}}, {\bf{v}}^{-1}]$}\label{m} 
A homomorphism of rings $R\to S$ equips $S$ with the structure of an $R-S$ bimodule.
In particular, the homomorphism $\phi^{\bf{c}}$ defines a $\mathcal{H} - J_{\bf{c}} \otimes \mathbb Z[{\bf{v}}, {\bf{v}}^{-1}]$ - bimodule structure on $J_{\bf{c}} \otimes \mathbb{Z}[{\bf{v}},{\bf{v}}^{-1}]$. Similarly, $\phi_{\bf{c}}$ defines a $\mathcal{H}_{\bf{c}} - J_{\bf{c}} \otimes \mathbb Z[{\bf{v}}, {\bf{v}}^{-1}]$ - bimodule structure on $J_{\bf{c}} \otimes \mathbb{Z}[{\bf{v}},{\bf{v}}^{-1}]$. 

In this section, we will give an alternative description of the (left) action of $\mathcal{H}$ on $J_{\bf{c}} \otimes \mathbb{Z}[{\bf{v}},{\bf{v}}^{-1}]$ as above. We will then conclude that this $\mathcal{H}$-module is isomorphic to $\mathcal{H}_{\bf{c}}$. We will obtain similar statements for the $\mathcal{H}_{\bf{c}}$-module $J_{\bf{c}} \otimes \mathbb{Z}[{\bf{v}},{\bf{v}}^{-1}]$.

Let  $\operatorname{ht}_{\bf c}$ denote  the projection ${\mathcal H}_{\leq \bf c} \twoheadrightarrow {\mathcal H}_{\bf c}$. Following \cite[Section 3]{L1} let us define the action of $\mathcal{H}$ on $\bigoplus_{{\bf{c}}' \leqslant {\bf{c}}} J_{{\bf{c}}} \otimes \mathbb Z[{\bf{v}}, {\bf{v}}^{-1}]$ by the formula: 
\begin{equation}\label{act_H_on_J}
C_x \cdot t_y = \psi(C_xC_y).
\end{equation}
Note that directly from the definitions $\psi\colon \mathcal{H}_{\leqslant {\bf{c}}} \iso \bigoplus_{{\bf{c}}' \leqslant {\bf{c}}} J_{{\bf{c}}} \otimes \mathbb Z[{\bf{v}}, {\bf{v}}^{-1}]$ is the isomorphism of $\mathcal{H}$-modules (in particular, it follows that the formula (\ref{act_H_on_J}) indeed defines the action of $\mathcal{H}$).

Note now that $\psi(\mathcal{H}_{<{\bf{c}}}) = \bigoplus_{{\bf{c}}' < {\bf{c}}} J_{{\bf{c}}} \otimes \mathbb Z[{\bf{v}},{\bf{v}}^{-1}]$ is an $\mathcal{H}$-submodule of $\bigoplus_{{\bf{c}}' \leqslant {\bf{c}}} J_{{\bf{c}}} \otimes \mathbb Z[{\bf{v}}, {\bf{v}}^{-1}]$. Modding out by this submodule, we obtain the action of $\mathcal{H}$ on $J_{{\bf{c}}} \otimes \BZ[{\bf{v}},{\bf{v}}^{-1}]$ such that for $y\in \bf c$
we have:
\begin{equation}\label{action_H_on_J}
C_x\cdot t_y = \psi_{{\bf{c}}}(\operatorname{ht_{\bf c}}(C_xC_y)).
\end{equation} 
It induces the action $\mathcal{H}_{\bf{c}} \curvearrowright J_{\bf{c}} \otimes \mathbb{Z}[{\bf{v}},{\bf{v}}^{-1}]$.

\begin{lemma}\label{lem_alt_descr_H_module}
(a) Two $\mathcal{H}$-actions on $J_{\bf{c}} \otimes \BZ[{\bf{v}},{\bf{v}}^{-1}]$ defined above (the one coming from $\phi^{\bf{c}}$ and the one given by the  formula (\ref{action_H_on_J}))  coincide.

(b) Two $\mathcal{H}_{\bf{c}}$-actions on $J_{\bf{c}} \otimes \BZ[{\bf{v}},{\bf{v}}^{-1}]$ defined above (the one coming from $\phi_{\bf{c}}$ and the one given by the formula (\ref{action_H_on_J})) coincide. 
\end{lemma}
\begin{proof}
It is enough to prove part (a), part (b) will follow.

Let us prove part (a). We need to check that 
\begin{equation*}
C_x \cdot t_y = \phi^{{\bf{c}}}(C_x)t_y
\end{equation*}
for $x \in W$ and $y \in {\bf{c}}$. So, our goal is to check that 
\begin{equation}\label{our_goal}
\psi_{\bf{c}}(\operatorname{ht}_{\bf{c}}(C_xC_y)) = \phi^{\bf{c}}(C_x)t_y.
\end{equation}

Note that directly from the definitions we have:
\begin{equation}\label{phi_via_psi}
\phi^{\bf{c}}(C_w) = \sum_{d \in \mathcal{D} \cap {\bf{c}}} \psi_{\bf{c}}(\operatorname{ht}_{\bf{c}}(C_wC_d)).
\end{equation}
The following formula follows from \cite[2.4 (d)]{Lcells2} (see also  \cite[Proof of Proposition 2]{BDD}):
\begin{equation}\label{ssylka}\psi_{\bf c}(\operatorname{ht}_{\bf{c}}(C_{x_1}C_{x_2}))t_{x_3} = \psi_{\bf c}(C_{x_1}\psi_{\bf c}^{-1}(t_{x_2}t_{x_3})).
\end{equation}
Now, setting $x_1=x$, $x_3=y$, summing (\ref{ssylka}) over $x_2 \in \mathcal D \cap \bf c$ and using (\ref{phi_via_psi}) together with the fact that ${\bf{1}}_{{\bf{c}}} = \sum_{x_2 \in \mathcal{D} \cap {\bf{c}}}t_{x_2}$ is the identity element of $J_{\bf{c}}$ we obtain the desired equation (\ref{our_goal}).
\end{proof}

\begin{cor}\label{cor_H_mod_is_iso}
(a) The $\mathcal{H}$-module $J_{\bf{c}} \otimes \mathbb{Z}[{\bf{v}},{\bf{v}}^{-1}]$ defined via $\phi^{\bf{c}}$ is isomorphic to $\mathcal{H}_{\bf{c}}$.

(b) The $\mathcal{H}_{\bf{c}}$-module $J_{\bf{c}} \otimes \mathbb{Z}[{\bf{v}},{\bf{v}}^{-1}]$ defined via $\phi_{\bf{c}}$ is isomorphic to $\mathcal{H}_{\bf{c}}$.
\end{cor}
\begin{proof}
It follows from Lemma \ref{lem_alt_descr_H_module} that $\psi$ is the desired isomorphism of $\mathcal{H}$ (resp. $\mathcal{H}_{\bf{c}}$) - modules.
\end{proof}

 




\subsection{Description of $J_e$ via finite centrally extended sets}\label{sec:low}


A key ingredient in our proof of the Theorem~\ref{thm_B} presented in the next section is the \textit{finite-set realization} of $J_{\bf{c}}$, which we now recall.

The following was conjectured (in a slightly different form) by Lusztig \cite{Lcells4}. It was proven in \cite{Xi} in type A and
in \cite{BO} in general.  
\begin{defn} A centrally extended finite $H$-set (for a reductive group $H$) is a finite $H$-set $Y$ equipped with some central extension $1 \to \mathbb G_m \to \widetilde{\operatorname{Stab}_H(y)} \to \operatorname{Stab}_H(y) \to 1$ for every $y \in Y$ such that for every $g \in G$ the conjugation isomorphism $C_g\colon \operatorname{Stab}_H(y) \to \operatorname{Stab}_H(gy)$ extends to an isomorphism between these central extensions satisfying natural compatibilities.
\end{defn}
For a centrally extended finite $H$-set $Y$, one defines an equivariant sheaf on $Y$ as a usual $H$-equivariant sheaf $\mathcal F$ with an additional data of $\widetilde{\operatorname{Stab}_H}(y)$-action on $\mathcal F_y$, which induces an action of $\mathbb G_m$-factor by identity character. This allows us to consider $K_H(\on{pt})$-algebra $K_H(Y)$.
\begin{thm}[\cite{BO}]\label{BO} Let ${\bf {c}} \subset W$ be a two-sided cell in $W$ and let $e \in \mathcal{N}$ denote the corresponding nilpotent element. Then there exists a centrally extended finite $Z_e$-set $Y$ such that $J_{\bf c} \simeq K_{Z_e}(Y \times Y)$.
\end{thm}
In particular, one has a central embedding $K_{Z_e}(\on{pt}) \to J_{\bf c}$, which will be of crucial importance for us.

\begin{remark} For $G=PGL_n$ the situation simplifies. 
First, the centralizers in question are connected so the $Z_e$-action is trivial.

Also, it turns out that no non-trivial central extensions appear.
 Thus, we arrive at the following statement proved in \cite{Xi}: $J_{\bf c}$ can be realized as a matrix algebra over the ring $K_{Z_e}(\on{pt})=K_0(Z_e-\on{mod})$. Note that $J$ corresponding to $\operatorname{PGL}_n$ in our notations is $J$ corresponding to the extended affine Weyl group associated with $\operatorname{SL}_n$ in the notations of \cite{Xi}.



\end{remark}

\subsection{Description of $\CH_e$ and $J_e$ via bimodules over non-commutative Springer}

We now recall, following \cite{BL}, the relation between the above algebras and the non-commutative Springer resolution $A$ introduced in \cite{BICM}, \cite{BM}.
 
Recall that $A$ is an algebra equipped with a derived equivalence 
\begin{equation*}
D^b(A-\on{mod}_{G^\vee \times \mathbb{C}^*})\simeq D^b(\on{Coh}_{G \times \mathbb{C}^*}(\widetilde{\mathcal N})),
\end{equation*}
where the subscript denotes the equivariant structure. 
It has a canonical filtration indexed by the poset of nilpotent conjugacy classes in $\mathcal{N}$ 
equipped with the closure order.

We will denote the associated graded piece $A_{\leqslant e}/A_{<e}$  by $A_e$. It is known (cf. \cite{BM}) that $K_0(A_e-\operatorname{bimod}_{Z_e \times \mathbb C^*})$ 
is isomorphic to $\mathcal H_e$. Now, it follows from \cite[Section 8]{BL} that there exists an isomorphism of $K_{Z_e \times \mathbb C^*}(\on{pt})=R_e$-algebras:

\begin{equation}\label{identification_J_bimod}
J_e \otimes \mathbb Z[{\bf{v}}, {\bf{v}}^{-1}] \simeq K_0(A_e-\operatorname{bimod}^{\operatorname{ss}}_{Z_e \times \mathbb C^*}),\end{equation}

where in the right hand side we take the $K$-group of the category of semisimple bimodules. Here the morphism $\psi$ can be identified with the tautological
isomorphism between the $K$-theories of the categories of semisimple and all finite length modules. 
The $R_e$-action on $J_e  \otimes \mathbb Z[{\bf{v}}, {\bf{v}}^{-1}]$ comes from the identification (\ref{identification_J_bimod}).

For future use we record  the following. 

\begin{cor}\label{center}
The morphisms $\phi_e$ and $\psi_e$ are $R_e$-linear.
\end{cor}

\begin{proof}
For $\psi_e$ this is immediate from the above. To deduce the claim for $\phi_e$, recall that for any $x \in W$, $\phi(C_x) = \sum_{d \in {\bf{c}} \cap \mathcal{D}} \psi(C_xC_d)$, the sum running over distinguished involutions inside the cell associated to $e$. Now the claim follows, since the $R_e$-action commutes with both $\psi$ and the multiplication in $\mathcal H_e$.
\end{proof}

\begin{remark} See \cite[Proposition 6]{Lcells3} for a related statement. \end{remark}

\begin{remark} Note that 
$\mathcal H_e$ is not a unital algebra, in particular, the $R_e$-action on $\CH_e$ does not come from a homomorphism $R_e 
\to \mathcal H_e$. 
\end{remark}

\section{Proof of Theorem \ref{thm_B}}\label{proof_of_prop_1}
\subsection{} Now we are ready to prove Theorem~\ref{thm_B}. Recall that it claims that we have isomorphisms of algebras:
\begin{equation*}
J_{{\bf{c}}} \otimes \mathbb{Z}[{\bf{v}}, {\bf{v}}^{-1}] \iso \operatorname{End}_{\mathcal{H}}(J_{\mathbf{c}} \otimes \mathbb{Z}[{\bf{v}}, {\bf{v}}^{-1}])^{\mathrm{opp}} \simeq \operatorname{End}_{\mathcal{H}}(\CH_{\bf{c}})^{\mathrm{opp}} \simeq \operatorname{End}^{R_e}_{\mathcal{H}_{\bf{c}}}(\CH_{\bf{c}})^{\mathrm{opp}}.
\end{equation*}
We have already observed in  Corollary \ref{cor_H_mod_is_iso} that $J_{{\bf{c}}} \otimes \BZ[{\bf{v}},{\bf{v}}^{-1}] \simeq \CH_{\bf{c}}$ as left $\CH$ (and $\CH_{\bf{c}}$)-modules so it remains to check that the natural homomorphisms 
\begin{equation}\label{J_action_homom}
J_{{\bf{c}}} \otimes \mathbb Z[{\bf{v}}, {\bf{v}}^{-1}] \rightarrow \operatorname{End}_{\mathcal{H}}(J_{\mathbf{c}} \otimes \mathbb Z[{\bf{v}}, {\bf{v}}^{-1}])^{\mathrm{opp}},\, J_{\bf{c}} \otimes \mathbb Z[{\bf{v}}, {\bf{v}}^{-1}] \rightarrow \operatorname{End}_{\mathcal{H}_{\bf{c}}}^{R_e}(J_{\bf{c}} \otimes \mathbb Z[{\bf{v}},{\bf{v}}^{-1}])^{\mathrm{opp}}
\end{equation} are actually isomorphisms. Note that both of them are injective since $J_{\bf{c}}$ is unital. So, it remains to check that the homomorpisms in (\ref{J_action_homom}) are surjective. We start from the first one, the proof for the second one is completely analogous. 

First of all, we claim that it is enough to check that the homomorphism 
\begin{equation}\label{J_action_homom_C}
\J_{{\bf{c}}} \otimes \mathbb C[{\bf{v}}, {\bf{v}}^{-1}] \rightarrow \operatorname{End}_{\He}(\J_{\mathbf{c}} \otimes \mathbb C[{\bf{v}}, {\bf{v}}^{-1}])^{\mathrm{opp}}
\end{equation}
is surjective. Indeed, since $J_{\bf{c}}$ is a free $\mathbb{Z}$-module, we have embeddings 
\begin{equation*}
J_{\bf{c}} \hookrightarrow \J_{\bf{c}},~ \operatorname{End}_{\mathcal{H}}(J_{\bf{c}} \otimes \BZ[{\bf{v}},{\bf{v}}^{-1}])^{\mathrm{opp}} \hookrightarrow \operatorname{End}_{\He}(\J_{\bf{c}} \otimes \BZ[{\bf{v}},{\bf{v}}^{-1}])^{\mathrm{opp}}.
\end{equation*}
Now the homomorphism (\ref{J_action_homom_C}) is given by $\J_{\bf{c}} \otimes \BC[{\bf{v}},{\bf{v}}^{-1}] \ni a \mapsto - \cdot a$, so if the operator $- \cdot a$ lies in $\operatorname{End}_{\mathcal{H}}(J_{\bf{c}} \otimes \BZ[{\bf{v}},{\bf{v}}^{-1}])^{\mathrm{opp}} \hookrightarrow \operatorname{End}_{\He}(\J_{\bf{c}} \otimes \BZ[{\bf{v}},{\bf{v}}^{-1}])^{\mathrm{opp}}$ then $a = 1 \cdot a \in J_{\bf{c}} \otimes \BZ[{\bf{v}},{\bf{v}}^{-1}]$.

So, we need to prove the surjectivity of (\ref{J_action_homom_C}). The proof consists of several steps. 
\\
\\
{\bf{A)}} Recall that $\J_{\bf{c}}$ is an algebra with a unit element (see Section \ref{s}). To prove the proposition it is sufficient to show vanishing of an $\He$-endomorphism of $\J_{\bf{c}} \otimes \mathbb C[{\bf{v}}, {\bf{v}}^{-1}]$ which sends ${\bf 1}_{\bf{c}}$ to $0$. This, in turn, can be reformulated as follows: 

\begin{equation}\label{goal_eq_zero}
V_{\bf c}: = \operatorname{Hom}_{\He}(\J_{\bf c} \otimes \mathbb C[{\bf v}, {\bf v}^{-1}] /\operatorname{Im}(\phi_{\bf c}), \J_{\bf c} \otimes \mathbb C[{\bf v}, {\bf v}^{-1}]) = 0,
\end{equation}
where abusing notations we denote the complexification of $\phi_{\bf{c}}$ by the same letter.


{\bf{B)}} To fix ideas, we first prove the equality (\ref{goal_eq_zero}) in the special case when $G=\on{GL}_n$. In this case centralizers of all nilpotent elements are connected, so $\R_e=\K_{Z_e \times \BC^*}(\on{pt})$ has no zero-divisors. Let us denote by $\R_e^{'}$ the image of the restriction homomorphism $\R:=\K_{G^\vee \times \mathbb C^*}(\on{pt}) \to \K_{Z_e \times \BC^*}(\on{pt})=\R_e$.

Now, since $\phi_{\bf c}$ is injective and $\R_e$-linear (cf. \cite{BDD}, and Corollary \ref{center}), $\phi_{\bf c}(\He_{\bf c})$ has the same $\R_e-$ (and, hence, $\R_e^{'}-$) rank as $\He_{\bf{c}}$, thus it is an
$\R_e^{'}$-submodule of full rank in $\J_{\bf{c}}$. (These ranks are finite by \cite[Proposition 6]{Lcells3}.)

Moreover, the bijective map $\psi_e$ is $\R_e$-linear as well (see Corollary \ref{center}), so $\operatorname{rk}_{\R_e^{'}} \operatorname{Im}(\phi_{\bf c}) = \operatorname{rk}_{\R_e^{'}} \J_{\bf c} \otimes \mathbb C[{\bf{v}}, {\bf{v}}^{-1}]$.

It means that $\J_{\bf c} \otimes \mathbb C[{\bf{v}}, {\bf{v}}^{-1}]/\operatorname{Im}(\phi_{\bf c})$ is $\R_e^{'}$-torsion. 

 Since $\R_e^{'}$ has no zero divisors (centralizers being connected) and $\J_{\bf c}\otimes \mathbb C[{\bf{v}}, {\bf{v}}^{-1}]$ is a torsion-free module over $\R_e$ (cf. Lemma \ref{tor} below) the statement follows
 once we check that an $\He$-linear morphism is $\R_e^{'}$-linear.
 The latter point follows from the well-known isomorphism between $\R$ and the center of $\He$ \cite{Berncen} and compatibility between the $\R$- and $\R_e$-actions on $\He_e$ (cf. \cite{Xi}).
\\
\\
{\bf{C)}} We now consider the general case, then $Z_e =: Z$ may have several connected components which we denote by $Z_j$, $j\in \Gamma:=\pi_0(Z)$.
 Each conjugacy class $\Omega_i$ in $\Gamma$ gives rise to an idempotent $a_i$ in $\R_e$, so that $\R_e = \bigoplus_i a_i\R_e$, and $a_i\R_e$ (and, hence, $a_i\R_e^{'}$) has no zero-divisors. The desired statement follows as above from the following two lemmas.
\begin{lemma}\label{lemma_1} \label{tor} $a_i \J_{\bf c} \otimes \mathbb C[{\bf{v}}, {\bf{v}}^{-1}]$ is a torsion-free module over $a_i \R_e$. \end{lemma}
\begin{lemma}\label{lemma_2} Any $\He_{\bf{c}}$-linear (in particular, any $\He$-linear) endomorphism of $\J_{\bf c} \otimes \mathbb C[{\bf{v}}, {\bf{v}}^{-1}]$  preserves the decomposition $\J_{\bf c} \otimes \mathbb C[{\bf{v}}, {\bf{v}}^{-1}] = \bigoplus_i a_i\J_{\bf c} \otimes \mathbb C[{\bf{v}}, {\bf{v}}^{-1}].$ \end{lemma}

\noindent {\bf D)} \textit{Proof of Lemma \ref{lemma_1}.} In Section \ref{sec:low} we recalled the realization of $J_{\bf c}$ as $Z_e$-equivariant $K$-theory of 
a centrally extended finite set. We claim that the statement of Lemma \ref{lemma_1} is true for any $\R_e$-module arising that way.

It suffices to consider the case when the action of $Z_e$ on the finite set is transitive.
If the corresponding central extension is trivial then the complexified equivariant $K$-theory 
is identified with the ring of conjugation invariant regular functions on a finite index
subgroup $H\subset Z_e\times \mathbb C^*$; in general it is a subspace in the ring 
of invariant functions on a finite covering $\tilde{H}$ of such a subgroup $H$.
The statement is clear since a component in $\tilde{H}$ maps surjectively to a
component in $Z_e\times \mathbb C^*$.

{\bf{E)}} \textit{Proof of Lemma \ref{lemma_2}.} As follows from Corollary \ref{center}, we have decomposition $$\He_{\bf c} = \bigoplus_i a_i\He_{\bf c}.$$

Now consider the $\He_{\bf{c}}$-action on $\J_{\bf c} \otimes \mathbb C[{\bf{v}}, {\bf{v}}^{-1}]$ via $\phi_{\bf{c}}$ (cf. Section \ref{m}), since $a_i$'s are orthogonal we get: $$
a_i\mathcal \He_{\bf c} \cdot (a_j \J_{\bf c} \otimes \mathbb C[{\bf{v}}, {\bf{v}}^{-1}]) = 0
$$
for $i \neq j$. 
Thus  $a_i\J_{\bf c}$ consists of all elements in $\J_{\bf c}$ annihilated by $a_j\He_{\bf c}$ for all $j\ne i$, so the decomposition of $\J_{\bf c}$ from Lemma \ref{lemma_2} is stable under $\He_{\bf c}$-linear endomorphisms.
\qed

This finishes the proof of Theorem \ref{thm_B}.




\section{Proof of Theorem \ref{thm_A}} \label{sec:main1}
\subsection{A homomorphism $\theta\colon K_e \rightarrow J_e \otimes \BC[{\bf{v}},{\bf{v}}^{-1}]$}\label{int}
\subsubsection{} Set 
\begin{equation*}
K_e := K_{Z_e \times \mathbb{C}^*}(\mathcal{B}_e^{\mathbb{C}^*} \times \mathcal{B}_e^{\mathbb{C}^*}),~\K_e := \K_{Z_e \times \mathbb{C}^*}(\CB_e^{\mathbb{C}^*} \times \CB_e^{\mathbb{C}^*}).
\end{equation*}
Clearly, $K_e$, is a (unital) algebra w.r.t. the natural convolution product.  Our goal in this section is to provide an injective homomorphism $\theta\colon K_e \to J_e \otimes \mathbb Z[{\bf{v}}, {\bf{v}}^{-1}]$; it will be shown in Section \ref{sect_surj_theta}  that $\theta$ is an isomorphism. 

We will construct $\theta$ by introducing an $\mathcal{H} - K_e$-bimodule $\mathcal F_e$, checking that $\mathcal F_e \simeq_{\mathcal H} \mathcal H_e$, and invoking Theorem~\ref{thm_B}.



Let us recall first (cf. \cite{CG}, \cite{KL}) that $\mathcal H$ can be identified with $G^\vee \times \mathbb C^*$-equivariant $K$-theory $K_{G^\vee \times \mathbb C^*}(\on{St})$ of the Steinberg variety $\on{St}:=\widetilde{\CN} \times_{\CN} \widetilde{\CN}$. We have the projection  $\pi\colon \on{St} \to \mathcal N$; for a locally closed subvariety $Z \subset \mathcal{N}$ we will denote $\pi^{-1}(Z)$ by $\on{St}_{Z} \subset \on{St}$.

From   \cite{Xi1} and 
\cite[Theorem 55, \S 11.3]{B3} it follows that 
$\mathcal H_e$ is identified  with $K_{G^{\vee}\times \mathbb C^*}(\on{St}_{\mathbb{O}_e})= K_{Z_e \times \mathbb C^*}(\CB_e \times \CB_e)$; the isomorphism is compatible with the action of $K_{Z_e \times \mathbb C^*}(\on{pt})$.
Here, the  $\mathbb C^*$-action comes from the element $h$ of an $\mathfrak{sl}_2$ triple $(e, h, f)$; we fix
such a triple. It is well-known that the centralizer $Z_{G^\vee}(e, h, f)$ is a maximal reductive
subgroup in $ Z_{G^\vee}(e)$, thus it is identified with  $Z_e$.

Let $p\colon \widetilde{\mathcal N} \to \mathcal N$ be the Springer resolution and $\Sigma_e:=e+\mathfrak{z}_{\mathfrak{g}^\vee}(f) \subseteq \mathfrak g^{\vee}$ the
Slodowy slice, we also let $\Lambda_e \subset \widetilde{\CN}$ denote $p^{-1}(\Sigma_e \cap \mathcal N)$. This is a smooth variety called Slodowy variety. We have $\CB_e \subseteq \Lambda_e$; moreover, $\Lambda_e^{\mathbb C^*} = \CB_e^{\mathbb C^*}$ (see e.g. \cite[Section 1.8]{L3}).

Now we define 
\begin{equation*}
\mathcal F_e := K_{Z_e \times \mathbb C^*}(\CB_e \times \CB_e^{\mathbb C^*}) = K_{Z_e \times \mathbb C^*}(\Lambda_e \times \CB_e^{\mathbb C^*}; \CB_e \times \CB_e^{\mathbb C^*}).
\end{equation*}

\begin{lemma}\label{torsion_free_F_K}
Both $K_e$ and $\CF_e$ are torsion-free as $\mathbb{Z}$-modules.
\end{lemma}
\begin{proof}
It follows from Proposition \ref{split} below that there exists an identification $\mathcal{F}_e \simeq \mathcal{H}_e$ of left $\mathcal{H}$-modules, so, in particular, they are isomorphic as $\mathbb{Z}$-modules. Let ${\bf{c}} \subset W$ be the two-sided cell corresponding to $e$. Recall that $\mathcal{H}_{e} = \mathcal{H}_{\leqslant \bf{c}}/\mathcal{H}_{< \bf{c}}$  which is free over $\mathbb{Z}$ with a basis consisting of $\{[C_w]{\bf{v}}^k\,|\, w \in {\bf{c}},\, k \in \mathbb{Z}\}$. We conclude that $\mathcal{F}_e$ is also free over $\mathbb{Z}$. Proposition \ref{sym} claims that there exists an isomorphism $\CF_e \simeq K_e$, implying that $K_e$ is also free over $\mathbb{Z}$.
\end{proof}

\subsubsection{The $\mathcal{H}-K_e$-bimodule structure on $\mathcal{F}_e$}\label{sect_bimodule_geom_structure} 

\begin{prop}\label{prop_h_vs_he}
    There are natural commuting actions of $\mathcal H$ and $K_e$ on $\mathcal F_e$. 


\end{prop}

\begin{proof}
Proposition essentially follows from \cite{L2}. We sketch the proof for the reader's convenience.

Let  us define the action $*$ of $\mathcal{H}$ on $\mathcal{F}_e$. Consider the  diagram. 

\[\begin{tikzcd}
	& {\widetilde{\mathcal N} \times  \widetilde{\mathcal N} \times \CB_e^{\mathbb C^*} =: X} \\
	{\widetilde{\mathcal N} \times \widetilde{\mathcal N}} && {\widetilde{\mathcal N} \times \CB_e^{\mathbb C^*}}
	\arrow["{p_{12}}", from=1-2, to=2-1]
	\arrow["{p_{23}}"', from=1-2, to=2-3],
\end{tikzcd}\]
where $p_{12}$, $p_{23}$ are the projections onto the corresponding factors.

Pick 
$h \in \mathcal H\simeq K_{G^\vee \times \mathbb C^*}(\widetilde{\CN} \times_{\CN} \widetilde{\CN})$
and $a \in K_{Z_e \times \mathbb C^*}(\CB_e \times \CB_e^{\mathbb C^*})$.
Now
\\
1) $a$  can be viewed as a class of an equivariant complex $\mathcal A$ on $\widetilde{\mathcal N} \times \CB_e^{\mathbb C^*}$ with support on $\CB_e \times \CB_e^{\mathbb C^*}$;
\\
2) $h$ can be viewed as a class of an equivariant complex $\mathcal G$ on $\widetilde{\mathcal N} \times \widetilde{\mathcal N}$ with support  on $\widetilde{\mathcal N} \times_{\mathcal N} \widetilde{\mathcal N}$.

We define $\mathcal G * \mathcal A= p_{13*}(p_{12}^*(\mathcal G) \otimes^{\mathbb L}_X p_{23}^*(\mathcal A))$, where $p_{13}\colon X \rightarrow \widetilde{\CN} \times \CB_e^{\BC^*}$ is the projection onto the first and the third factors. Then we set $h * a :=[\mathcal G * \mathcal A]$ (cf. ).

 The fact that this defines an action of $\mathcal H$ commuting with the action of $K_e$ (defined below) follows by diagram 
 chase.


Let us define the action of $K_e = K_{Z_e \times \mathbb{C}^*}(\CB_e^{\mathbb{C}^*} \times \CB_e^{\mathbb{C}^*})$ on 
\begin{equation*}
\mathcal{F}_e = K_{Z_e \times \mathbb{C}^*}(\Lambda_e \times \CB_e^{\mathbb C^*};\CB_e \times \CB_e^{\mathbb{C}^*})=K_{Z_e \times \mathbb{C}^*}(\widetilde{\mathcal{N}} \times  \CB_e^{\mathbb C^*};\CB_e \times \CB_e^{\mathbb{C}^*}).
\end{equation*}
Consider the following diagram: 
\[\begin{tikzcd}[scale cd=0.85]
 & {\Lambda_e \times \CB_e^{\mathbb C^*} \times \CB_e^{\mathbb C^*}} 
 \\
 {\Lambda_e \times \CB_e^{\mathbb C^*}} & {\Lambda_e \times \CB_e^{\mathbb C^*}} & {\CB_e^{\mathbb C^*} \times \CB_e^{\mathbb C^*}}
	\arrow["{q_{13}}"', from=1-2, to=2-1]
	\arrow["{q_{12}}"', from=1-2, to=2-3]
	\arrow["{q_{23}}", from=1-2, to=2-2]
 \end{tikzcd}\]
it is clear that for $Q \in \mathcal F_e$
and
$R \in K_e$ 
the formula:
\begin{equation*}
Q * R = q_{13*}(q_{12}^*Q \otimes q_{23}^*R)
\end{equation*}
gives 
a well-defined  right $K_e$-action
on $\mathcal F_e$. 

\end{proof}

\begin{lemma}
The induced action of $\mathcal{H}_{\leqslant e} \subset \mathcal{H}$ on $\mathcal{F}_e$ factors through the action of $\mathcal{H}_e$.
\end{lemma}
\begin{proof}
Let us first of all recall that we have the identification $\mathcal{H}_e \simeq K_{Z_e \times \mathbb{C}^*}(\mathcal{B}_e \times \mathcal{B}_e)$. After this identification, the realization of $\mathcal{H}_e$ as a subquotient $\mathcal{H}_{\leqslant e}/\mathcal{H}_{<e}$ of $\mathcal{H} \simeq K_{G^\vee \times \mathbb{C}^*}(\on{St})$ can be described as follows. We have  $\mathcal{H}_{\leqslant e}=K_{G^\vee \times \mathbb{C}^*}(\on{St}_{\overline{\mathbb{O}}_e})$, then the embedding $K_{G^\vee \times \mathbb{C}^*}(\on{St}_{\overline{\mathbb{O}}_e}) \hookrightarrow K_{G^\vee \times \mathbb{C}^*}(\on{St})$ is given by the pushforward for the closed embedding $\on{St}_{\overline{\mathbb{O}}_e} \hookrightarrow \on{St}$ and the quotient $K_{G^\vee \times \mathbb{C}^*}(\on{St}_{\overline{\mathbb{O}}_e}) \twoheadrightarrow K_{Z_e \times \mathbb{C}^*}(\mathcal{B}_e \times \mathcal{B}_e) = K_{G^\vee \times \mathbb{C}^*}(\on{St}_{\mathbb{O}_e})$ is given by the restriction to the open subset $\on{St}_{\mathbb{O}_e} \hookrightarrow \on{St}_{\overline{\mathbb{O}}_e}$.

So, our goal is to check that $\mathcal{H}_{<e} = K_{G^\vee \times \mathbb{C}^*}(\on{St}_{\overline{\mathbb{O}}_e \setminus \mathbb{O}_e})$ acts trivially on $\CF_e = K_{Z_e \times \mathbb{C}^*}(\mathcal{B}_e \times \mathcal{B}_e^{\mathbb{C}^*})$. Indeed, let $\mathcal{G}$ be an equivariant complex representing some class in $K_{G^\vee \times \mathbb{C}^*}(\on{St}_{\overline{\mathbb{O}}_e \setminus \mathbb{O}_e})$, then the support of $\mathcal{G}$ is contained in $\on{St}_{\overline{\mathbb{O}}_e \setminus \mathbb{O}_e}$. Let $\mathcal{A}$ be an equivariant complex on $\widetilde{\mathcal{N}} \times \mathcal{B}_e^{\mathbb{C}^*}$ representing a class in $K_{Z_e \times \mathbb{C}^*}(\mathcal{B}_e \times \mathcal{B}_e^{\mathbb{C}^*})$, then the support of $\mathcal{A}$ is contained in $\mathcal{B}_e \times \mathcal{B}_e^{\mathbb{C}^*}$. We conclude that the support of $p_{12}^*(\mathcal{G}) \otimes^{\mathbb{L}}_X p_{23}^*(\mathcal{A})$ (see the notations from the proof of Proposition \ref{prop_h_vs_he}) is contained in $p_{23}^{-1}(\mathcal{B}_e \times \mathcal{B}_e^{\mathbb{C}^*}) \cap p_{12}^{-1}(\on{St}_{\overline{\mathbb{O}}_e \setminus \mathbb{O}_e}) = \varnothing$, hence the (derived) tensor product above is equal to zero.
\end{proof}


\subsubsection{Alternative description of the action of $\mathcal{H}_e$ on  $\mathcal{F}_e$}
We can define the convolution algebra structure on $K_{Z_e \times \mathbb C^*}(\CB_e \times \CB_e)$ by identifying it with $K_{Z_e \times \mathbb C^*}(\Lambda_e \times \Lambda_e; \CB_e \times \CB_e)$, the $K$-group of $Z_e \times \mathbb C^*$-equivariant coherent sheaves on $\Lambda_e \times \Lambda_e$ with set-theoretic support on $\CB_e \times \CB_e$ (cf. \cite{TX}). This algebra is 
isomorphic to $\mathcal H_e$ (see \cite[Theorem B.2]{TX}). Let us describe the action $K_{Z_e \times \mathbb C^*}(\CB_e \times \CB_e) \curvearrowright \mathcal{F}_e = K_{Z_e \times \mathbb{C}^*}(\mathcal{B}_e \times \mathcal{B}_e^{\mathbb{C}^*})$ geometrically. Consider the following diagram 
\[\begin{tikzcd}[scale cd=0.85]
	& {\Lambda_e \times \Lambda_e \times \CB_e^{\mathbb C^*}} 
 \\
	{\Lambda_e \times \Lambda_e} & {\Lambda_e \times \CB_e^{\mathbb C^*}} & {\Lambda_e \times \CB_e^{\mathbb C^*}} 
	\arrow["{\pi_{12}}"', from=1-2, to=2-1]
	\arrow["{\pi_{23}}", from=1-2, to=2-3]
	\arrow["{\pi_{13}}"', from=1-2, to=2-2]
 \end{tikzcd}\]
Since $\CB_e$ is proper, $\Lambda_e$ is smooth, and $\CB_e^{\mathbb C^*}$ is smooth and proper, it is clear that for 
\begin{equation*}
[\mathcal{P}] \in K_{Z_e \times \mathbb{C}^*}(\mathcal{B}_e \times \mathcal{B}_e),~[\mathcal{Q}] \in \mathcal{F}_e = K_{Z_e \times \mathbb{C}^*}(\mathcal{B}_e \times \mathcal{B}_e^{\mathbb{C}^*})
\end{equation*}
the formula:
\begin{equation}\label{act_H_e_geom}
[\mathcal{P}]*[\mathcal{Q}] = [\pi_{13*}(\pi_{12}^*\mathcal{P} \otimes \pi_{23}^*\mathcal{Q})]
\end{equation}
give 
a well-defined  left $K_{Z_e \times \mathbb{C}^*}(\mathcal{B}_e \times \mathcal{B}_e)$-action
on $\mathcal F_e$. 

\begin{prop}
After the identification $\mathcal{H}_e \simeq K_{Z_e \times \mathbb{C^*}}(\mathcal{B}_e \times \mathcal{B}_e)$, the action given by (\ref{act_H_e_geom}) coincides with the action of $\mathcal{H}_e$ induced by the action of $\mathcal{H}$ on $\mathcal{F}_e$ described in Proposition \ref{prop_h_vs_he}.
\end{prop}
\begin{proof}
Same argument as in the proof of \cite[Theorem B.2]{TX} works.
\end{proof}

\subsubsection{}
In view of Theorem~\ref{thm_B}, the next  key proposition yields the desired homomorphism 
\begin{equation}\label{obv}\theta\colon  K_e \to \operatorname{End}_{\CH}(\mathcal F_e)^{\mathrm{opp}} = \operatorname{End}_{\CH}(\mathcal H_e)^{\mathrm{opp}} =J_e \otimes \mathbb C [{\bf{v}}, {\bf{v}}^{-1}].\end{equation}

 \begin{prop}\label{theorem} \label{split} 
 We have a canonical isomorphism
 of left $\mathcal H$-modules:
 $\mathcal F_e \simeq_{\mathcal H} \mathcal H_e$.
\end{prop}

\begin{remark} 
Notice that $\CF_e$ splits as a direct sum indexed by components of $\CB_e^{\mathbb C^*}$, thus
Proposition \ref{split} implies such a decomposition for $\mathcal H_e\simeq K^{Z_e\times \mathbb C^*}(\CB_e \times \CB_e)$.

Similar direct sum decompositions are found in the literature.
In particular, \cite[Theorem 2.10]{halpern} provides a semi-orthogonal decomposition for the derived category
of equivariant coherent sheaves in a rather general situation, which implies a direct sum decomposition
for the Grothendieck group. Existence of such a decomposition compatible with the convolution action would
imply Proposition~\ref{split}.\footnote{The similarity between Theorem \ref{split} and results of \cite{halpern}
was pointed out to us by Do Kien Hoang.}



On the other hand,  \cite[Lemma 14.9]{L3} essentially proves that for  $G^\vee=SL_n$ the Bialynicky-Birula filtration on $K_{T \times \mathbb C^*}(\CB_e)$ can be split by means of the action of standard generators of Hecke algebra on $K_{T \times \mathbb C^*}(\CB_{e, 1})$. 

We do not know if either of the two approaches yields a proof of Proposition~\ref{split}.
\end{remark}

\begin{remark}
The proof of Proposition \ref{split} below  can be generalized
to yield a similar direct sum decomposition for an arbitrary projective variety $X$ with a $\BC^*$-action such that for every connected component $F \subset X^{\BC^*}$ the attractor of this component is smooth (in particular, the component itself is smooth) and the analog of Lemma \ref{diff} below holds.

\end{remark}

Before proceeding to the proof of the Proposition we  recall some geometric properties of the variety $\CB_e$
found in \cite {DLP}. Recall that we fix an $\mathfrak{sl}_2$-triple $e,h,f \in \mathfrak{g}^\vee$. The adjoint action of $h$ on $\mathfrak{g}^\vee$ induces the decomposition $\mathfrak{g}^\vee=\bigoplus_{i \in \BZ} \mathfrak{g}^\vee_i$. Let $L^\vee$ (resp. $P$) be the connected algebraic subgroup of $G^\vee$ whose Lie algebra is $\mathfrak{g}^\vee_0$ (resp. $\bigoplus_{i \geqslant 0}\mathfrak{g}^\vee_i$). Recall that $W_f$ is the Weyl group of $G^\vee$, let $W_L \subset W_f$ be the Weyl group of $L^\vee$.
Recall that $\mathcal{B}^{\mathbb{C}^*}=\bigsqcup_{\bar{w} \in W_L \backslash W_f}L^{\vee}wB/B$.  For $\bar{w} \in W_L \backslash W_f$ we set $\CB_{\bar{w}}:=PwB/B \subset \mathcal{B}$. We also set $\CB_{e,\bar{w}}:=\CB_{\bar{w}} \cap \CB_e$. Note that $\CB_{e,\bar{w}}$ consists of points $y \in \CB_e$ such that $\underset{t \rightarrow 0}{\on{lim}}\, t \cdot y \in L^\vee \bar{w}B/B$, where the action of $\mathbb{C}^*$ comes from the cocharacter of the center of $L^\vee$ that integrates $h$.



\begin{lemma}\label{lemma_springer_w_part_smooth}
Variety $\CB_{e,\bar{w}}$ is smooth. 
\end{lemma}
\begin{proof}
\cite[Proposition 3.2]{DLP}.
\end{proof}

\begin{warning}
Variety $\CB_{e,\bar{w}}$ may be disconnected. It may also be empty.  
\end{warning}

\begin{lemma}\label{lemma_limit_morph_springer_affine_fibr}
The map $x \mapsto \underset{t \rightarrow 0}{\on{lim}}\,t \cdot x$ 
is an affine fibration $\CB_{e,\bar{w}} \twoheadrightarrow \CB_{e,\bar{w}}^{\BC^*}$.
\end{lemma}
\begin{proof}
Follows from \cite[Theorem 4.1]{BB} (the assumptions of this theorem are satisfied for $X=\CB_{e,\bar{w}}$ by Lemma \ref{lemma_springer_w_part_smooth} together with \cite[Corollary 2 in Section 3]{su}). 

Another argument (that works only over fields of characteristic $0$) can be found in \cite[Section 3.4]{DLP}, where the authors use the result of Bass-Haboush (see \cite[Section 1.5]{DLP}) to obtain the statement. Finally, another argument (that works over fields of arbitrary characteristic) is given in \cite[Section 5]{lu_unipotent}.  We are grateful to George Lusztig for pointing out this reference to us. 
\end{proof}

\begin{defn}
A partition of a variety $X$ as a finite union of locally closed subvarieties $X_i$ is said to be an $\alpha$-partition if the subvarieties in the partition can be indexed $X_1, \ldots, X_n$ in such a way that $X_1 \cup \ldots \cup X_k$ is closed in $X$ for $k=1,\ldots,n$.
\end{defn}

\begin{lemma}\label{lem_al_part_b_e}
Varieties $\CB_{e,\bar{w}}$ form an $\alpha$-partition of $\CB_e$.
\end{lemma}
\begin{proof}
See \cite[Section 3.4]{DLP}.
\end{proof}

\begin{proof} (of Proposition \ref{split}) \textbf{Step 1.} 
By Lemma \ref{lem_al_part_b_e} varieties $\CB_{e,\bar{w}}$ form an $\alpha$-partition of $\CB_e$, i.e.  there exists a labeling $\CB_{e,\bar{w}_i}$, $i=1,\ldots,n$ of these varieties such that $\bigcup_{i=1}^k\CB_{e,\bar{w}_i}$ is closed for $k=1,\ldots,n$.


This yields a filtration (that we denote by $\Phi$) on $K_{Z_e \times \mathbb C^*}(\CB_e \times \CB_e)$.

Moreover, the locally closed strata $\CB_{e,\bar{w}}$ are the attracting varieties for $\CB_{e,\bar{w}}^{\mathbb C^*}$, which
is a union of components of $\CB_{e}^{\mathbb C^*}$.
 By Lemma \ref{lemma_limit_morph_springer_affine_fibr} the map $\pi_{\bar{w}}\colon \CB_{e, \bar{w}} \to \CB_{e, \bar{w}}^{\mathbb C^*}$ is a vector bundle, hence it induces (cf. \textit{loc.~cit.}) an isomorphism $\pi_{\bar{w}}^*\colon K_{Z_e \times \BC^*}(\CB_{e,\bar{w}}^{\BC^*}) \iso K_{Z_e \times \BC^*}(\CB_{e,\bar{w}})$. 
 
 We claim that one gets a canonical isomorphism of $K_{Z_e \times \mathbb C^*}(\on{pt})$-modules: $\operatorname{gr_{\Phi}}K_{Z_e \times \mathbb C^*}(\CB_e \times \CB_e) = \mathcal F_e$. 

To check this, one has to show that all of the maps $\Phi_{k}:=K_{Z_e \times \BC^*}(\CB_e \times (\bigcup_{i=1}^{k}\CB_{e,\bar{w}_i})) \to K_{Z_e \times \BC^*}(\CB_e \times \CB_e)$ are injective. 

For this, see the Lemma~\ref{diff} below.

In order to finish the proof of the Proposition it now suffices to provide an $\mathcal H_e$- and $R_e$-equivariant splitting of $\Phi$. 

 \textbf{Step 2.} For every $\bar{w}$, $\pi_{\bar{w}}$ together with the locally closed embedding $\iota_{\bar{w}}\colon \mathcal{B}_{e,\bar{w}} \hookrightarrow \CB_e$ fits into the diagram:
$$
\begin{tikzcd}
	{\CB_{e, \bar{w}}}  	\arrow{d}{\pi_{\bar{w}}} \arrow{r}{\iota_{\bar{w}}} & {\CB_e} \\
	{\CB_{e, \bar{w}}^{\mathbb C^*}.}
\end{tikzcd}
$$

We claim that $\operatorname{Im}(\iota_{\bar{w}} \times \pi_{\bar{w}})$ is locally closed inside $\CB_e \times \CB_{e, w}^{\mathbb C^*}$, or, equivalently, inside $\overline{\CB}_{e, \bar{w}} \times \CB_{e, \bar{w}}^{\mathbb C^*}$.
To see this, consider  the diagram:

\[\begin{tikzcd}
	{\CB_{e, \bar{w}}} & {\CB_{e, \bar{w}} \times \CB_{e, \bar{w}}^{\mathbb C^*}} & {\overline{\CB}_{e,\bar{w}} \times \CB_{e, \bar{w}}^{\mathbb C^*}},
	\arrow[from=1-1, to=1-2]
	\arrow[from=1-2, to=1-3]
\end{tikzcd}\]

\noindent in which the first map is the graph map for $\pi_{\bar{w}}$ which is clearly closed, and the second one is the open embedding. The claim follows.

Now let $\Gamma_{\bar{w}}$ be the closure of the image of $\iota_{\bar{w}} \times \pi_{\bar{w}}$ inside $\CB_e \times \CB_{e, \bar{w}}^{\mathbb C^*}$. It is a proper $Z_e \times \mathbb C^*$-subvariety with an open subset isomorphic to $\CB_{e,\bar{w}}$. Moreover, there are canonical $Z_e \times \mathbb C^*$-equivariant projections  $\widetilde{\iota}_{\bar{w}}$ and $\widetilde{\pi}_{\bar{w}}$ from $\Gamma_{\bar{w}}$ to  $\CB_e$ and $\CB_{e,\bar{w}}^{\mathbb C^*}$ respectively.

\textbf{Step 3}. Now  induction by $k=1,\ldots,n$ 
shows that the map of the form $$\bigoplus_{i=1}^{n} (\operatorname{Id}_{\CB_e} \times \widetilde{\iota}_{\bar{w}_i})_{*}(\operatorname{Id}_{\CB_e} \times \widetilde{\pi}_{\bar{w}_i})^*$$ splits the filtration. (We recall that $\operatorname{Id}_{\CB_e} \times \pi_{\bar{w}}$ is a vector bundle, hence it induces an isomorphism in equivariant $K$-theory.)

Note that $(\operatorname{Id}_{\CB_e} \times \widetilde{\pi}_{\bar{w}})^*$ is well-defined: pullback in $K$-theory is well-defined for morphisms with smooth target or a base change of such morphisms, while $\operatorname{Id}_{\CB_e} \times \widetilde{\pi}_{\bar{w}}$ is a base change of the morphism $\widetilde{\pi}_{\bar{w}}$ with smooth target $\CB_{e, \bar{w}}^{\mathbb C^*}$.


\textbf{Step 4}. The argument similar to the one in the last paragraph shows that $\mathcal{H}$ has a canonical left convolution action on $K_{Z_e \times \mathbb C^*}(\CB_e \times \Gamma_{\bar{w}})$. Moreover, $(\operatorname{Id}_{\CB_e} \times \widetilde{\iota}_{\bar{w}})_{*}$ and $(\operatorname{Id}_{\CB_e} \times \widetilde{\pi}_{\bar{w}})^{*}$ are $\mathcal H$-equivariant. 

Let us, for example, prove this for $(\operatorname{Id}_{\CB_e} \times \widetilde{\pi}_{\bar{w}})^{*}$. 

We have:
\begin{equation}\label{diagr_action}
\begin{tikzcd}
	{} & {\widetilde{\mathcal{N}} \times \widetilde{\mathcal{N}} \times \Gamma_w} \arrow{rr}{{\operatorname{Id} \times \operatorname{Id} \times \widetilde{\pi}_{\bar{w}}}} && {\widetilde{\mathcal{N}} \times \widetilde{\mathcal{N}} \times \CB_e^{\mathbb C^*}} \arrow{d}{{\pi_{13}}} \\
	&  {\widetilde{\mathcal{N}}\times \Gamma_{\bar{w}}} \arrow{rr}{{\operatorname{Id} \times \widetilde{\pi}_{\bar{w}}}} && {\widetilde{\mathcal{N}} \times \CB_e^{\mathbb C^*}}
	\arrow["{q_{13}}"', from=1-2, to=2-2]
\end{tikzcd}
\end{equation}

We need to prove that for $[\mathcal A] \in K_{G^\vee \times \mathbb C^*}(\widetilde{\mathcal{N}} \times_{\mathcal{N}} \widetilde{\mathcal{N}})$ and $[\mathcal G] \in K_{Z_e \times \mathbb C^*}(\widetilde{\mathcal{N}} \times \CB_e^{\mathbb{C}^*}; \CB_e \times \CB_e^{\mathbb{C}^*})$ the following equality holds:

\begin{equation}\label{comm_H_act}
\Big[(\operatorname{Id} \times \widetilde{\pi}_{\bar{w}})^*\pi_{13*}(\pi_{12}^*\mathcal A \otimes^{\mathbb L} \pi_{23}^*\mathcal G)\Big] = \Big[q_{13*}(q_{12}^*\mathcal A \otimes^{\mathbb L} q_{23}^*(\operatorname{Id} \times \widetilde{\pi}_{\bar{w}})^*\mathcal G)\Big],
\end{equation}
where $\pi_{ij}$, (resp. $q_{ij}$), $\{i,j\} \subset \{1,2,3\}$ are projections of $\widetilde{\mathcal{N}} \times \widetilde{\mathcal{N}} \times \Gamma_w$ (resp. $\widetilde{\mathcal{N}} \times \widetilde{\mathcal{N}} \times \mathcal{B}_e^{\mathbb{C}^*}$) to the corresponding factors.

The equality (\ref{comm_H_act}) is clear since the diagram (\ref{diagr_action}) is Cartesian, the map $\pi_{13}$ is flat, and the result of \cite[Proposition 1.4]{To} holds.

This finishes the proof.
\end{proof}

\begin{lemma}\label{diff} The natural maps $K_{Z_e \times \BC^*}(\CB_e \times (\bigcup_{i=1}^{k}\CB_{e,\bar{w}_i})) \to K_{Z_e \times \BC^*}(\CB_e \times \CB_e)$ are injective for all $k$.
\end{lemma}

\begin{proof} 

\textbf{Step 1.} To prove the sought-for injectivity, it suffices to construct the map 
\begin{equation*}
\kappa\colon K_{Z_e \times \mathbb C^*}(\mathcal B_e \times \mathcal B_e) \to 
K_{Z_e \times \mathbb C^*}(\mathcal B_e \times \mathcal B_{e}^{\mathbb{C}^*})
\end{equation*} 
so that the maps $\kappa \circ (\operatorname{Id}_{\CB_e} \times \widetilde{\iota}_{\bar{w}_i})_{*}(\operatorname{Id}_{\CB_e} \times \widetilde{\pi}_{\bar{w}_i})^*$ composed with the projections $K_{Z_e \times \mathbb{C}^*}(\mathcal{B}_e \times \mathcal B_{e}^{\mathbb{C}^*}) \rightarrow K_{Z_e \times \mathbb{C}^*}(\mathcal{B}_e \times \mathcal B_{e,\bar{w}_i}^{\mathbb{C}^*})$ are injective for all $i$.

We will make use of  \textit{the Drinfeld-Gaitsgory degeneration} (cf.~\cite[Section 2]{DG}), and \textit{the specialization in the equivariant $K$-theory} (cf.~\cite[Section 5.3]{CG}).

First of all, we briefly recall  both constructions.

\textbf{Step 2.} First, for an algebraic variety $M$ equipped with a $\mathbb C^*$-action, the \textit{DG-degeneration} is the family $\widetilde{M} \to \mathbb C$ so that $\widetilde{M}_t$ can be canonically identified with $M$, and $\widetilde{M}_0 \simeq \bigsqcup_{\bar{w}_i} M_{\bar{w}_i}^+ \times_{M_{\bar{w}_i}^{\mathbb C^*}} M_{\bar{w}_i}^-$. 
Here, for the given  $M_{\bar{w}_i}^{\mathbb C^*}$, $M_{\bar{w}_i}^+$ stands for the corresponding attracting set, and $M_{\bar{w}_i}^-$ stands for the repelling one. Moreover, there exists a canonical global trivialization of $\widetilde{M}$ over $\mathbb G_m$: $\pi_M\colon\widetilde{M}|_{\mathbb G_m} \simeq M \times \mathbb G_m$.

By inspection of the constructions from \textit{loc. cit.}, one sees that for $Z$ equipped with the action of an algebraic group $H$ (so that the $H$-action commutes with $\mathbb C^*$), the family also carries an action of $H$ compatible with the trivial action on the base.

Second, if $X \to \mathbb C$ is any $H$-equivariant algebraic family (for example, the one above), the map $\operatorname{lim}_0 \colon K_H(X \setminus X_0) \to K_H(X_0)$ can be defined.

\textbf{Step 3.} Now, let us consider $\mathcal B_e \times \mathcal B_e$ with the $\mathbb C^*$-action on the second factor as $M$. 

The desired map $\kappa$ is as follows.

For any class $[\mathcal F] \in K_{Z_e \times \mathbb C^*}(M)$, one may consider the class
\begin{equation*}
\pi_{+*}\operatorname{lim}_0\pi_M^*([\mathcal F \boxtimes \mathcal{O}_{\mathbb{G}_m}]) \in K_{Z_e \times \mathbb C^*}(M^+) \simeq K_{Z_e \times \mathbb C^*}(M^{\mathbb C^*})
\end{equation*}
for $\pi_+$ being the projection $M^+ \times_{M^{\mathbb C^*}} M^- \to M^+$.

Let us now prove the fact that the composition $\kappa \circ (\operatorname{Id}_{\CB_e} \times \widetilde{\iota}_{\bar{w}_i})_{*}(\operatorname{Id}_{\CB_e} \times \widetilde{\pi}_{\bar{w}_i})^*[\mathcal G]$ restricted to $M^+_{\bar{w}_i}$ is non-zero for a non-zero $[\mathcal G]$.

Let $S \subset M$ be the support of $(\operatorname{Id}_{\CB_e} \times \widetilde{\iota}_{\bar{w}_i})_{*}(\operatorname{Id}_{\CB_e} \times \widetilde{\pi}_{\bar{w}_i})^*[\mathcal G]$. Clearly, $S$ is contained in $\on{Im}(\mathcal{B}_e \times \widetilde{\iota}_{\bar{w}_i})$ so it lies in the closure $\overline{M^+_{\bar{w}_i}}$ to the component of $M^{\mathbb{C}^*}$ corresponding to $\bar{w}_i$. Set $Z:=\overline{M^+_{\bar{w}_i}}$ and let $\widetilde{Z}$ be the Drinfeld-Gaitsgory interpolation of $Z$.
Note that $\widetilde{Z}$ is closed in $\widetilde{M}$ by  \cite[Proposition 2.3.2 (i)]{DG}. Note also that $M^+_{\bar{w}_i} \subset Z^+ \times_{Z^{\mathbb{C}^*}} Z^-=Z_0$ is a union of connected components of $Z_0$.


Let $j_0$ be an evident open (and closed) embedding $M^+_{\bar{w}_i} \hookrightarrow Z^+ \times_{Z^{\mathbb C^*}} Z^-=Z_0$, and let $j_t$ be an evident open embedding $M^+_{\bar{w}_i} \to Z_t$. They glue to an open embedding $j\colon M^+_{\bar{w}_i} \times \mathbb{C} \rightarrow \widetilde{Z}$. Let $j_{\neq 0}$ be the restriction of $j$ to $M^+_{\bar{w}_i} \times \mathbb{G}_m$.

Then,

\begin{align}j_0^*\operatorname{lim}_0\pi_M^*((\operatorname{Id}_{\CB_e} \times \widetilde{\iota}_{\bar{w}_i})_{*}(\operatorname{Id}_{\CB_e} \times \widetilde{\pi}_{\bar{w}_i})^*[\mathcal G] \boxtimes \mathcal{O}_{\mathbb{G}_m})= \\  \operatorname{lim}_0 j_{\neq 0}^*\pi^* ((\operatorname{Id}_{\CB_e} \times \widetilde{\iota}_{\bar{w}_i})_{*}(\operatorname{Id}_{\CB_e} \times \widetilde{\pi}_{\bar{w}_i})^*[\mathcal G] \boxtimes \mathcal{O}_{\mathbb{G}_m})= \\ [\mathcal G] \neq 0\end{align}

Here in the second step we use that the embedding $j_{\neq 0}$ degenerates, via the Drinfeld-Gaitsgory family, to the embedding $j_0$
and take the limit for the corresponding trivial family $M^+ \times \mathbb C$, -- and \cite[Lemma 5.3.6]{CG} (cf. also \cite[Theorem~5.3.9]{CG}). We have also omitted  the Thom isomorphism $\pi_{\bar{w}_i}^*$ from the notation.

The claim follows.

 

\end{proof}

The proof of the following fact is similar to the above discussion.

\begin{prop}\label{sym} As a right $K_e$-module, $\mathcal F_e$ is isomorphic to $K_e$. 
\end{prop}

Since $\CB_e^{\mathbb C^*}$ is smooth, $K_e$ is a unital algebra,  where the unit element is the class $[\mathcal{O}_{\Delta_{\CB_e^{\BC^*}}}]$ of the structure sheaf of the diagonal $\Delta_{\CB_e^{\BC^*}} \subset \CB_e^{\BC^*} \times \CB_e^{\BC^*}$. 

\begin{cor}\label{cor:main} 
The map $\theta$ is injective. 
\end{cor}
\begin{proof}
Let $\Xi\colon  K_e \iso \CF_e$ be the identification of right $K_e$-modules (see Proposition \ref{sym} above). It follows from the definitions that for $x \in K_e$, we have $x=\Xi^{-1}(\theta(x)(\Xi(1)))$, i.e., $x$ is uniquely determined by $\theta(x)$. The claim follows.
\end{proof}

\subsection{Surjectivity of $\theta$}\label{sect_surj_theta} To finish the proof of 
Theorem~\ref{thm_A} it remains to show that $\theta$ is surjective. Recall that 
$\Xi\colon K_e \iso \mathcal F_e$
 is the isomorphism of Proposition~\ref{sym}, it is compatible with the $R_e$-action and the right $K_e$-action.

Consider the identifications $K_e \iso \CF_e \simeq \CH_e \simeq J_e \otimes \mathbb{Z}[{\bf{v}},{\bf{v}}^{-1}]$ given,
respectively, by  $\Xi$, by Proposition \ref{theorem} and by Corollary \ref{cor_H_mod_is_iso}. Let $a \in J_e \otimes \mathbb{Z}[{\bf{v}},{\bf{v}}^{-1}]$ be the image of $1 \in K_e$.

\begin{lemma}
The element $a \in J_e \otimes \BC[{\bf{v}},{\bf{v}}^{-1}]$ is  left invertible.\label{lem_a_right_inv}
\end{lemma}
\begin{proof}
It follows from Proposition \ref{sym} that the element $\Xi(1)$ generates $\mathcal F_e$ under the right action of $K_e$,  hence, also under the right action of $J_e \otimes \mathbb{Z}[{\bf{v}},{\bf{v}}^{-1}]=\on{End}_{\CH}(\CF_e)^{\mathrm{opp}}$; here we use that the $K_e$-action on $\CF_e \simeq_{\mathcal{H}} \mathcal{H}_e$ comes from the homomorphism $K_e \rightarrow \on{End}_{\CH}(\CF_e)^{\mathrm{opp}}$. Generators of a free rank one right module are exactly right invertible
elements. We conclude that there exists $b \in J_e \otimes \mathbb{Z}[{\bf{v}},{\bf{v}}^{-1}]$ such that $ab=1$.
Recall that $\J_e$ is the complexification of $J_e$.
Since $\J_e$ is left-Noetherian (being a finite module over a Noetherian central subalgebra, see e.g. \cite[Proposition 1.6]{Lcells3}), it is  a Dedekind-finite ring, so the right invertible element $a$ is also left invertible. In particular, it follows that the element $b \in J_e \otimes \mathbb{Z}[{\bf{v}},{\bf{v}}^{-1}]$ is the left inverse to $a$.
\end{proof}

\begin{cor}\label{cor_uniq_det}
An element $\varphi \in \on{End}_{\CH}(\CF_e)$ is uniquely determined by its value on $\Xi(1)$.
\end{cor}
\begin{proof}
By Theorem \ref{thm_B} such an endomorphism $\varphi$ is given by 
 right multiplication by an element of $J_e \otimes \mathbb{Z}[{\bf{v}},{\bf{v}}^{-1}]$, since $a$ is left
  invertible, the claim follows.
\end{proof}

Since $\Xi(1)$ is a free generator of $\CF_e$ as a right $K_e$-module,  
 Corollary \ref{cor_uniq_det} implies surjectivity of $\theta$. This completes the proof of Theorem \ref{thm_A}, 
 establishing the isomorphism
\begin{equation}
\label{fin1}K_e \simeq J_e \otimes \mathbb Z[{\bf{v}}, {\bf{v}}^{-1}].
\end{equation}

\begin{remark}\label{bimod}
As pointed out above, by Theorem \ref{thm_A} the monoidal category $D^b(\on{Coh}_{Z_e}(\CB_e \times \CB_e))$ can be viewed
as a categorification of the algebra $J_e$. It would be interesting to find a compatible categorification
of the homomorphism 
$\phi_{{\bf{c}}}\colon \CH_{\bf{c}} \rightarrow J_{\bf{c}} \otimes \BC[{\bf{v}},{\bf{v}}^{-1}]$.
We do not have a proposal for a {\em monoidal functor} categorifying $\phi_{\bf c}$; however,
our proof provides a categorification for the {\em  bimodule} structure on the target ring arising from that homomorphism: the above argument shows this bimodule is identified with $\CF_e$, the bimodule
category $D^b(\on{Coh}_{Z_e\times \BC^*}(\CB_e\times \CB_e^{\BC^*}))$ is its categorification.
\end{remark}

\begin{cor}
We have $\operatorname{End}_{\mathcal H_e}^{R_e}(\mathcal F_e)^{\mathrm{opp}} = J_e \otimes \mathbb C[{\bf{v}}, {\bf{v}}^{-1}].$
\end{cor}
\begin{proof}
Follows from Theorem \ref{thm_B} together with Proposition \ref{split}.
\end{proof}

\subsection{Towards the geometric description of $\phi^{\bf{c}}$ and $\phi_{\bf{c}}$}

\subsubsection{} It follows from the definitions that the homomorphisms
\begin{equation*}
\phi^{\bf{c}}\colon K_{G^\vee \times \mathbb{C}^*}(\widetilde{\mathcal{N}} \times_{\mathcal{N}} \widetilde{\mathcal{N}})  \rightarrow K_{Z_e \times \mathbb{C}^*}(\mathcal{B}_e^{\mathbb{C}^*} \times \mathcal{B}_e^{\mathbb{C}^*}),~\phi_{\bf{c}}\colon K_{Z_e \times \mathbb{C}^*}(\mathcal{B}_e \times \mathcal{B}_e) \rightarrow K_{Z_e \times \mathbb{C}^*}(\mathcal{B}_e^{\mathbb{C}^*} \times \mathcal{B}_e^{\mathbb{C}^*})
\end{equation*}
can be described as follows. 
Let 
\begin{equation*}
\Xi\colon K_e = K_{Z_e \times \mathbb{C}^*}(\mathcal{B}_e^{\mathbb{C}^*} \times \mathcal{B}_e^{\mathbb{C}^*}) \iso K_{Z_e \times \mathbb{C}^*}(\mathcal{B}_e \times \mathcal{B}_e^{\mathbb{C}^*}) = \CF_e
\end{equation*}
be the identification given by $\bigoplus_{i=1}^{n}(\widetilde{\iota}_{\bar{w}_i}\times \operatorname{Id}_{\mathcal{B}_e^{\mathbb{C}^*}})_*(\widetilde{\pi}_{\bar{w}_i}\times \operatorname{Id}_{\mathcal{B}_e^{\mathbb{C}^*}})^*$ (see the notation in the proof of Step $3$ of Proposition \ref{split}). 

\begin{prop}
The homomorphism $\phi^{\bf{c}}$ is given by  
\begin{equation}\label{descr_homom_phi_geom}
K_{G^\vee \times \mathbb{C}^*}(\widetilde{\mathcal{N}} \times_{\mathcal{N}} \widetilde{\mathcal{N}}) \ni x \mapsto \Xi^{-1}(x * \Xi([\Delta_{\CB_e^{\mathbb{C}^*}}])) \in K_{Z_e \times \mathbb{C}^*}(\mathcal{B}_e^{\mathbb{C}^*} \times \mathcal{B}_e^{\mathbb{C}^*}),
\end{equation}
where  $*$ is the convolution action of $K_{G^\vee \times \mathbb{C}^*}(\widetilde{\mathcal{N}} \times_{\mathcal{N}} \widetilde{\mathcal{N}})$ on $K_{Z_e \times \mathbb{C}^*}(\mathcal{B}_e \times \mathcal{B}_e^{\mathbb{C}^*})$.

The homomorphism $\phi_{\bf{c}}$ has the same description with $*$ being replaced by the convolution action of $K_{Z_e \times \mathbb{C}^*}(\mathcal{B}_e \times \mathcal{B}_e)$ on $K_{Z_e \times \mathbb{C}^*}(\mathcal{B}_e \times \mathcal{B}_e^{\mathbb{C}^*})$.
\end{prop}
\begin{proof}
 The homomorphism $\phi^{\mathbf{c}}$ is given by the bimodule $\CF_e = K_{Z_e \times \mathbb{C}^*}(\CB_e \times \CB_e^{\mathbb{C}^*})$. Namely, for $x \in K_{G^\vee \times \mathbb{C}^*}(\widetilde{\mathcal{N}} \times_{\mathcal{N}} \widetilde{\mathcal{N}})$, its image in $K_e=K_{Z_e \times \mathbb{C}^*}(\CB_e^{\mathbb{C}^*} \times \CB_e^{\mathbb{C}^*})$ under $\phi^{\bf{c}}$ is obtained as follows: we consider the operator $x * - \in \on{End}_{K_e}(\CF_e)^{\mathrm{opp}}$ and use the isomorphism $\Xi\colon K_e \iso \CF_e$ of right $K_e$-modules to identify $\on{End}_{K_e}(\CF_e)^{\mathrm{opp}} = \on{End}_{K_e}(K_e)^{\mathrm{opp}} = K_e$. Then, $\phi^{\bf{c}}(x)$ is the element of $K_e$ corresponding to $x * -$. In other words, $\phi^{\bf{c}}(x)$ is the element of $K_e$ such that 
\begin{equation*}
x * \Xi(y) = \Xi(y * \phi^{\bf{c}}(x))
\end{equation*}
for every $y \in K_e$. Substituting $y=[\Delta_{\CB_e^{\mathbb{C}^*}}]$, we conclude that $\phi^{\bf{c}}(x) = \Xi^{-1}(x * \Xi([\Delta_{\CB_e^{\mathbb{C}^*}}]))$.
\end{proof}

\begin{remark}
Note that $\Xi([\Delta_{\CB_e^{\mathbb{C}^*}}])$ can be explicitly described, it is equal to the structure sheaf of the disjoint union $\bigsqcup_{i}\Gamma_{\bar{w}_i}$.     
\end{remark}

\subsubsection{} To every character $\lambda$ of the maximal torus of $G^\vee$ we can associate the corresponding induced line bundle $\mathcal{O}_{\mathcal{B}}(\lambda)$. Let $\mathcal{O}_{\widetilde{\mathcal{N}}}(\lambda)$ be the pull back of $\mathcal{O}_{\mathcal{B}}(\lambda)$ to $\widetilde{\mathcal{N}}=T^*\mathcal{B}$. Let $\Delta\colon \widetilde{\mathcal{N}} \hookrightarrow \widetilde{\mathcal{N}} \times_{\mathcal{N}} \widetilde{\mathcal{N}}$ be the diagonal embedding. The elements $[\Delta_*\mathcal{O}_{\mathcal{B}}(\lambda)] \in K_{G^\vee \times \mathbb{C}^*}(\widetilde{\mathcal{N}} \times_{\mathcal{N}} \widetilde{\mathcal{N}}) \simeq \mathcal{H}$ form the so-called ``lattice part'' of the affine Hecke algebra $\mathcal{H}$ (c.f. \cite[Section 7]{CG}).

\begin{remark}\label{rem_descr_action}
Note that the action of $[\Delta_*\mathcal{O}_{\mathcal{B}}(\lambda)]$ on $K_{Z_e \times \mathbb{C}^*}(\CB_e \times \CB_e^{\mathbb{C}^*})$ is given by $(\mathcal{O}_{\mathcal{B}_e}(\lambda) \boxtimes \CO_{\CB_e^{\mathbb{C}^*}}) \otimes -$.
\end{remark}

We will denote by $\CO_{\CB_e^{\mathbb{C}^*}}(\la)$ the restriction of $\CO_{\mathcal{B}}(\la)$ to $\CB_e^{\mathbb{C}^*}$. Abusing the notation, we will denote by $\Delta_{*}\CO_{\CB_e^{\mathbb{C}^*}}(\la)$ the push forward of $\CO_{\CB_e^{\mathbb{C}^*}}(\la)$ under the diagonal embedding $\Delta\colon \CB_{e}^{\mathbb{C}^*} \hookrightarrow \CB_{e}^{\mathbb{C}^*} \times \CB_{e}^{\mathbb{C}^*}$.

\begin{warning}\label{warning_problem}
It is {\emph{not}} true in general that for $y \in K_{G^\vee \times \mathbb{C}^*}(\widetilde{\CN} \times_{\CN} \widetilde{\CN})$, the element $\phi^{\bf{c}}([\Delta_*\mathcal{O}_{\widetilde{\mathcal{N}}}(\lambda)] * y)$ coincides with $\phi^{\bf{c}}(y) * \Delta_*\mathcal{O}_{\mathcal{B}_e^{\mathbb{C}^*}}(\lambda)$. 
\end{warning}

Informally,  the reason that leads to the Warning can be illustrated by the following example. Consider $\mathbb{P}^1$ with the action of $\mathbb{C}^*$ on it given by $[a:b] \mapsto [ta:b]$. Then, the set $(\mathbb{P}^1)^{\mathbb{C}^*}$ consits of two points $0 = [0:1]$, $\infty = [1:0]$. The identification $\Xi\colon K(\{0,\infty\}) \iso K(\mathbb{P}^1)$ is induced by the correspondence $\Gamma := \{0\} \times \mathbb{P}^1 \sqcup \infty \times \infty$. It follows from the definition that this identification is given by 
\begin{equation*}
[\mathbb{C}_0] \mapsto [\CO_{\mathbb{P}^1}],~[\mathbb{C}_{\infty}] \mapsto [\mathbb{C}_\infty],
\end{equation*}
where by $\mathbb{C}_p$ we denote the skyscraper sheaf at the point point $p$. 

We see that 
\begin{equation*}
[\CO_{\mathbb{P}^1}] = \Xi([\mathbb{C}_0]) = \Xi([\CO_{\mathbb{P}^1}(1)|_{\{0\}}]) \neq [\CO_{\mathbb{P}^1}(1)] \otimes \Xi([\mathbb{C}_0]) = [\CO_{\mathbb{P}^1}(1)],  
\end{equation*}
i.e., $\Xi$ does not commute with tensoring by the line bundle $\CO_{\mathbb{P}^1}(1)$.

let us now work out the details of the similar phenomenon in the real situtation, i.e., for $\mathfrak{g}=\mathfrak{sl}_3$. 
Let $e$ be the subregular nilpotent. Then, the Springer fiber $\CB_e$ has two irreducible components labeled by simple roots $\alpha_1$, $\alpha_2$. Each of these components is isomorphic to $\mathbb{P}^1$ so we will denote the $i$'th component by $\mathbb{P}^1_i$.  Components $\mathbb{P}^1_1$, $\mathbb{P}^1_2$ intersect transversally at one point to be denoted $p$. The set $\CB_e^{\mathbb{C}^*}$ consists of three points, one of them is $p$ and two other $q_1, q_2$ are such that $q_i \in \mathbb{P}^1_i$.

It is easy to see that the correspondence $\Gamma := \bigsqcup_{i}\Gamma_{\bar{w}_i}$ is equal to 
\begin{equation}\label{our_corresp_subreg_case}
\Gamma = {\mathbb{P}}^1_1 \times \{q_1\} \sqcup {\mathbb{P}}^1_2 \times \{q_2\} \sqcup \{p\} \times \{p\} \subset (\mathbb{P}^1_{1} \cup \mathbb{P}^1_{2}) \times \{p,q_1,q_2\}
\end{equation}

Let $\widetilde{\pi}$ be the projection of the correspondence (\ref{our_corresp_subreg_case}) onto $\CB_e^{\mathbb{C}^*} = \{p,q_1.q_2\}$ and let $\widetilde{\iota}$ be the projection of (\ref{our_corresp_subreg_case}) onto $\CB_e = \mathbb{P}^1_1 \cup \mathbb{P}^1_2$.

Let $\lambda$ be any dominant weight, let $k_1, k_2 \in \mathbb{Z}$ be  such that $\CO_{\CB}(\lambda)|_{\mathbb{P}^1_i} = \CO_{\mathbb{P}^1_i}(k_i)$.  
Take $y=1$ (where $y$ is as in Warning \ref{warning_problem}.)

Hence, it follows from (\ref{descr_homom_phi_geom}) that the homomorphism $\phi^{\bf{c}}$ sends $[\Delta_* \CO_{\widetilde{\mathcal{N}}(\la)}]$ to (for the first equality see Remark \ref{rem_descr_action} above)
\begin{multline*}
\Xi^{-1}([\Delta_* \CO_{\widetilde{\mathcal{N}}(\la)}] * [\CO_{\bigsqcup_{i}\Gamma_{\bar{w}_i}}]) = \Xi^{-1}([(\CO_{\CB_e}(\la) \boxtimes \CO_{\CB_e^{\mathbb{C}^*}}) \otimes \CO_{\bigsqcup_{i}\Gamma_{\bar{w}_i}}])  = \\
=\Xi^{-1}([\CO_{\mathbb{P}^1_1 \times \{q_1\}}(k_1) \sqcup \CO_{\mathbb{P}^1_2 \times \{q_2\}}(k_2) \sqcup \CO_{\{p\} \times \{p\}}]).  
\end{multline*}

So, our goal is to compare 
\begin{equation*}
\Xi(\Delta_*[\CO_{\mathcal{B}_e^{\mathbb{C}^*}}(\la)]) = \Xi(\Delta_{*}[\CO_{\{q_1, q_2, p\}}])~\text{with}~[\CO_{\mathbb{P}^1_1 \times \{q_1\}}(k_1) \oplus \CO_{\mathbb{P}^1_2 \times \{q_2\}}(k_2) \oplus \CO_{\{p\} \times \{p\}}].
\end{equation*}
Recall that $\Xi = (\widetilde{\iota} \times \on{Id}_{\CB_e^{\mathbb{C}^*}})_*(\widetilde{\pi} \times \on{Id}_{\CB_e^{\mathbb{C}^*}})^*$, it follows that $\Xi(\Delta_{*}[\CO_{\{q_1, q_2, p\}}])$ is equal to 
\begin{equation*}
[\CO_{\Gamma}] = [\CO_{\mathbb{P}^1_1 \times \{q_1\} \sqcup \mathbb{P}^1_2 \times \{q_2\} \sqcup \{p \times p\}}].
\end{equation*}

Clearly, 
\begin{equation*}
[\CO_{\mathbb{P}^1_1 \times \{q_1\} \sqcup \mathbb{P}^1_2 \times \{q_2\} \sqcup \{p \times p\}}] \neq  [\CO_{\mathbb{P}^1_1 \times \{q_1\}}(k_1) \oplus \CO_{\mathbb{P}^1_2 \times \{q_2\}}(k_2) \oplus \CO_{\{p\} \times \{p\}}]
\end{equation*}
when $\la \neq 0$ (i.e., $k_1 \neq 0$ or $k_2 \neq 0$.)

\subsection{Case $e=0$}

\subsubsection{} Let us consider the case $e=0$ in more detail. We then have $\CB_e=\CB$, the action of $\mathbb{C}^*$ on $\CB$ is trivial. It follows that the splitting $\Xi$  is also trivial and the correspondence $\bigsqcup_i \Gamma_{\bar{w}_i}$ is equal to $\Delta_{\CB}$ (the diagonal $\Delta_\CB \hookrightarrow \CB \times \CB$). Let ${\bf{c}}_0$ be the two-sided cell corresponding to $e=0$.
Theorem \ref{thm_A} in this case establishes the isomorphism: 
\begin{equation*}
K_{G^\vee \times \mathbb{C}^*}(\CB \times \CB) \simeq J_{{\bf{c}}_0} \otimes \mathbb{Z}[{\bf{v}},{\bf{v}}^{-1}]
\end{equation*}
such that the homomorphism $\phi^{{\bf{c}}_0}\colon K_{G^\vee \times \mathbb{C}^*}(\widetilde{\mathcal{N}} \times_{\mathcal{N}} \widetilde{\mathcal{N}}) \rightarrow K_{G^\vee \times \mathbb{C}^*}(\CB \times \CB)$ is given by 
\begin{equation*}
x \mapsto x * [\Delta_{\CB}],
\end{equation*}
where $*$ corresponds to the standard action of $K_{G^\vee \times \mathbb{C}^*}(\widetilde{\mathcal{N}} \times_{\mathcal{N}} \widetilde{\mathcal{N}})$ on  $K_{G^\vee \times \mathbb{C}^*}(\CB \times \CB)$ via convolution.

\begin{remark}
The map $x \mapsto x * [\Delta_{\CB}]$ can  be explicitly described as follows. Let us consider $\mathfrak N := \left(\widetilde{\mathcal N} \times_{\mathcal N} \widetilde{\mathcal N}\right) \times \mathcal N$ together with its three projections $p_i$ onto each of the factors. Let also $\pi$ be the natural map $\mathfrak N \to \widetilde{\mathcal N} \times_{\mathcal N} \widetilde{\mathcal N}$

Unwrapping the definitions, one gets $x * [\Delta_{\CB}] = (p_{1*} \times p_{3*})(\pi^*(x) \otimes^{\mathbb L} (\pi_2 \times \pi_3)^*\mathcal O_{\Delta})$: here we have identified $K_{G^\vee \times \mathbb C^*}(T^*\mathcal B \times T^*\mathcal B)$ and $K_{G^\vee \times \mathbb C^*}(\mathcal B \times \mathcal B)$ by means of the Thom isomorphism $\tau$.

Let us also consider the natural map $\delta: \widetilde{\mathcal N} \times_{\mathcal N} \widetilde{\mathcal N} \to \mathfrak N, \ (a, b) \mapsto (a, b, b).$ By the projection formula, $x * [\Delta_{\CB}] = (p_{1*} \times p_{3*})\delta_*(x) = j_*(x)$ for the natural map $j: \widetilde{\mathcal N} \times_{\mathcal N} \widetilde{\mathcal N} \to \widetilde{\mathcal N} \times \widetilde{\mathcal N}$. 
\end{remark}

The identification $K_{G^\vee}(\CB \times \CB) \simeq J_{{\bf{c}}_0}$ was already obtained in \cite{Xi0} (when the derived subgroup of $G^\vee$ is simply-connected) and in \cite{Ni} (in general).  Moreover, in \cite{Xi2} the homomorphism $\phi^{{\bf{c}}_0}\colon K_{G^\vee}(\CB \times \CB) \rightarrow K_{G^\vee}(\CB \times \CB)  \simeq J_{{\bf{c}}_0}$ was described for $G$ such that the derived subgroup of $G^\vee$ is simply-connected. For such $G^\vee$, $K_{G^\vee}(\CB \times \CB)$ is isomorphic to the matrix algebra $K_{G^\vee}(\CB) \otimes_{K_{G^\vee}(\on{pt})} K_{G^\vee}(\CB)$ (see \cite[Proposition 1.6]{KL}), this identification is used in construction of $\phi^{{\bf{c}}_0}$ given in \cite{Xi2}). 

The geometric construction in \textit{loc. cit.} is equivalent to the one from the Remark above (cf. Theorem 3.5 in \textit{loc. cit.}) up to the conjugation by some explicit matrix.

\subsubsection{Example when a specialization of $\J_{{\bf{c}}_0}$ at $s \in \on{Spec}\K_{Z_e}(\on{pt})$ is not semisimple}\label{ex_fiber_not_semisimple}
It is known that $\J_{{\bf{c}}_0}$ is {\emph{not}} isomorphic to a matrix algebra over $\K_{Z_e}(\on{pt})$ in general (see \cite[Section 8.3]{Xi}). Actually, it is even not true that the fiber of $\J_{{\bf{c}}_0}$ at every point of $\on{Spec}\K_{Z_e}(\on{pt})$ is semisimple (c.f. \cite[Corollary 1]{BDD} and \cite[Example 4.4]{kenta}).  

For example, for $G=\on{SL}_2$, and $s=[\on{diag}(1,-1)] \in \on{Spec}\K_{\on{PGL}_2}(\on{pt})$, the fiber of $\J_{{\bf{c}}_0} \simeq \K_{\on{PGL}_2}(\mathbb{P}^1 \times \mathbb{P}^1)$ at $s$ is {\emph{four-dimensional}} but the are only {\emph{two}} irreducible representations of  $\K_{\on{PGL}_2}(\mathbb{P}^1 \times \mathbb{P}^1)_s$ and both of them are {\emph{one-dimensional}}. To see this, note that by Theorem \ref{thm_C} (see Section \ref{sec_class_irred} and Proposition \ref{Thom} below), irreducible modules over $\K_{\on{PGL}_2}(\mathbb{P}^1 \times \mathbb{P}^1)_s$ are all of the form $K((\mathbb{P}^1)^s)_{\rho}$, where $\rho$ is an irreducible representation of $Z_{\on{PGL}_2}(s)/Z_{\on{PGL}_2}(s)^0 = \mathbb{Z}/2\mathbb{Z}$ acting simply transitively on $(\mathbb{P}^1)^s=\{0,\infty\}$. So, $K((\mathbb{P}^1)^s)$ is two dimensional and is the direct sum of two irreducible one-dimensional $\K_{\on{PGL}_2}(\mathbb{P}^1 \times \mathbb{P}^1)_s$-modules.

\section{Representation theory of $J_e$ from the geometric perspective.}\label{sec:fam}
\subsection{Simple modules over $\J_e$}\label{sec_class_irred} The goal of this Section is  classification of simple $\J_e$-modules 
resulting in the proof of Theorem \ref{thm_C}.

We first recall Lusztig's classification \cite[Theorem 4.2]{Lcells4}. It is shown in 
{\em loc. cit.} that for every pair $(s,\rho)$ of a semisimple element $s \in Z_e$ and $\rho \in \on{Irrep}(Z_{Z_e}(s)/{Z_{Z_e}(s)^0})$ there exists  unique irreducible $\J_e$-module $E(s,e,\rho)$  characterized by the following property: $\phi^{e*}_q E(s,e,\rho)$ is isomorphic to $K(e,s,\rho,q):=\K(\CB^{sq}_e)_{\rho}$ for generic $q$.
Moreover, Lusztig proved that every irreducible $\J_e$-module is of the form $E(s,e,\rho)$ for some $(s,\rho)$
as above.

Our goal is to prove that:

(a) irreducible modules over $\J_e=\K_{Z_e}(\CB_e^{\BC^*} \times \CB_e^{\BC^*})$ are of the form $\K(\CB_e^{\BC^*,s})_{\rho}$,

(b) we have: $E(s,e,\rho)=\K(\CB_e^{\BC^*,s})_{\rho}$.   

Part $(a)$ of this theorem immediately follows from  Proposition \ref{Thom}.
Part $(b)$ is a consequence of the next Lemma. Set $\Gamma=\Gamma^s_e:=Z_{Z_e}(s)/{Z_{Z_e}(s)^0}$.

\begin{lemma}\label{pullb}
There are canonical isomorphisms of $\mathcal H_e-\Gamma^s_e$- and $\CH-\Gamma^s_e$-modules respectively: 
\begin{equation*}
\phi_{e,q}^*\K(\CB_{e}^{\mathbb C^*,s}) \simeq \K(\CB_{e}^{sq}),\, \phi_q^{e*}\K(\CB_{e}^{\mathbb C^*,s}) \simeq \K(\CB_{e}^{sq}).
\end{equation*}
\end{lemma}

\begin{proof}
\noindent 
We prove the claim for $\phi^{e*}_{q}$, the argument for $\phi_{e,q}^{*}$ follows.

We have checked (see Remark \ref{bimod}) that the 
 $\mathcal H-J_e \otimes \BC[{\bf{v}},{\bf{v}}^{-1}]$-bimodule 
 corresponding to the homomorphism $\phi^e$ is identified with $\CF_e$, thus 
 we have a canonical isomorphism:
 $$\phi^{e*}M \simeq \mathcal \CF_e \otimes_{K_e} M$$
holding for any $J_e \otimes \mathbb C[{\bf{v}}, {\bf{v}}^{-1}]$-module $M$. 


Thus we are reduced to constructing an isomorphism:
\begin{equation}
\K_{Z_e \times \mathbb C^*}(\CB_e \times \CB_e^{\mathbb C^*})_q \otimes_{\K_e} \K(\CB_{e}^{\mathbb C^*,s}) \simeq_{\He} \K(\CB_{e}^{sq}).
\end{equation}

Let us denote by $C$ the product of our $\mathbb C^*$ and $\langle s \rangle$ (Zariski closure of the
cyclic subgroup generated by $s$). 
Setting $\K_e^s:= \K_C(\CB_e^{\mathbb C^*} \times \CB_e^{\mathbb C^*})^\Gamma$, we have
an isomorphism of $\He$-modules:

\begin{equation}\label{red}
\K_{Z_e \times \mathbb C^*}(\CB_e \times \CB_e^{\mathbb C^*}) \otimes_{\K_e} \K(\CB_{e}^{\mathbb C^*, s}) _q\simeq_{\He} \K_C(\CB_e \times \CB_e^{\mathbb C^*})^\Gamma \otimes_{\K_e^s} \K(\CB_{e}^{\mathbb C^*, s})_q,
\end{equation}
this follows from the fact that
$\K_{Z_e \times \mathbb C^*}(\CB_e \times \CB_e^{\mathbb C^*})$ (respectively, 
$\K_C(\CB_e \times \CB_e^{\mathbb C^*})^\Gamma $) is a free rank one module over $\K_e$ (respectively, $\K_e^s$),
both sides of \eqref{red} are isomorphic to 
$\K(\CB_{e}^{\mathbb C^*, s})_q$.

 Consider 
 $$
 {M}_{e,s,q}:=\K_C(\CB_e \times \CB_e^{\mathbb C^*})^\Gamma_q \otimes_{\K_e^s} \K(\CB_{e}^{\mathbb C^*, s}).$$ Here the central subalgebra $\K_C(\on{pt})\subset \K_e^s$ acts on the right factor via the character  
 corresponding to $sq$, we will denote  reduction by that character by the subscript $\bullet_{sq}$. 
 By Section \ref{compl} we get (cf. also (\ref{loc}) below):
\begin{equation}\label{mod} \K_C(\CB_e \times \CB_e^{\mathbb C^*})^\Gamma_{sq} = \K(\CB_{e}^{sq} \times  \CB_{e}^{\mathbb C^*, s})^\Gamma = \bigoplus_{\rho \in \on{Irrep}(\Gamma)} \K(\CB_{e}^{sq})_{\rho} \otimes \K(\CB_{e}^{\mathbb C^*, s})_{\rho^*},
\end{equation}
where in the last equality we use \cite[Theorem 5.6.1]{CG} (the implication $(b)\Rightarrow(a)$, for  the equivariance under the trivial group). 

Strictly speaking, to use it, we should establish the Kunneth formula for the smooth variety $\CB_{e}^{\mathbb C^*, s}$. However, it immediately follows from the second of the following two identifications:

\begin{equation}\label{Kun} 
K(\CB_e^{s q} \times \CB_e^{sq}) \simeq H_*(\CB_e^{s q} \times \CB_e^{s q}), \ K(\CB_e^{\BC^*,s} \times \CB_e^{\BC^*, s}) \simeq H_*(\CB_e^{\BC^*,s} \times \CB_e^{\BC^*,s}).\end{equation}

For the first of these two equations, note that $\mathcal B_e \times \mathcal B_e$ is also a Springer fiber (for $G^\vee \times G^\vee$). Thus, $H_*(\CB_e^{s q} \times \mathcal B_e^{sq})$ is isomorphic to the Chow group $A_*(\CB_e^{s q}\times \mathcal B_e^{sq})$, cf. \cite[Theorem 3.9]{DLP}. Moreover, the Chow group $A_*(\CB_e^{sq} \times \mathcal{B}_e^{sq})$ is isomorphic to the $K$-theory $K(\CB_e^{sq} \times \mathcal{B}_e^{sq})$ by \cite[Theorem III.1 (b)]{BFM}). 

The second identification also follows from \cite{DLP} and \cite{BFM} in a similar fashion.

Now, the results of section \ref{compl} also imply that
\begin{equation}\label{alg} (\K_e^s)_{sq} = \bigoplus_{\rho \in \on{Irrep}(\Gamma)} \operatorname{End} \K(\CB_{es}^{\mathbb C^*})_{\rho}.
\end{equation}

Now (\ref{alg}) acts on (\ref{mod}) via the action on the second tensor factor.

Thus $$M_{e,s,q} = \bigoplus_{\rho \in \on{Irrep}(\Gamma)} \K(\CB_{e}^{sq})_{\rho} \otimes (\K(\CB_{e}^{\mathbb C^*, s})_{\rho^*} \otimes_{\operatorname{End} \K(\CB_{e}^{\mathbb C^*, s})_{\rho}} \K(\CB_{e}^{\mathbb C^*, s})) =$$
$$ =\bigoplus_{\rho \in \on{Irrep}(\Gamma)} \K(\CB_{e}^{sq})_{\rho} \otimes (\K(\CB_{e}^{\mathbb C^*, s})_{\rho^*} \otimes_{\operatorname{End} \K(\CB_{e}^{\mathbb C^*, s})_{\rho}} (\K(\CB_{e}^{\mathbb C^*, s})_{\rho} \otimes \rho)) = \bigoplus_{\rho \in \on{Irrep}(\Gamma)} \K(\CB_{e}^{sq})_{\rho} \otimes \rho.$$
\end{proof}



\subsection{Families of modules over $J_e$ and $\CH_e$}
In the previous section we proved that for any $q \in \BC^*$ there exists an identification $\phi_q^{e*}E(e,s,\rho)=K(e,s,\rho,q)$. 

(It should be noted that for $q$ being not a root of the Poincar\'e polynomial of $W_f$, this was already shown in \cite{BK} using the algebraic result from \cite{Xi3}.)

Now, we will explain that this collection of isomorphisms fits into an algebraic family, using our geometric results on $J$.

This fact is stated in 
\cite{BK} and used in  the proof of \cite[Theorem 1.8 (3)]{BK}, cf. footnote \ref{foot2} and section \ref{end} below.


\subsubsection{} We start by introducing certain algebraic family whose fibers are of the form $K(e, s, q) := \bigoplus_{\rho \in \on{Irrep}\Gamma^s_e} \rho \otimes K(e, s, \rho, q).$

If $Z_e$ is connected modulo the center of $G^\vee$ and has simply connected derived subgroup, then the $K$-theory $\K_{Z_e \times \mathbb C^*}(\CB_e)$ yields a family which has such fibers for all $s$:
 the specialization at the maximal ideal of a semisimple conjugacy class $\gamma_{sq}$ of some $sq$, is
\begin{equation*}
\K_{Z_e \times \mathbb C^*}(\CB_e)_{\gamma_{sq}}\simeq \K(\CB_e^{sq})=K(e,s,q)=\bigoplus_{\rho \in \on{Irrep}\Gamma^s_e}   \rho \otimes K(e, s, \rho, q), 
\end{equation*}
where the first isomorphism follows from localization theorem \cite{GE1} (see Appendix A). 

In general, the situation is more complicated. 

\begin{remark}\label{foot}\label{rem}  \textit{Let  $\Gamma$ be a reductive group which is either disconnected or not simply connected. For a semisimple
conjugacy class $\gamma_g$ of some $g \in \Gamma$ the specialization $\K_\Gamma(X)_{\gamma_g}$  may be not isomorphic to $\K(X^g)$. 
In particular, the finite group $Z_{\Gamma}(g)/Z_{\Gamma}(g)^0$ acts trivially on the former, while it may act non-trivially on the latter
 (cf. \cite[5.2]{BDD}).}
 
{Nor do we have an isomorphism\begin{footnote}{However,  we have (cf. Appendix~\ref{appK1}): 
\begin{equation} \K_\Gamma(X)_{\gamma_g} \simeq \K_{Z_\Gamma(g)}(X^g)_g.\end{equation}}\end{footnote}
 between  $\K_{\Gamma}(X)_{\gamma_g}$ and the invariants $\K(X^g)^{\frac{Z_\Gamma(g)}{Z_\Gamma(g)^0}}$.}

\textit{A counterexample is provided by $\on{PGL}_2$ acting on $X=\mathbb P^1$. Let $s$ be the element $[\operatorname{diag}(1, -1)] \in \on{PGL}_2$, let $T$
be the diagonal torus and $Z=Z_{\on{PGL}_2}(s)$. Then, $X^s = Z/T$ is a set of two points. Thus, $\K_Z(X^s) = \K_T(\on{pt}) = \mathbb C\Lambda$, where $\Lambda$ is a character lattice. It is a module of rank $2$ over $\K_Z(\on{pt})$, so $\K_Z(X^s)_s$ has dimension  (at least) two. On the other hand, $\K(X^s)$ is the one-dimensional space of $\mathbb Z/2\mathbb Z$-invariants.}
\end{remark}

\subsubsection{}  
Being unable to define a single family with required fibers, we instead consider (following \cite{BK}) a collection of families defined for every semisimple $s\in Z_e$.

We fix such $s$, and a torus  $C\subset Z_{Z_e}(s)$, and proceed to construct a family of modules over $\operatorname{Spec}(\K_{C\times \mathbb C^*}(\on{pt}))$ whose specialization to a
point $(\chi, q) \in C \times \mathbb C^*$ is isomorphic to $\K(\CB_e^{s\chi q}) = H_*(\CB_e^{s\chi q},\mathbb{C})$. (The latter equality holds since the
 homology group $H_*(\CB_e^{s\chi q},\mathbb{Z})$ is isomorphic to the Chow group $A_*(\CB_e^{s\chi q})$; cf. \cite[Theorem 3.9]{DLP} and the Chow group $A_*(\CB_e^{s\chi q})$ is isomorphic to the $K$-theory $K(\CB_e^{s\chi q})$ by \cite[Theorem III.1 (b)]{BFM}).
 
Let us denote the diagonalizable group  $\langle C, s\rangle$ by $H$. We set $\mathcal K(C,s) := \K_{H \times \mathbb C^*}(\CB_e)|_{Cs \times \mathbb C^*}$, where the subscript refers to restriction to the closed subset of $\operatorname{Spec} (\K_{H \times \mathbb C^*}(\on{pt}))$.

From Section \ref{compl} it follows that for any $\chi$ as above we have:
\begin{equation}\label{loc}\widehat{\mathcal K(C,s)}^{\chi q}\simeq  \widehat{\K_{H\times \mathbb C^*}(\CB_e)}^{s\chi q} \simeq
\widehat{\K_{ H \times \mathbb C^*}(\CB_e^{s\chi q})}^{1}\simeq \widehat{\K_{ H^0 \times \mathbb C^*}(\CB_e^{s\chi q})}^{1}, \end{equation}
where by $\widehat \ \ ^t$ we mean completion at the maximal ideal of an element $t$ and $H^0$ denotes the identity component in $H$.

Applying localization theorem to the torus $H^0\times \BC^*$  we conclude that the specialization $\mathcal{K}(C,s)_{\chi q}$ is identified with $\K(\CB_e^{s\chi q})$.

The ring $\He = \K_{Z_e \times \mathbb C^*}(\CB_e \times \CB_e)$ acts naturally on $\mathcal K(C,s)$, it is easy to see
that the action on the fiber coincides with one introduced in  Proposition~\ref{Thom}
(cf. also \cite[5.11.7, 5.11.10]{CG}).

\begin{remark} 
The module $\mathcal K(C,s)$ 
is  similar to the semiperiodic module of \cite{BL}; it can be viewed as its generalization for $s \notin Z_G(T_0)$, in the notation of \textit{loc. cit.}.
\end{remark}

In a similar manner, one can form a family $\mathcal L(C,s):=\K_{H \times \BC^*}(\CB_e^{\mathbb C^*})|_{Cs \times \BC^*}$ over $\operatorname{Spec} (\K_{C \times \mathbb C^*}(\on{pt}))$ with fibers $\K(\CB_e^{\mathbb C^*,s\chi})$ and equip it with an action of  $J_e \otimes \BC[{\bf{v}},{\bf{v}}^{-1}] = \K_e$. We set $\mathcal L(C,s,q):=\mathcal L(C,s)|_{Cs \times \{q\}}$.





\subsection{}\label{fiber}
Now, the natural goal is to obtain a ``version in families'' of Lemma \ref{pullb} above, i.e. to prove that the family $\mathcal{K}(C,s)$ is isomorphic to the pullback of the family $\mathcal{L}(C,s)$ under the homomorphism $\phi_e$.

We need the following auxiliary result. 


Let us consider a natural morphism, 
\begin{equation*} \kappa\colon \K_{Cs \times \mathbb C^*} (\CB_e^{\mathbb C^*}) \otimes_{\mathcal O(C \times \mathbb C^*)} \K_{Cs \times \mathbb C^*} (\CB_e^{\mathbb C^*}) \to \K_{Cs \times \mathbb C^*} (\CB_e^{\mathbb C^*} \times \CB_e^{\mathbb C^*}).\end{equation*}

\begin{lemma}\label{kunneth} $\kappa$ is an isomorphism.
\end{lemma}

\begin{proof} It follows from \cite[Theorem 1.14]{L3} (see also \cite{DLP})
 that  $\K_{Cs \times \mathbb C^*} (\CB_e)$ is a free module over $\K_{Cs \times \mathbb C^*}(\on{pt})=\mathcal{O}(C \times \mathbb C^*)$.


Thus, it suffices to check the claim fiberwise. Similarly to  calculation (\ref{loc}) it follows from:

\begin{equation*}
K(\CB_e^{\mathbb C^*,s\chi} \times \CB_e^{\mathbb C^*,s\chi}) \simeq K(\CB_e^{\mathbb C^*,s\chi}) \otimes K(\CB_e^{\mathbb C^*,s\chi}).
\end{equation*}

The last isomorphism is clear from the argument similar to the formulas (\ref{Kun}).
\end{proof}

Set $M^\vee:=Z_{G^\vee}(C)^0$, $\Gamma_M:= \pi_0(Z_{M^\vee}(e,s))$. 

Let us denote the algebra $\K_{Cs \times \mathbb C^*}(\CB_e^{\mathbb C^*} \times \CB_e^{\mathbb C^*})$ by $\K(M)$. 

Similarly to the proof of  Lemma \ref{pullb}, one obtains a morphism

\begin{equation*}
\He_e \otimes_{\K_e} \K_{Cs \times \mathbb C^*}(\CB_e^{\mathbb C^*}) \to \K_{Cs \times \mathbb C^*}(\CB_e \times \CB_e^{\mathbb C^*})^{\Gamma_M} \otimes _{\K(M)^{\Gamma_M}} \K_{Cs \times \mathbb C^*}(\CB_e^{\mathbb C^*}).
\end{equation*}

It is an isomorphism since it is an isomorphism on fibers  and all involved modules are flat. 

Now we claim that, moreover, $\K_{Cs \times \mathbb C^*}(\CB_e \times \CB_e^{\mathbb C^*})^{\Gamma_M} \otimes_{\K(M)^{\Gamma_M}} \K_{Cs \times \mathbb C^*}(\CB_e^{\mathbb C^*}) \simeq \K_{Cs \times \mathbb C^*}(\CB_e)$.

Indeed, we observe that 
$$
\bigoplus_{\rho \in \on{Irrep}\Gamma_M} \K_{Cs \times \mathbb C^*}(\CB_e^{\mathbb C^*})_{\rho} \otimes_{\mathcal{O}(Cs \times \mathbb C^*)} \K_{Cs \times \mathbb C^*}(\CB_e^{\mathbb C^*})_{\rho^*}\simeq  \operatorname{End}_{\mathcal{O}(Cs \times \mathbb C^*)} \K_{Cs \times \mathbb C^*}(\CB_e^{\mathbb C^*})_{\rho} =: \mathcal A,
$$
which follows fiberwise from the argument similar to the equations (\ref{Kun}).

Moreover, since the vector bundle (over the corresponding torus) $\K_{Cs \times \mathbb C^*}(\CB_e^{\mathbb C^*})$ is self-dual in a canonical way (thanks to Poincar\'e pairing), there exists the canonical morphism 
\begin{equation*}
\rho^* \otimes \mathcal O(C) \to \K_{Cs \times \mathbb C^*}(\CB_e^{\mathbb C^*})_{\rho} \otimes_{\mathcal A} \K_{Cs \times \mathbb C^*}(\CB_e^{\mathbb C^*}).
\end{equation*}

One easily sees that, in fact, it establishes the isomorphism. Indeed, it is enough to check this fiberwise, and we reduce the statement to an elementary linear algebra observation.

We conclude:

$$
\K_{Cs \times \mathbb C^*}(\CB_e^{\mathbb C^*})_{\rho} \otimes_{\mathcal A} \K_{Cs \times \mathbb C^*}(\CB_e^{\mathbb C^*}) \simeq \rho^* \otimes \mathcal O(C).
$$

Now, as in Lemma~\ref{pullb}, we obtain:

\begin{thm}\label{auxe} As $\He_e-\Gamma_M$-modules,
$\phi_e^*\mathcal L(C, s) = \K_{Cs \times \mathbb C^*}(\CB_e).$
\end{thm}

Now the arguments in the previous chapter generalize verbatim to prove the following $\Gamma_M$-equivariant isomorphism.

\begin{thm}\label{H}
\begin{equation}
\phi^{e*}\mathcal L(C, s) \simeq_{\He} \K_{Cs \times \mathbb C^*}(\CB_e).
\end{equation}
\end{thm}

\begin{remark} This  proves the version of Corollary 2.6 in \cite{BK} ``\textit{in families}''.
\end{remark}

So, Theorem \ref{thm_ident_families_pull_back} is proven. 





 \section{Braverman-Kazhdan's  spectral description of $J_e$}\label{end}

 \subsection{}\label{intro} 
 We are now ready to reprove Theorem 1.8 (3) in \cite{BK} addressing the points in footnotes \ref{foot2}, \ref{foot3} using our geometric approach to $\J_e$. 

 Let $\bf G$ be a version of $G$ over a local field; let $q$ be a characteristic of a residue field; we assume that $q$ is large enough. (Our convention for bold letters is slightly different from the one in \cite{BK}).

 Let $\bf M$ be a Levi subgroup of $\bf G$. Let $\sigma$ be an irreducible tempered representation of $\bf M$ and let $\chi$ be a character of $\bf M$.

 Then (cf. \cite[p.78]{Bern}) all representations of the form $\operatorname{Ind}^{\bf G}_{\bf M} (\sigma \otimes \chi)$ can be realized on the same vector space that we denote $V_{\sigma}$. 
 
 Let $Z_M$ be the connected component of  $1 \in Z(M^\vee)$ (recall that $Z(M^\vee)$ is the center of $M^\vee$). Then we obtain an action of $\bf G$ on the trivial vector bundle over a torus $Z_M$.

 By taking Iwahori-invariants, we obtain a family of $\He$-modules over $Z_M$ (cf. \cite[1.2]{BK}): let us denote it by $\mathcal V(Z_M, \sigma).$ 
 In particular, setting $\sigma=K(s, e, \rho, q)$ we get
 a family that we denote by $\mathcal V(Z_M, s,\rho,q)$
 (since $e$ is fixed it does not enter the notation).
 

 \subsection{}\label{BK_formulation_thm}

It is known that: 
\\
 1) for any $L^\vee$ containing $M^\vee$, the family $\mathcal V(Z_L, s, \rho,q)$ can be naturaly realized as a subfamily $\mathcal V(Z_M, s, \rho_M,q)$, where $\rho_M$ is defined as pull
 back of $\rho$ under $\pi_0(Z_{M^\vee}(e, s)) \to \pi_0(Z_{L^\vee}(e, s))$;
 \\
2) for any compact $s'$ and a triple $(L^\vee, t, \theta)$ conjugate to $(M^\vee, s, \rho)$, the family $\mathcal V(Z_M, s, \rho,q)$ can be rationally identified with $\mathcal V(Z_L, t, \theta,q)$ via the intertwining operator.

 The next statement is an equivalent form of Theorem 1.8 (3) in \cite{BK}.

 \begin{thm}\label{BK}
     Let  $\Pi=\prod_{M, s, \rho} \on{End}^{\mathrm{rat}.}_{\mathcal O(Z_M)} \mathcal V(Z_M, s,\rho,q)$, where the product is taken over all triples
      $M,\, s,\rho $ as above with $s$ compact. Let $\mathcal S_e\subset \Pi$ be the subalgebra 
 consisting of elements $\phi$ which satisfy the following conditions:
     \\
     a) $\varphi $ does not have poles  at points of families $\mathcal V(Z_M,s,\rho,q)$ corresponding to 
     non-strictly positive characters of Levi subgroups;
     \\
     b) $\varphi$ is compatible with 1) and 2) above.
     
     Then $\mathcal S_e \simeq \J_e$.
 \end{thm}

 The goal of this section is to summarize the proof  as an application of our geometric description of $\J_e$.

\subsection{} 


Let $\CB_{e, M^\vee}$ be the Springer fiber in $M^\vee$. Then there is an embedding of the lowest (in the natural ``Bialynicki-Birula'' order) component of $Z_M$-fixed points: $i\colon \CB_{e, M^\vee} \to \CB_e$. We have a natural morphism
\begin{equation}\label{ind}
i_*\colon \K_{\langle Z_M, s \rangle \times \mathbb C^*}(\CB_{e, M})_q|_{Z_Ms \times \BC^*} \to \K_{\langle Z_M, s \rangle \times \mathbb C^*}(\CB_e)_q|_{Z_Ms \times \BC^*}.
\end{equation}

Note that the source of this map is a trivial family of $\He_M$-modules,
 the $\rho$-multiplicity subspace
 for the action of $\pi_0(Z_{M^\vee}(e, s))$ on its fiber equals $\sigma$.

\subsection{}\label{comp} 
For a finite group $\Gamma$ with an irreducible representation $\rho$ and any representation $V$, we will denote the multiplicity space $\on{Hom}_{\Gamma}(\rho, V)$ by $V_{\rho}.$
We proceed to compare the geometrically defined family $\mathcal{K}(Z_M, s,q)_{\rho}$ with $\mathcal V(Z_M, s, \rho,q)$.

\begin{prop}\label{prop_op_ex_rest_iso} There exists an open set $\mathcal U \subset Z_M$, containing all non-strictly positive characters, so that over $\mathcal U$ the natural map

\begin{equation}\label{ind_mor_v_to_k} \He_G \otimes_{\He_M} K_{\langle Z_M, s \rangle \times \mathbb C^*}(\CB_{e, M})_q|_{Z_Ms} \to \K_{\langle Z_M, s \rangle \times \mathbb C^*}(\CB_e)_q|_{Z_Ms}\end{equation}

induced by $i_*$, is an isomorphism.
\end{prop}

\begin{proof}
It is enough to prove this statement fiberwise. In this form it is contained in \cite[2.2]{BK}.
    \end{proof}

\begin{cor}\label{cor_ident_fam_S_E} $\mathcal V(Z_M,s,\rho,q)|_{\mathcal U} \simeq \mathcal K(Z_M,s,q)_{\rho}|_{\mathcal U}.$
\end{cor}
\begin{proof}
It follows from the definitions that $\mathcal V(Z_M,s,\rho,q)$ identifies with the LHS of (\ref{ind_mor_v_to_k}). Now the claim follows from Proposition \ref{prop_op_ex_rest_iso}.
\end{proof}

It follows from Corollary \ref{cor_ident_fam_S_E} that $\mathcal S_e$ embeds into 
\begin{equation}\label{target}
\prod_{M, s, \rho} \on{End}^{\mathrm{rat}.}_{\mathcal O(Z_M)} \mathcal K(Z_M, s,q)_{\rho}.
\end{equation}
Moreover, by 2) in Section \ref{BK_formulation_thm}, and the fact that any character is conjugate to a non-strictly positive one, LHS can be replaced by $\prod_{M, s, e, \rho} \on{End}_{\mathcal O(Z_M)} \mathcal{K}(Z_M,s,q)_{\rho}$. 

Moreover, similarly to Theorem~\ref{H}, we are reduced to showing that the following holds.

\begin{prop}\label{fin}
    Let $\mathcal E$ be the subalgebra of $\prod_{M, s} \on{End}_{\mathcal O(Z_M)} \mathcal L(Z_M,s,q)=:\tilde{E}$ consisting of elements $\phi = (\phi(M, s))_{M, s}$ satisfying the following 
    property.

    For any pair of Levi subgroups $M^\vee, \, L^\vee \subset G^\vee$ whose Lie algebras contain $e$  and elements
  $s\in M^\vee$, $t \in L^\vee$; $\chi\in Z_M$,     $\chi' \in Z_L$ such that
    $s\chi$ is conjugate to  $t\chi'$ we have:
     \begin{equation}\label{BK1}
     \phi(M, s)_\chi = \phi(L, t)_{\chi'}.
     \end{equation}
     
     Then $\mathcal E \simeq \J_e$ via the natural action map $\alpha\colon \J_e \to \mathcal E$.
\end{prop}



The proof of Proposition~\ref{fin} will occupy the rest of this section.

\subsection{} First we reduce $\tilde{E}$ to a finite product.

It is known (cf. \cite{M}, \cite{Kn}) that there exists a set-theoretic lifting $l\colon \pi_0(Z_{Z_e}(s))=\Gamma^s_e \to Z_{Z_e}(s)$, so that for any $\gamma \in \pi_0(Z_{Z_e}(s))$, one has:

1) $\operatorname{Ad}_{l(\gamma)}$ preserves a pinning of $Z_{Z_e}(s)^0$ (say, $B^{\pm}_{\gamma}$, $T_\gamma$);

2) every element of $Z_{Z_e}(s)^0\gamma$ is conjugate to an element in $(T_{\gamma}^{\gamma})^0 \gamma$ (the upper index $\gamma$ stands for the invariants of $\operatorname{Ad}$-action).

We will denote $(T_{\gamma}^{\gamma})^0$ by $C(\gamma)$, and we will denote by $L(\gamma)^{\vee}$ the Levi subgroup  $Z_{G^\vee}((T_{\gamma}^{\gamma})^0)^0 \subset G^\vee$. 

From 2) above it follows that the natural projection $$\pi = \prod \pi_\gamma\colon \mathcal E \to \prod_{\gamma \in \Gamma_e^s} \operatorname{End}_{\mathcal O(Z_{L(\gamma)})} \mathcal L(Z_{L(\gamma)}, \gamma,q)$$ is injective. 





\subsection{} We calculate the completion $\widehat{\mathcal E}^s$ for a semisimple conjugacy class $[s] \in Z_e$.  Let $T(s)$ be a maximal torus inside $Z_{Z_e}(s)$. Let $M(s)^{\vee}$ be $Z_{G^\vee}(T(s))^0$ and recall that $T(s)=Z_{M(s)}$ is the connected component of $1 \in Z(M(s)^\vee)$. 

As above, the map 
\begin{equation*}
\pi_s\colon \mathcal E \to \operatorname{End}_{\mathcal O(T(s))} \mathcal L(Z_{M(s)},s,q) \times \prod\limits_{\gamma \in \Gamma_e, \\ C(\gamma)\gamma \cap [s] = \emptyset} \operatorname{End}_{\mathcal O(C(\gamma))} \mathcal L(Z_{L(\gamma)}, \gamma,q)
\end{equation*}
(which is defined in the same manner as $\pi$ above) is injective.

\begin{prop}\label{appendix} For any $\gamma \in \Gamma_e$, such that $Z_{L(\gamma)}\gamma \cap [s] = \emptyset$, we have

$$ \widehat{\operatorname{End}_{\mathcal O(C(\gamma))} \mathcal L(Z_{L(\gamma)}, \gamma,q)}^s = 0.
$$

Here the left hand side is the completion of $\operatorname{End}_{\mathcal O(C(\gamma))} \mathcal L(Z_{L(\gamma)}, \gamma,q)$
at the maximal ideal
$\K_{Z_e}(\on{pt})$ corresponding to $s$. 
\end{prop}

\begin{proof} \textbf{Step 1.} It suffices to construct an $\operatorname{Ad}(Z_e)$-invariant function $f$ on $Z_e$ such
that $f(s) = 0$ and $f|_{C(\gamma)\gamma} = 1$. Thus it is enough to prove that in the coarse quotient $Z_e/\!/\operatorname{Ad}(Z_e)$ the image of $C(\gamma)\gamma$ is closed and does not contain the image of $s$.

This follows from the following properties of such a quotient.

\textbf{Step 2.} Let $K$ be a possibly non-connected algebraic group. According to the results of Mohrdieck (\cite{M}; cf. also \cite{Kn}) there exists a set-theoretic lifting $l\colon \Gamma_K:= \pi_0(K) \to K$, so that:

 1) there exists a maximal torus $T\subset K^0$ normalized by $l(\gamma)$ for any $\gamma \in \Gamma_K$.
 
 2) for a certain finite group $\mathcal W$ we have: $T^{l(\gamma)}\gamma/\mathcal W \simeq G^0\gamma/\!/G^0.$ 





\end{proof}

Combining the above statements we get the following.

\begin{lemma} The map $\widehat{\pi_s}^s:\widehat{\mathcal E}^s\to \widehat{\operatorname{End}_{\mathcal O(T(s))} \mathcal L(Z_{M(s)}, s)}^s$
is injective.
\end{lemma}

\subsection{}  Recall that $T(s)$ is a maximal torus inside $Z_{Z_e}(s)^0$, let $W_s$ being the Weyl group of $Z_{Z_e}(s)^0$. 
Then the image of $\widehat{\pi_s}^s$ lies inside  \begin{equation}\label{c}((\widehat{\operatorname{End}_{\mathcal O(T(s))}} \mathcal L(Z_{M(s)},s,q)^{s})^{W_s})^{\pi_0(Z_{Z_e}(s))}. \end{equation}

Let $\mathfrak{m}_{1} \subset \K_{Z_{Z_e}(s)}(\on{pt})$ be the augmentation ideal of $1$. Let  $\mathfrak m \subset \K_{Z_{Z_e}(s)^0}(\on{pt})$ be the ideal generated by the image of $\mathfrak m_{1}$ in $\K_{Z_{Z_e}(s)^0}(\on{pt})$ (under the restriction homomorphism $\K_{Z_{Z_e}(s)}(\on{pt}) \rightarrow \K_{Z_{Z_e}(s)^0}(\on{pt})$). Let us note that:

1) $\mathcal L(Z_{M(s)},s,q)$ is the vector bundle over $T(s)$ with  fiber $\K(\CB_e^{\mathbb C^*,s})$; thus, it can be non-canonically identified with $\K(\CB_e^{\mathbb C^*,s}) \otimes R_s$, where $R_s$ is $\K_{T(s)}(\on{pt})$;

2) as  was already stated, the completion is in the sense of $\K_{Z_e}(\operatorname{pt})$-modules;

3) $\widehat{\K_{Z_e}(\on{pt})}^s = (\widehat{\K_{Z_{Z_e}(s)^0}(\on{pt})}^{\mathfrak m})^{\pi_0(Z_{Z_e}(s))}$ (cf. Remark \ref{rem_compl} for $Y=\on{pt}$).

Now we see that the image of $\widehat{\pi_s}^s$ lies inside $(\operatorname{End}(\K(\CB_e^{\mathbb C^*,s})) \otimes \widehat{\K_{Z_{Z_e}(s)^0}(\on{pt})}^\mathfrak{m})^{\pi_0(Z_{Z_e}(s))}$.

But, similarly to (\ref{loc}), by the results of Section \ref{compl},
\begin{equation}\label{conj}
\widehat{\J_e}^s = \widehat{\K_e}^s = (\operatorname{End}(\K(\CB_e^{\mathbb C^*,s})) \otimes \widehat{\K_{Z_{Z_e}(s)^0}(\on{pt})^{\mathfrak{m}}})^{\pi_0(Z_{Z_e}(s))}.
\end{equation}

Now it follows that the map $\alpha$ (introduced in (\ref{fin})) induces an identical isomorphism on completions (and, in particular, $\widehat{\pi_s}^s: \widehat{\mathcal E}^s \to (\operatorname{End}(\K(\CB_e^{\mathbb C^*,s})) \otimes \widehat{\K_{Z_{Z_e}(s)^0}(\on{pt})^{\mathfrak{m}}})^{\pi_0(Z_{Z_e}(s))}$ is surjective). Thus, Proposition~\ref{fin} is proven. 
\\
\\
\begin{remark} Isomorphisms  in 1) and (\ref{conj}) in this subsection are non-canonical, cf. Appendix~\ref{appK1}. 
As follows from \textit{loc.~cit.}, both of them are uniquely determined by the choice of a trivialization of the vector bundle $\K_{T(s)}(\mathcal{B}_{e}^{\mathbb C^*,s})$ over $\K_{T(s)}(\on{pt})$. 
Above we implicitly assumed that the same trivialization is used in both cases.
\end{remark}

\section{The structure of the cocenter of $J$}\label{cocenter_J_section}
\subsection{} In this concluding section we will describe $C(\J)$ and   $C(\He)_q$ 
for  generic $q \in \mathbb{C}^*$, 
where $C$ stands for the {\em{cocenter}}, i.e. the 0-th Hochschild homology of an associative ring.
   This will prove Conjecture $2$ from \cite{BDD}. 
The main theorem is as follows.

For a reductive algebraic group $H$ let ${\bf{Comm}}_H$,
$\operatorname{Comm}_H$ 
be the corresponding commuting variety equipped with the natural (respectively, reduced) scheme structure.

We will be also interested in the categorical (coarse) quotients of these schemes: 
\begin{equation*}
{\bf{Comm}}_{H}/\!/H=:{\bf  C}_{H},\, \operatorname{Comm}_{H}/\!/H=: \mathcal C_{H}.
\end{equation*}

\begin{remark}By the result of \cite{LNY} the ring ${\mathcal O} ({\bf C}_H) ={\mathcal O}({\bf{Comm}}_{H})^{H}$  is reduced if the reductive group $H$ is connected. One expects that this theorem generalizes to the case of a not necessarily connected reductive group, such a generalization would imply that ${\bf  C}_{H}=\mathcal C_{H}$.
\end{remark}

 Let $\mathcal{O}^a(\mathcal{C}_{Z_e})=\mathcal O^{a}(e)\subset \mathcal O \mathcal(\mathcal C_{Z_e})$ (where 
 ``$a$''  stands for ``admissible'') denote the vector space of regular  functions $f$ on $\mathcal C_{Z_e}$
 satisfying the following property.
\\
\\
For a semisimple $s \in Z_e$ let $f_s$ denote the pull-back of $f$ under the map $x\mapsto (x,s)$
from the centralizer $Z_{Z_e}(s)$ of $s$ in $Z_e$ to $\operatorname{Comm}_{Z_e}$.
Then $f\in \mathcal O^{a}(e)$ iff for any semisimple $s\in Z_e$
the function $f_s$ is a linear combination of \textit{admissible} characters of the group $Z_{Z_e}(s)$.
Here by an admissible character we understand the character of an irreducible 
representation  appearing in $\K(\CB_e^s)$, (or, equivalently, in $\K(\CB_e^{\mathbb C^*,s})$; cf. \cite[2.8]{Lcells4}).  

Notice that an admissible representation factors
through the group of components of $Z_{Z_e}(s)$, thus for $f\in \mathcal O^{a}(e)$ the function $f_s$
is locally constant for every $s\in Z_e$.
\\
\begin{thm}[Theorem \ref{intro_final_cocenter}]\label{final}
We have a canonical isomorphism $C(\K_{Z_e}(\CB_e^{\BC^*} \times \CB_e^{\BC^*})) \simeq \mathcal O^a(e)$.
\\
\end{thm}

\begin{cor} We have isomorphisms:
\begin{equation}\label{OCJ}\mathcal O^a(e) \simeq C(\J_e),
\end{equation}
\begin{equation}\label{OCH}
\bigoplus\limits_{e\in \mathcal N/\sim}\mathcal O^a(e) \simeq C(\He_q),
\end{equation}
\\
\\
where  isomorphism \eqref{OCJ} depends on a choice of $q \in \{\pm 1\}$; in \eqref{OCH}
 $q$ is assumed not to be the root of unity.
\end{cor}

\proof The first isomorphism follows from  Theorem \ref{final} and Theorem \ref{thm_A}. 
The second one then follows from 
\cite[Theorem 1]{BDD} which establishes an isomorphism $C(J_e)\simeq C(\mathcal{H}_q)$ when $q$ is not a root of unity. \qed

\begin{remark} Another approach to constructing isomorphism \eqref{OCH}, as well as a stronger version 
describing the unipotent part of the cocenter of the $p$-adic group, 
will be presented in \cite{BCKV} (see also a related result \cite{ACR}). 

In that approach it is realized as a ``decategorification'' of a result of Ben Zvi-Nadler-Preygel \cite{BNP}
describing the trace of the affine Hecke category as coherent sheaves on commuting pairs of elements
in $G$. 
\end{remark}

The rest of this section is devoted to  proofs. 

\subsection{} 
Let $X$ be a scheme, $\F$ a locally free coherent sheaf on $X$ and $\phi\colon \F\to \F$ and endomorphism.
To this data one assigns a regular function $\on{Tr}(\phi)\in \Gamma(X,\mathcal O_X)$; this is a special case
of the more general  Hattori-Stallings trace (cf. \cite{Hat}, \cite{Sta}). 
The trace is additive on short exact sequences, thus it extends to the
derived category of perfect complexes.
The construction is manifestly
local in the fppf topology, so it extends to algebraic stacks. 

Let now $X$ be an algebraic stack over a field $k$ and $I(X)=X\times _{X^2_k}X$ be the inertia stacks.
We get two natural isomorphisms between the composed morphisms of stacks $I(X)\to X^2\overset{pr_1}{\longrightarrow} X $ and $I(X)\to X^2\overset{pr_2}{\longrightarrow} X $,
composing the first one with the inverse of the second we get an automorphism of the first composition
(we denote that composition by $pr$). Thus
for $\F\in \on{Coh}(X)$ the sheaf $pr^*(\F)$ acquires a canonical 
automorphism $c_\F$. For example, if $X=S/H$ where $S$ is a scheme and 
$H$ an algebraic group then $I(X)= \tilde I(X)/H$ where $\tilde I(X)=\{(x,h)\ |\ h(x)=x\}\subset S\times H$.
In this case the action of $c_\F$ on the fiber of $pr^*(\F)$ at $(x,h)$ equals the action of $h$ on the fiber of $\F$ at $x$.

\begin{remark}
One expects an isomorphism  $R\Gamma({\mathcal O}(I(X)))\simeq \on{HH}_*(\on{Coh}(X))$, having such an isomorphism one could define the function $c_\F$ as the image
of the class $[\F]$ under the trace map from $K$-theory to Hochschild homology.\footnote{We thank Jakub L\" owit who pointed it out to us.}
 We were not able to find a reference for such an isomorphism, so we resorted to the above less direct construction.  
\end{remark}

\begin{lemma}\label{loc_cons} Let  $X$
 be a scheme over a characteristic zero field $k$ and $H$ an affine algebraic group over $k$. Let $\F$ be a perfect complex
 on the stack $X/H$ and $f=\on{Tr}(c_\F)\in {\mathcal O} (I(X))$. For $s\in H(k)$ let $i_s\colon X^s\to I(X)$ be given by $x\mapsto (x,s)$.
 Then $i_s^*(f)$ is locally constant.
\end{lemma}

\proof Without loss of generality we can assume that $k$ is algebraically closed.

Let $H_s$ be the Zariski closure of the cyclic group $\langle s\rangle$ in $H$.
Thus $H_s$ is abelian algebraic group, so it is a product of a diagonalizable and a vector group.
The restriction of an equivariant coherent sheaf $\F$ to $X^s$ carries an action of the diagonalizable group $H_s^{\mathrm{diag}}$, 
thus $\F_s:=\F|_{X^s}$ splits as a direct sum of subsheaves $\F_s^\chi$ where $\chi$ runs over the
characters of $H_s$, where $H_s$ acts on $\F_s^\chi$ via $\chi$. It follows that such a decomposition is also
well-defined for a perfect complex.
Since the Euler characteristic of a perfect complex is a locally constant function, the claim follows. \qed

\quad

We now consider $X=H/H$, the quotient of a reductive algebraic group by the conjugation action.
Thus $I(X)={\bf Comm}_H/H$; applying the above construction we get a map $\tau_H\colon \K^H(H)\to \mathcal O({\bf C}_H)$
$[\F]\mapsto \on{Tr}(c_\F)$. By Lemma \ref{loc_cons} it lands in the space $\mathcal O_l({\bf C})$ of functions
satisfying the local constancy condition stated in the Lemma.

\subsection{} We now proceed to prove   the following reformulation of Theorem~\ref{final}.

\begin{thm} Set $X =\CB_e^{\mathbb C^*}$, let  ${\bf{a}}\colon Z_e\times X\to X$ be the action map, and $\pi\colon Z_e\times X\to X$ the  projection. Consider the map
\begin{equation*}
c\colon  \J_e \simeq \K_{Z_e}(X^2) \xrightarrow{({\bf{a}} \times \operatorname{Id})^*}  \K_{Z_e}(Z_e \times X)  \xrightarrow{\pi_*} \K_{Z_e}(Z_e) \xrightarrow{\tau_{Z_e}}\mathcal O_l({\bf{C}}_{Z_e}) \to \mathcal O(\mathcal{C}_{Z_e}),
\end{equation*}
where the last arrow is  the  restriction to the reduced subscheme $\mathcal{C}_{Z_e}\subset {\bf C}_{Z_e}$. (Compare with the definiton of the map $T$ in \ref{mapt}.)

Then:
1. $[J_e, J_e]$ lies in the kernel of $c$, and $c$ induces an injective map on $J_e/[J_e, J_e]$. 

2. The image of $c$ lies inside $\mathcal O^a(e)$. 

3. The image of $c$ equals $\mathcal O^a(e)$. 
\end{thm}

We will start with the proof of $2$, and then proceed with $1$ and $3$.

\subsection{Admissibility}

\subsubsection{} Recall that $X=\CB^{\BC^*}_e$. For $\mathcal F \in \operatorname{Coh}_{Z_e}(X \times X)$, we will denote $\sigma^*\mathcal F$ by $\bar{\mathcal F}$, where $\sigma$ is the involution swapping factors. For $\mathcal F$ as above, and $\mathcal{P}=\pi_* ({\bf{a}} \times \on{Id})^*\mathcal{F}$,  $f_{\mathcal P}|_{\{s\} \times Z_{Z_e}(s)}$ (by Lemma~\ref{loc_cons}) can be considered as a function on $\pi_0(Z_{Z_e}(s))=:\Gamma^s_e$ invariant under the conjugation. Let us denote this function by $f_{\mathcal F}^s$. Recall that for $\gamma \in \Gamma^s_e$ we have 
\begin{equation*}
f^s_{\mathcal{F}}(\gamma)=\on{Tr}_s i_\gamma^* \pi_* ({\bf{a}} \times \on{Id})^* \CF,
\end{equation*}
where $i_\gamma\colon \{\tilde{\gamma}\} \hookrightarrow Z_e$ is an embedding in $Z_e$ of some  lifting $\tilde{\gamma} \in Z_e$ of $\gamma$.

\begin{lemma}\label{Lmadm} 
Let $\rho$ be any irreducible representation of $\Gamma^s_e$ and let $\chi_\rho$ be its character. Then the invariant pairing $\langle \chi_\rho, f_{\mathcal F}^s \rangle$ is equal to $\operatorname{Tr}_{\bar{\mathcal F}}(\K(\CB_e^{\mathbb C^*,s})_\rho)$ (recall that $\bar{\mathcal{\CF}}$ acts on $\K(\CB_e^{\mathbb C^*,s})_\rho$ via the convolution).

\end{lemma}

\begin{proof} 

We start with a piece of notation.

Let us for $\gamma \in \Gamma_e^s$ consider the operator $\gamma \mathcal F := \gamma \circ \mathcal F$ acting on $\K(X^s)$.  Let also $\operatorname{Eu}$ stand for the Euler characteristic.

\textbf{Step 1.} Since (all $^*$- and $_*$-functors below are derived) $$\langle \chi_\rho, f_{\mathcal F}^s \rangle = \frac{1}{|\Gamma^s_e|}\sum_{\gamma \in \Gamma^s_e} \chi_{\rho}(\gamma^{-1}) f^s_{\CF}(\gamma)=\frac{1}{|\Gamma^s_e|}\sum_{\gamma \in \Gamma^s_e} \chi_{\rho}(\gamma^{-1}) \operatorname{Tr}_s i_\gamma^*\pi_*({\bf{a}} \times \operatorname{Id})^* \mathcal F,$$
(here and below $\operatorname{Tr}_s$ is always a {\textbf{graded}} trace), and 
$$
\operatorname{Tr}_{\bar{\mathcal F}}(\K(\CB_e^{\mathbb C^*,s})_\rho) = \frac{1}{|\Gamma^s_e|}\sum_{\gamma \in \Gamma^s_e} \chi_{\rho}(\gamma^{-1}) \operatorname{Tr}_{\gamma\bar{\mathcal F}}(\K(\CB_e^{\mathbb C^*,s})),
$$

we have to prove that 
$$
\operatorname{Tr}_{\gamma\bar{\mathcal F}}(\K(\CB_e^{\mathbb C^*,s})) = \operatorname{Tr}_s (i_\gamma^*\pi_*({\bf{a}} \times \operatorname{Id})^* \mathcal F).
$$

Equivalently, we are interested in the following equality:

\begin{equation}\label{traces1}
\operatorname{Tr}_{\gamma\bar{\mathcal F}}(\K(\CB_e^{\mathbb C^*,s})) = \operatorname{Tr}_s H^*(X, \Gamma_{\gamma}^*\mathcal F),
\end{equation}
where $\Gamma_{\gamma}$ stands for the graph embedding $\gamma\colon X \to X \times X$ and $X=\CB^{\BC^*}_e$.

(Let us recall, that we can interchange $K$-theory and (co)homology of $X$ because of \cite[Theorem 3.9]{DLP} and \cite[Theorem III.1 (b)]{BFM}.)

\textbf{Step 2.} Now we would like to use the results of Section \ref{traces}. Namely, Proposition~\ref{tr} together with the proof of Proposition \ref{Thom} say that \begin{equation}\label{tr1}\operatorname{Tr}_{\gamma\bar{\mathcal F}}(\K(\CB_e^{\mathbb C^*,s})) = \sum_{\chi \in \BC^*}\chi\on{Eu}((\lambda_s^{-1} \boxtimes \mathcal O_{X^s}) \otimes^{\mathbb L} (\Gamma^s(\gamma) * i^*_s \bar{\mathcal{F}}) \otimes^{\mathbb L} \mathcal{O}_\Delta)\end{equation} for $\Delta$ being the diagonal inside $X^s \times X^s$,  $\Gamma^s(\gamma)$ being a structure sheaf of the graph of $\gamma$ inside $X^s \times X^s$, and $i_s\colon X^s \times X^s \hookrightarrow X \times X$ being the natural embedding.

But, using a base change from \cite[Proposition 1.4]{To}, as in the proof of Proposition \ref{theorem}, for the diagram 

\[\begin{tikzcd}
	{X^s \times X^s} && {X^s \times X^s \times X^s} \\
	X^s && {X^s \times X^s,}
	\arrow["{\pi_{13}}", from=1-3, to=2-3]
	\arrow["{\pi_1}", from=1-1, to=2-1]
	\arrow["i_\Delta", from=2-1, to=2-3]
	\arrow["{(a, b) \mapsto (a, b, a)}"', from=1-1, to=1-3]
    \arrow["j", from=1-1, to=1-3]
\end{tikzcd}\]
one  sees that for any sheaves $\mathcal A$ and $\mathcal B$ on $X^s \times X^s$,\begin{equation}\label{basec}(\mathcal A * \mathcal B) \otimes^{\mathbb L} \mathcal O_{\Delta} = i_{\Delta}^*\pi_{13*}(\pi_{12}^*\mathcal A \otimes^{\mathbb L} \pi_{23}^*\mathcal B) = \pi_{1*} j^* (\pi_{12}^*\mathcal A \otimes^{\mathbb L} \pi_{23}^*\mathcal B) = \pi_{1*} (\mathcal A \otimes^{\mathbb L} \sigma^* \mathcal B)=\pi_{1*} (\mathcal A \otimes^{\mathbb L} \bar{\mathcal B}),\end{equation}

where all of the functors are, as usual, derived (this also holds below); and $i_{\Delta}$ is an embedding of the diagonal.

So, by the proof of Proposition~\ref{Thom}, we get (in the notation of \textit{loc. cit.}) \begin{equation}\label{LHS}\operatorname{Tr}_{\gamma\bar{\mathcal F}}(\K(\CB_e^{\mathbb C^*,s})) = \sum_{\chi \in \BC^*} \chi \operatorname{Eu}((\lambda_s^{-1} \boxtimes \mathcal O_{X^s}) \otimes^{\mathbb L} \Gamma^s(\gamma) \otimes^{\mathbb L} i_s^*\mathcal F_{\chi}) = \end{equation} 
$$= \sum_{\chi \in \BC^*}\chi \operatorname{Eu} (\Gamma_{\gamma}^s)^* ((\lambda_s^{-1} \boxtimes \mathcal O_{X^s})\otimes^{\mathbb L} i_s^*\mathcal F_{\chi}) = \sum_{\chi \in \BC^*} \chi \operatorname{Eu}(\lambda_s^{-1} \otimes^{\mathbb L} (\Gamma_{\gamma}^s)^* i_s^*\mathcal F_\chi).$$

Here 
$\Gamma_{\gamma}^s$ is an embedding $X^s \to X^s \times X^s$  of the graph of $\gamma$, and we use the following simple observation: for ${\bf{a}}_{\gamma}^s\colon X^s \to X^s$ being a $\gamma$-action map, $(\Gamma_{\gamma}^s)^*(\lambda_s^{-1} \boxtimes \mathcal{O}_{X^s}) = i_{\Delta}^* (1 \boxtimes {\bf{a}}_{\gamma}^s)^*(\lambda_s^{-1} \boxtimes \mathcal{O}_{X^s}) = \lambda_s^{-1} \otimes \mathcal{O}_{X^s} = \lambda_s^{-1}$.

\textbf{Step 3.} On the other hand, by a localization theorem (cf. \cite[5.11.8]{CG} and the proof of our Proposition~\ref{Thom}), 

\begin{equation}\label{via_loc_comp}
\operatorname{Tr}_s H^{*}(X, \Gamma_{\gamma}^*\mathcal F) = \operatorname{Tr}_s H^{*} (X^s, \lambda(s)^{-1} \otimes^{\mathbb L} {i'_s}^*\Gamma_{\gamma}^*\mathcal F),
\end{equation}
where $i'_s\colon X^s \hookrightarrow X^s \times X^s$ is the natural embedding.

It is clear that (\ref{via_loc_comp}) is equal to (\ref{LHS}).
\end{proof}

\subsubsection{} By the previous lemma, only admissible characters of $\Gamma^s_e$ may enter $f_{\mathcal F}^s$. This finishes the proof of the admissibility.

\subsection{Injectivity} 

The fact that $[\J_e, \J_e]$ lies inside the kernel, follows from the proof of Lemma~\ref{Lmadm}. However, the more geometric proof of the same fact exists (which also does not require our discussion about the reducedness above). For this, see \ref{mapt}.

In view of Lemma~\ref{Lmadm}, injectivity follows once we check the density of characters for $\J$, i.e., the statement that an element
$h\in \J$ whose trace in every finite dimensional $\J$-module vanishes lies in $[\J,\J]$. By 
 \cite[Theorem 1]{BDD} the cocenter
of $\J$ is isomorphic to the cocenter of $\He_q$ for almost all $q$.
 Thus the density of characters for $\J$ follows from the density of characters for $\He_q$ together with the fact that pull back of finite dimensional $\J$ modules generate the Grothendieck group of finite dimensional $\He_q$ modules (see  \cite[Lemma 1.9]{Lcells3}).
 \qed

\subsection{Surjectivity}
 By the results of the previous subsections, we have a morphism $c\colon \K_{Z_e}(X \times X) \to \K_{Z_e}(Z_e)$, whose image lies in $\mathcal O^a(e)$. We proceed to prove that in fact every element of the latter subalgebra lies inside $\operatorname{Im} c$.


 We start with the following statement.
 Let $R$ be $\K_{Z_{Z_e}(s)^0}(\on{pt})$, set $Y:=\operatorname{Spec} R$.

Consider  $Y_1:= \on{Spec}(\oplus_{\gamma \in \Gamma_e^s} \mathcal O(Y^\gamma))^{\Gamma_e^s}$,
the spectrum of the ring of global functions on the inertia stack of $Y/\Gamma$. Any point of $Y_1$ (which is reduced, since $Y$ is smooth and $\Gamma_e^s$ is finite) can be interpreted as (the conjugacy class of) a pair $(x, \gamma)$, where $x \in Y^\gamma.$
Recall that we have the natural homomorphisms $\K_{Z_e(s)}(Z_e(s)) \rightarrow \mathcal{O}(Y)^{\Gamma^s_e} \rightarrow \mathcal{O}(Y_1)$ so $\mathcal{O}(Y_1)$ is a module over $\K_{Z_e(s)}(Z_e(s))$.


\begin{prop}\label{inertia} The natural morphism $\kappa\colon \K_{Z_{Z_e}(s)}(Z_{Z_e}(s)) \to \mathcal O(Y_1)$, $\mathcal F \mapsto \phi_{\mathcal F}$; $\phi_{\mathcal F}\colon (x, \gamma) \mapsto \operatorname{Tr}_{x} \mathcal F_{\gamma}$, becomes an isomorphism after completion at $1$ (as of $\K_{Z_e(s)}(Z_e(s))$-modules).

\end{prop}

\begin{proof}
The proof is based on the material of Appendix~\ref{appK}.

\textbf{Step 1.} First of all, let us note that, by Section \ref{compl}, 
$\widehat{\K_{Z_{Z_e}(s)}(Z_{Z_e}(s))}^{\mathfrak m_{Z_{Z_e}(s),1}} \simeq (\widehat{\K_{Z_{Z_e}(s)^0}(Z_{Z_e}(s))}^{\mathfrak m})^{\Gamma_e^s}$. 

(Here we complete 
the RHS at $\mathfrak{m}=(\mathfrak m_{Z_{Z_e}(s),1}) \subset \K_{Z_{Z_e}(s)^0}(\on{pt})$).

Let us note that $\K_{Z_{Z_e}(s)^0}(Z_{Z_e}(s)) \simeq \bigoplus_{\gamma \in \Gamma_e^s} \K_{Z_{Z_e}(s)^0}(Z_{Z_e}(s)^0\gamma)$. (Here we've chosen some set-theoretic lifting of $\Gamma_e^s$ to $Z_{Z_e}(s)$.)

Thus, it suffices to see the equality 
$$
\widehat{\K_{Z_{Z_e}(s)^0}(Z_{Z_e}(s)^0\gamma)}^{\mathfrak m} \simeq \widehat{\mathcal O(Y^\gamma)}^{\mathfrak m}.
$$

\textbf{Step 2.} Let $Z$ be the cover of $Z_{Z_e}(s)^0$, so that $[Z, Z]$ is simply connected. Functoriality of the construction of such a cover allows us to lift an automorphism $\operatorname{Ad}_{\gamma}$ of $Z_{Z_e}(s)^0$ to $Z$. Let us denote the latter automorphism by $A_{\gamma}$.

First of all, let $U_Z \subset Z/_{\on{ad}}Z=\on{Spec}\K_{Z}(\on{pt})$ be a $\Gamma^s_e$-invariant open neighbourhood of $1$ such that the composition $U_Z \subset Z/_{\on{ad}}Z \to Z_{Z_e}(s)^0/_{\on{ad}}Z_{Z_e}(s)^0$ is an isomorpism onto its image that we denote by $U_{Z_{Z_e}(s)^0}$. Let $\mathfrak{m}_U \subset \BC[U_{Z_{Z_e}(s)^0}]$ be the ideal corresponding to $\mathfrak{m}$ and let $\mathfrak{n}_U \subset \BC[U_Z]$ be its preimage. It follows from Proposition~\ref{disc} that we have an isomorphism 
\begin{equation*}\label{cover}\widehat{\K_{Z_{Z_e}(s)^0}(Z_{Z_e}(s)^0\gamma)|_{U_{Z_{Z_e}(s)^0}}}^{\mathfrak m_U}  
\simeq \widehat{\K_{Z}(Z_{Z_e}(s)^0\gamma)|_{U_{Z}}}^{\mathfrak n_U}.
\end{equation*}


Moreover, results of Proposition \ref{merk} show us that 
$$
\K_Z(Z_{Z_e}(s)^0\gamma) \simeq \K_{Z^2}(Z_{Z_e}(s)^0\gamma) \otimes_{\K_{Z^2}(\operatorname{pt})} \K_Z(\operatorname{pt}).$$

Here $Z^2$-action on $Z_{Z_e}(s)^0\gamma$ is obtained via the natural projection $Z \to Z_{Z_e}(s)^0$ and regular right-left translations of the latter group on itself.

Now, let us note that the action of $Z^2$ on $Z_{Z_e}(s)\gamma$ is transitive with a stabilizer of $\gamma$ being isomorphic to a finite cover of $Z_{Z_e}(s)^0.$ Let us denote this group by $S$. 

Then we get 
$$
\K_Z(Z_{Z_e}(s)^0\gamma) \simeq \K_{S}(\operatorname{pt}) \otimes_{\K_{Z^2}(\operatorname{pt})} \K_Z(\operatorname{pt});
$$
here $\K_{Z^2}(\operatorname{pt})$-action on $\K_S(\operatorname{pt})$ differs from the natural one (i.e., from the one induced by the diagonal embedding $Z \to Z^2$) by the $A_{\gamma}$-twisting.

\textbf{Step 3.} 



After restricting to $U_{Z}$ and completing at $\mathfrak{n}_U$ we obtain:
$$
\widehat{\K_{Z}(Z_{Z_e}(s)^0\gamma)|_{U_Z}}^{\mathfrak n_U} \simeq \K_{S}(\operatorname{pt})\otimes_{\K_{Z^2}(\operatorname{pt})} \widehat{\K_{Z}(\operatorname{pt})|_{U_Z}}^{\mathfrak n_U}=\K_{S}(\operatorname{pt})|_{U_Z \times U_Z}\otimes_{\BC[U_Z \times U_Z]} \widehat{\K_{Z}(\operatorname{pt})|_{U_Z}}^{\mathfrak n_U}.
$$
Recall now that $U_Z \iso U_{Z_{Z_e}(s)^0}$ and  the composition $(U_Z \times U_Z) \cap S \iso (U_{Z_{Z_e}(s)^0} \times U_{Z_{Z_e}(s)^0}) \cap Z_{Z_e}(s)^0$ is an isomorphism. 

We conclude that $\widehat{\K_{Z_{Z_e}(s)^0}(Z_{Z_e}(s)^0\gamma)}^{\mathfrak m} \simeq \K_{Z_e(s)^0}(\operatorname{pt})\otimes_{\K_{(Z_e(s)^0)^2}(\operatorname{pt})} \widehat{\K_{Z_e(s)^0}(\operatorname{pt})}^{\mathfrak m}$.





And, hence, 
$\widehat{\K_{Z_{Z_e}(s)^0}(Z_{Z_e}(s)^0\gamma)}^{\mathfrak m} \simeq \widehat{\mathcal O(Y^{\gamma})}^{\mathfrak m}$  (since $Y^\gamma$ is an intersection of a diagonal and of a graph of $\gamma$ acting on $Y$).

\end{proof}

\subsubsection{} Now we can start the proof of the surjectivity. 

Let us now, first of all, recall that (cf. Section \ref{appK1})  for any $Z_e$-variety $V$, $\widehat{\K_{Z_e}(V)}^s \simeq \widehat{\K_{Z_{Z_e}(s)}(V)}^1$. In particular, 
\begin{equation}\label{reduct}
\widehat{\K_{Z_e}(Z_e)}^s \simeq \widehat{\K_{Z_{Z_e}(s)}(Z_{Z_e}(s)})^1
\end{equation}
For a subalgebra $\mathcal{O}^a(e)$, let us denote the image of its completion at $1$ in the RHS of (\ref{reduct}) by $\widehat{\mathcal O}^1$.

Now, we have to prove that  the map \begin{equation}\label{cs}c_s\colon \widehat{K_{Z_{Z_e}(s)}(X^s \times X^s)}^1 \to \widehat{\mathcal O}^1, \end{equation} defined analogously to $c$, is surjective (recall that $X=\CB^{\BC^*}_e$).

\subsubsection{} The main idea is to use the Morita equivalence to get another realization of the image of the same map.

To do this, let us denote $\K_{Z_{Z_e}(s)^0}(X^s)$ by $M$. This $M$ is a module over $R=\K_{Z_{Z_e}(s)^0}(\on{pt})$, and $(\operatorname{End}_R(M))^{\Gamma_e^s}$ can be considered as $R^{\Gamma_e^s}$-algebra.

Then, (cf. Section \ref{compl}) the LHS of (\ref{cs}) is isomorphic as an algebra to $(\operatorname{End}_{\widehat{R}^{\mathfrak{m}}} \hat M^\mathfrak{m})^{\Gamma_e^s}$ (indeed, we have a natural $\Gamma^s_e$-equivariant homomorphism $\operatorname{End}_{\widehat{R}^{\mathfrak{m}}} (\hat M^\mathfrak{m}) \to \widehat{\K_{Z_{Z_e}(s)^0}(X^s \times X^s)}^{\mathfrak{m}}$ that becomes an isomorphism at the fiber at $\mathfrak{m}_{Z_{Z_e}(s)^0,1}$ and both of the algebras are free modules over $\widehat{R}^{\mathfrak{m}}$ so our map must be an isomorphism, passing to $\Gamma^s_e$-invariants we obtain the desired statement). 

\begin{lemma}\label{lemma_mor_equiv}
The algebra $(\operatorname{End}_{\hat{R}^{\mathfrak{m}}} \hat M^{\mathfrak{m}})^{\Gamma_e^s}$ is Morita-equivalent to $(\epsilon(\hat{R}^{\mathfrak{m}} \# \Gamma_e^s) \epsilon)^{\mathrm{opp}}$, where $\epsilon$ is an idempotent in the group algebra of $\Gamma_e^s$ corresponding to the set of $\Gamma_e^s$-characters appearing in $K(X^s)$ (and the completion should be understood in the same sense, as above).
\end{lemma}
\begin{proof}
Indeed, recall that indecomposable projective modules over $\hat{R}^{\mathfrak{m}} \# \Gamma_e^s$  have the form $\hat{R}^{\mathfrak{m}} \otimes \rho$, where $\rho$ is one of irreducible representations of $\Gamma_e^s$. 
Recall now that $\hat{M}^{\mathfrak{m}} \simeq \hat{R}^{\mathfrak{m}} \otimes \K(X^s)$ as $\widehat{R}^{\mathfrak{m}} \# \Gamma_e^s$-module. In particular, $\hat{M}^{\mathfrak{m}}$ is a projective $\hat{R}^{\mathfrak{m}} \# \Gamma_e^s$-module.

Recall also that $(\operatorname{End}_{\hat{R}^{\mathfrak{m}}} \hat M^{\mathfrak{m}})^{\Gamma_e^s} = \operatorname{End}_{\widehat{R}^{\mathfrak{m}} \# \Gamma_e^s} (\hat M^{\mathfrak{m}})$, so 
\begin{equation*}
\operatorname{Hom}_{\hat R^{\mathfrak{m}} \# \Gamma_e^s}(\hat M^{\mathfrak{m}}, -)\colon \widehat{R}^{\mathfrak{m}} \# \Gamma_e^s-\on{mod} \twoheadrightarrow  \operatorname{End}_{\hat{R}^{\mathfrak{m}} \# \Gamma_e^s} (\hat M^{\mathfrak{m}})^{\mathrm{opp}}-\on{mod}
\end{equation*}
is the quotient functor with the kernel consisting of $W \in \hat R^{\mathfrak{m}} \# \Gamma_e^s-\on{mod}$ such that $\on{Hom}_{\widehat{R}^{\mathfrak{m}} \# \Gamma_e^s}(M,W)=0$. 
We need to check that the kernel
consists of $W \in \hat R^{\mathfrak{m}} \# \Gamma_e^s-\on{mod}$ such that $\epsilon W=0$. To see that, it is enough to show that for any  representation $\rho$ of $\Gamma^s_e$ and any $\hat R^{\mathfrak{m}} \# \Gamma_e^s-$module $W$, we have $\on{Hom}_{\hat R^{\mathfrak{m}} \# \Gamma_e^s}(\hat R^{\mathfrak{m}} \otimes \rho,W) = 0$ iff $W_{\rho}=0$ This is clear since 
\begin{equation*}
\on{Hom}_{\hat R^{\mathfrak{m}} \# \Gamma_e^s}(\hat R^{\mathfrak{m}} \otimes \rho,W)=\on{Hom}_{\Gamma^s_e}(\rho,W).
\end{equation*}
\end{proof}

It follows from Lemma \ref{lemma_mor_equiv} that: \begin{equation}\label{cocenter}C((\operatorname{End} \hat{M}^{\mathfrak{m}})^{\Gamma_e^s})) \simeq C(\epsilon(\hat R^{\mathfrak{m}} \# \Gamma_e^s)\epsilon).\end{equation}

Thus, it follows that for some $\Gamma_e^s$-invariant Zariski neighborhood $U$ of $1 \in \operatorname{Spec} R$, one has
$$C((\operatorname{End} (M|_U))^{\Gamma_e^s}) \simeq C(\epsilon(R \# \Gamma_e^s)\epsilon|_U).$$

\subsubsection{} It follows from the results of Baranovsky (cf. \cite{Bar}) that there exists a map $R\#\Gamma_e^s \to \mathcal O(Y_1)$ which induces an isomorphism $C(R\#\Gamma_e^s) \iso \mathcal O(Y_1).$

One can construct such an isomorphism (let us call it $c_s^0$) as the $0$-th graded component of the following chain of equalities ($\on{HH}$ corresponds to Hochschild homology):

$$\operatorname{HH}^{R\#\Gamma_e^s}_{*} (R\#\Gamma_e^s) \simeq (\operatorname{HH}^{R}_{*} (R\#\Gamma_e^s))_{\Gamma_e^s} \simeq (\oplus_{\gamma \in \Gamma_e^s} \operatorname{HH}^{R}_{*} R\gamma)_{\Gamma_e^s}.$$ Indeed, we get

$$\on{HH}_0^{R\#\Gamma_e^s}(R\#\Gamma_e^s) \simeq (\oplus_{\gamma \in \Gamma_e^s} \mathcal O(Y^{\gamma}))_{\Gamma_e^s} \simeq (\oplus_{\gamma \in \Gamma_e^s} \mathcal O(Y^{\gamma}))^{\Gamma_e^s}.$$

Thus, one sees that the map $\widetilde{c_s}\colon R \#\Gamma_e^s \twoheadrightarrow C(R \# \Gamma_e^s) \xrightarrow{c_s^0} \mathcal O(Y_1)$ (where the first arrow is an evident surjection) can be characterized by the following condition: for any $r \in R \# \Gamma_e^s$, $p \in (\operatorname{Spec} R)/\Gamma_e^s$,
\begin{equation}\label{Bar}\langle \widetilde{c_s}(r)|_{\{p\} \times   \operatorname{Stab}_{\Gamma_e^s}(p)}, \chi_{\rho}|_{\operatorname{Stab}_{\Gamma_e^s}(p)} \rangle = \operatorname{Tr}_r {V^p_{\rho}},\end{equation}
where ${V^p_{\rho}}$ is the irreducible $R \# \Gamma_e^s$-module corresponding to the pair $(p,\rho|_{\on{Stab}_{\Gamma^s_e}(p)})$ (cf. proof of Proposition \ref{prop_im_equal} below). 

We will be interested in $\widetilde{c_s}|_{\epsilon (R\#\Gamma^s_e) \epsilon} =: c_s'$.

\begin{prop}\label{prop_im_equal} The isomorphism of  Proposition \ref{inertia} induces an isomorphisms
between 
the images of $c_s$ and of $\widehat{c_s'}^1$.
\end{prop}

\begin{proof}
Irreducible representations of $\epsilon(R\#\Gamma_e^s)\epsilon|_U$ are easily seen to be in bijection 
with pairs $(p, \rho)$, where
$p \in U$, 
$\rho \in \operatorname{Irrep} (\operatorname{Stab}_{\Gamma_e^s}(p))$, 
are such that $\rho$ appears in $\phi|_{\operatorname{Stab}_{\Gamma_e^s}(p)}$
 for some irreducible representation $\phi$ of $\Gamma_e^s$  with $\phi(\epsilon)\ne 0$.

 The cocenters of two Morita-equivalent algebras
are canonically isomorphic, the isomorphism is compatible with evaluating traces on the corresponding finite-dimensional representations.
In view of this observation, Proposition \ref{prop_im_equal} follows from the previous paragraph together with  (\ref{Bar}) and Lemma~\ref{Lmadm}. 
\end{proof}

\subsubsection{} Let us consider a subring $\mathcal O \subseteq \mathcal O(Y_1)$ consisting of functions $f$, such that at the every fiber of the projection $Y_1 \to Y/\Gamma^s_e$, they are sums of the restrictions of $\Gamma_e^s$-characters lying in the set corresponding to $\epsilon$.

\begin{remark}
Note that the notation ``$\mathcal{O}$'' is consistent with the previous one. Namely, our $\hat{\mathcal O}^{\mathfrak{m}}$ is the completion of $\mathcal{O}$ at $\mathfrak{m}$.
\end{remark}

Clearly, $\operatorname{Im}(c_s') \subseteq \mathcal O$. To finish the present subsection, it remains to show that, in fact, equality holds.

Indeed, let us take $f \in \mathcal O$. Since $\mathcal O \subseteq \mathcal O(Y_1)$, we can rewrite $f$ as $\widetilde{c_s}(r)$ for some $r \in R \# \Gamma$. But then, since $f \in \mathcal O$, one has $f = \widetilde{c_s}(r) = \widetilde{c_s}(\epsilon r \epsilon) \in \operatorname{Im} c_s'$.

\appendix

\section{K-theoretic appendix}\label{appK}
Here we summarize general properties of equivariant $K$-groups used in the paper. None of this is original; almost everything is taken from \cite{G}, from \cite{GE1}, or from \cite{Merk} (cf. below).

\subsection{Completion in equivariant $K$-theory} \label{appK1}

\subsubsection{} Let $H$ be an algebraic group acting on a  variety $X$ (we are not assuming that $X$ is smooth in this section, but we will actually apply the results of this section only to smooth $X$).
Pick a semisimple element $s \in H$ and set $Z:=Z_{H}(s)$.

Let $\gamma \subset H$ be the conjugacy class of $s$. It follows from \cite[Theorem 4.3]{GE1} that there exists the isomorphism of completions 
\begin{equation*}
\widehat{\K_Z(X^s)}^{s} \iso \widehat{\K_H(X)}^{\gamma}. 
\end{equation*}
In particular, we have an isomorphism of fibers $\K_H(X)_{\gamma} \simeq \K_Z(X^s)_s$.

\begin{remark}
If $X$ is smooth, then we have the natural (forgetting the equivariance and restriction) maps: 
\begin{equation}
\K_H(X) \to \K_Z(X) \to \K_Z(X^s).
\end{equation}
Their composition induces the isomorphism of completions $\widehat{\K_H(X)}^{\gamma} \iso \widehat{\K_Z(X^s)}^{s}$ (see \cite[Theorem 4.3 (b)]{GE1}).
\end{remark}

\subsubsection{}\label{Rem} Now we want to understand $\K_Z(X^s)_s$. For connected $Z$, it follows from  \cite[Theorem 1.1]{G} that $\K_Z(X^s)_s \simeq \K_Z(X^s)_1 \simeq \K(X^s)$; but this is not true in general, as we have already mentioned in Remark \ref{foot}. Nevertheless,  Graham explained to us that using the arguments from \textit{loc. cit.}, one can show that $\K_Z(X^s)_s$ surjects onto $\K(X^s)^{\pi_0(Z)}$; but this would not be enough for us.


\subsubsection{}\label{id} First of all (following \cite[Section 5.2]{GE1}) we  identify $\K_{Z}(X^s)_s$ with $\K_Z(X^s)_1$ as follows. Let us construct the isomorphism $\K_{Z}(X^s)_s \iso \K_Z(X^s)_1$. Consider a $Z$-equivariant coherent sheaf $\mathcal F$ on $X^s$. The action of $s \in Z$ is trivial on $X^s$ so the $Z$-equivariant structure on $\mathcal F$ induces the decomposition $\mathcal F=\bigoplus_{\chi \in \mathbb C^*}{\mathcal F}_\chi$, where  ${\mathcal F}_\chi \subset \mathcal F$ is the subsheaf of $\mathcal F$ on which $s$ acts via the multiplication by $\chi$.


The desired isomorphism $\K_{Z}(X^s)_s \iso \K_Z(X^s)_1$ is induced by the automorphism 
\begin{equation}\label{twist_action}
[\mathcal F] \mapsto \sum_{\chi \in \mathbb C^*}\chi[\mathcal F_\chi].
\end{equation}

\begin{remark}\label{rem_works_compl_shift} The same argument (cf. \textit{loc. cit.}) shows that $\widehat{\K_{Z}(X^s)}^s \simeq \widehat{\K_Z(X^s)}^1$.
\end{remark}

\subsubsection{} Now it remains to describe $\K_Z(X^s)_1$.
We set $A:=Z$, $Y:=X^s$. Our goal is to describe $\K_{A}(Y)_1$. Let $\mathfrak{m}_{A,1} \subset \K_A(\on{pt})$ be the maximal ideal of $1$. Recall that we have the restriction homomorphism $\K_A(\on{pt}) \rightarrow \K_{A^0}(\on{pt})$ and let $(\mathfrak{m}_{A,1}) \subset \K_{A^0}(\on{pt})$ be the ideal generated by the image of $\mathfrak{m}_{A,1}$.

\begin{remark}\label{rem_int_m_with_inv}
Note that $(\mathfrak{m}_{A,1})^{\pi_0(A)}$ is the maximal ideal in $\K_{A^0}(\on{pt})^{\pi_0(A)}$ that is equal to the ideal $\mathfrak{m}_{A^0,1}^{\pi_0(A)}$. In particular, $(\mathfrak{m}_{A,1})$ is equal to the ideal generated by $\mathfrak{m}_{A^0,1}^{\pi_0(A)}$ in $\K_{A^0}(\on{pt})$. 
\end{remark}

\begin{prop}{}\label{iso_mod_G_inv_m}
We have an isomorphism 
\begin{equation*}
\K_A(Y)/\mathfrak{m}_{A,1}\K_A(Y) \simeq (\K_{A^0}(Y)/(\mathfrak{m}_{A,1})\K_{A^0}(Y))^{\pi_0(A)}.
\end{equation*}
\end{prop}

\begin{proof} Consider the induction and restriction functors 
\begin{equation*}
\mathrm{res}\colon \operatorname{Coh}_A(Y) \to \operatorname{Coh}_{A^0}(Y),~\mathrm{ind}\colon \operatorname{Coh}_{A^0}(Y) \to  \operatorname{Coh}_A(Y).
\end{equation*}
Functor $\mathrm{res}$ here just sends $\mathcal{F} \in \operatorname{Coh}_A(Y)$ to itself considered as $A^0$-equivariant sheaf. 
\\
Functor $\mathrm{ind}$ sends $\mathcal{P} \in \operatorname{Coh}_{A^0}(Y)$ to the following sheaf:
consider the projection morphism $\pi \colon A \times Y \to Y$. We have an action of $A \times A^0$ on $A \times Y$ given  by $(a,a_0) \cdot (a',y)=(aa'a_0^{-1},a_0y)$.  

Moreover, morphism $\pi$ is $A \times A^0$-equivariant (action of $A$ on $Y$ is trivial); so $\pi^*(\mathcal P)$ is $A \times A^0$-equivariant. The action of $A^0$ on $A \times Y$ is free, so there exists the unique coherent $A$-equivariant sheaf $\widetilde{\mathcal P}$ on $(A \times Y)/{A^0}$ such that its pullback to $A\times Y$ is isomorphic to $\pi^*\mathcal P$. 

Finally, we define $\mathrm{ind}(\mathcal P):=\mu_*\widetilde{\mathcal P}$,  where $\mu\colon (A \times Y)/{A^0} \to Y$ is the natural ($A$-equivariant) morphism sending $[(a',y)]$ to $a'y$.

It is clear that, for any $\mathcal F$ (resp., any $\mathcal P$),
\begin{equation}
\mathrm{ind} \circ \mathrm{res} (\mathcal F) = \mathcal F \otimes \mathbb C[\pi_0(A)],\; \mathrm{res} \circ \mathrm{ind} (\mathcal P)=\sum_{g \in \pi_0(A)}g^*\mathcal P.
\end{equation}

Note that the functors $\mathrm{ind},\,\mathrm{res}$ are exact so they induce maps $[\mathrm{ind}],\, [\mathrm{res}]$ between $\K_A(Y)$ and $\K_{A^0}(Y)^{\pi_0(A)}$.
Note also, that both $[\mathrm{ind}]$ and $[\mathrm{res}]$ are linear for the action of $\K_A(\on{pt})$ (the action of $\K_A(\on{pt})$ on $\K_{A^0}(Y)^{\pi_0(A)}$ is via the restriction homomorphism $\K_A(\on{pt}) \rightarrow \K_{A^0}(\on{pt})^{\pi_0(A)}$), so (passing to the fiber at $\mathfrak{m}_{A,1}$) we
get maps in both directions between $\K_{A}(Y)/\mathfrak{m}_1\K_{A}(Y)$ and $(\K^{A^0}(Y)/(\mathfrak{m}_{A,1})\K^{A^0}(Y))^{\pi_0(A)}$.
We see that restriction of both $[\mathrm{ind}] \circ [\mathrm{res}]$, $[\mathrm{res}] \circ [\mathrm{ind}]$ to
\begin{equation*}
\K_{A}(Y)/\mathfrak{m}_1\K_{A}(Y),\, (\K^{A^0}(Y)/(\mathfrak{m}_{A,1})\K^{A^0}(Y))^{\pi_0(A)}
\end{equation*}
is just the multiplication by $|\pi_0(A)|$ so we obtain the desired isomorphism.
\end{proof}

\subsubsection{}\label{compl} Now let us note that $\K_{A^0}(Y)/(\mathfrak{m}_{A,1})\K_{A^0}(Y)$ is a finitely generated module over $R=\K_{A^0}(\on{pt})/({\mathfrak{m}}_{A,1})$, and its fiber over one is $\K(Y)$ (use \cite[Theorem 1.1]{G}).

Moreover, the ring $R$ is local commutative, so every flat module over it is free.

We conclude that if $\K_{A^0}(Y)/(\mathfrak{m}_{A,1})$  is flat over $R$ (that happens, for example, if $\widehat{\K_{A^0}(Y)}^{1}$ is flat over $\widehat{\K_{A^0}(\on{pt})}^{1}$
), then we have a (non canonical) $\pi_0(A)$-equivariant isomorphism: 
\begin{equation}
\K_{A^0}(Y)/(\mathfrak{m}_{A,1})\K_{A^0}(Y) \simeq \K(Y) \otimes R.
\end{equation}
Passing to $\pi_0(A)$-invariants and using Proposition \ref{iso_mod_G_inv_m} we get the isomorphism: 
\begin{equation}\label{B15}
\K_A(Y)/\mathfrak{m}_{A,1}\K_A(Y) \simeq (\K(Y) \otimes R)^{\pi_0(A)}.
\end{equation}

 \begin{remark}\label{rem_compl} The situation is similar to the one in Remark \ref{rem_works_compl_shift}: namely, under the same assumptions as above, one gets
\begin{equation}\label{B16}\widehat{\K_A(Y)}^{\mathfrak{m}_1} \simeq (\K(Y) \otimes \widehat{\K_{A^0}(\on{pt})}^{(\mathfrak{m}_{A,1})})^{\pi_0(A)}
.\end{equation} 
\end{remark}

 \begin{remark} Our ``local flatness assumption'' holds for $Y = \CB_e^{\mathbb C^*,s}$ and $A = Z_{Z_e}(s)$, to see this recall first that if $T(s) \subset Z_{Z_e}(s)$ is a maximal torus then $\K_{T(s)}(\CB_e^{\mathbb C^*,s})$ is free over $\K_{T(s)}(\on{pt})$ (same argument as in \cite[Lemma 1.10]{L3}). Now,  Let $Z$ be the cover of $Z_{Z_e}(s)^0$, so that $[Z,Z]$ is simply connected. Let $T \subset Z$ be the preimage of $T(s)$ in $Z$. It follows from Proposition \ref{disc} that we have $\widehat{\K_{Z_{Z_e}(s)^0}(X)}^{1} \simeq \widehat{\K_{Z}(X)}^{1}$ so it is enough to check that $\widehat{\K_{Z}(X)}^{1}$ is flat over $\widehat{\K_{Z}(\on{pt})}^{1}$. Recall now that by \cite[7.1]{Merk} we have the isomorphism $\K_T(X) \simeq \K_Z(X) \otimes_{\K_Z(\on{pt})} \K_{T}(\on{pt})$ so it remains to check that $\widehat{\K_T(X)}^{1}$ is flat over $\widehat{\K_T(\on{pt})}^{1}$ (we use fpqc descent here). This again follows from the identification $\widehat{\K_T(X)}^{1} \simeq \widehat{\K_{T(s)}(X)}^{1}$. 
\end{remark}


\subsubsection{Passing to a covering} 
Suppose that $G$ is a connected reductive group. Let $\pi\colon H \rightarrow G$ be  the covering of $G$ so that $[H, H]$ is simply connected. Let $\Gamma$ be the (finite central) kernel of $\pi$.

Suppose also that $H$ acts on some (smooth) algebraic variety $X$, and that $\Gamma$ lies in the kernel of this action. Recall that $\on{Spec}(\K_H(\on{pt}))=H/\!/_{\on{ad}}H$, $\on{Spec}(\bf{}K_G(\on{pt}))=G/\!/_{\on{ad}}G$, and the the natural map $H/\!/_{\on{ad}}H \to G/\!/_{\on{ad}}G$ is an \'etale morphism. Let $U_H \subset H/\!/_{\on{ad}}H$ be an open neighbourhood of $1 \in H/\!/_{\on{ad}}H$ such that the composition $U \subset H/\!/_{\on{ad}}H \to G/\!/_{\on{ad}}G$ is an isomorphism onto the image that we denote by $U_G \subset G/\!/_{\on{ad}}G$.

\begin{prop}\label{disc}
\begin{equation}
\K_{H}(X)|_{U_H} \simeq \K_{G}(X)|_{U_G}.
\end{equation}
\end{prop}

\begin{proof}
Let us note that the formula (\ref{twist_action}) (for $s$ being some $\gamma \in \Gamma$) gives a well-defined action of $\Gamma$ on $\K_{H}(X)$.

\textbf{Step 1.}  We claim that $\K_G(X)=\K_H(X)^{\Gamma}$.  Indeed, since any $H$-equivariant sheaf $\mathcal F$ has a direct sum decomposition $\mathcal F = \bigoplus_{\chi \in \operatorname{Char} \Gamma} \mathcal F_\chi$, there is a direct sum decomposition $\operatorname{Coh}_{H}(X) \simeq \bigoplus_{\chi \in \on{Char}\Gamma} \operatorname{Coh}_{H}(X)_{\operatorname \chi}$; -- and $\K_G(X) \simeq K_0(\operatorname{Coh}_{H}(X)_{\operatorname{triv}}) \otimes_{\mathbb{Z}} \mathbb{C} \simeq \K_H(X)^{\Gamma}$. 
We now need to compare the restrictions $K_H(X)|_{U_H}$, $(K_H(X)^\Gamma)|_{U_G}$.

\textbf{Step 2.} Set $K:=\K_H(X)$ and $L:=K^{\Gamma}$. Let us consider the restriction $K_{\Gamma U_H}$ of $K$ to the open set $\bigsqcup_{\gamma \in \Gamma} \gamma U$. 
Note that $L|_{U_G}=(K_{\Gamma U_H})^\Gamma$, so our goal is to identify $K|_{U_H}$ with $(K_{\Gamma U_H})^\Gamma$.

The identification of $K|_{\gamma U_H}$ with $K|_{U_H}$ from (\ref{twist_action}) gives a natural $\Gamma$-equivariant isomorphism:
$$K_{\Gamma U} = \bigoplus_{\gamma \in \Gamma} K|_{\gamma U} = K|_{U_H}^{\oplus |\Gamma|},$$
where $\Gamma$ acts on the RHS via permuting the factors. It follows that $K|_{U_H}$ is naturally isomorphic to $(K_{\Gamma U_H})^\Gamma$.

\begin{remark} 
As a corollary of Proposition \ref{disc} we conclude that $\widehat{\K_H(X)}^{1} \simeq \widehat{\K_G(X)}^{1}$.
\end{remark}



\end{proof}


\subsection{Modules over convolution algebras}\label{conv} 

Let $G$ be a (possibly, disconnected) algebraic group acting on a smooth variety $X$. Let us now assume that $G$ is reductive, and that the cycle morphisms $A_*(X^s) \otimes_{\mathbb{Z}} \mathbb{C} \to H_*(X^s,\BC)$, $A_*(X^s \times X^s) \otimes_{\mathbb{Z}} \mathbb{C}  \to H_*(X^s \times X^s,\mathbb{C})$, are isomorphisms for every semisimple $s \in G$.

(This is true, for example, for $X = \mathcal B_e^{\mathbb C^*}$ (see our discussion around the formula (\ref{Kun}).)

We also assume that $\K_{G^0}(X)$ is flat over $\K_{G^0}(\on{pt})$ in some neighbourhood of $1$. We claim that the following corollary of the above results holds. 

\begin{prop}\label{Thom} 
$(a)$ Every simple module over the algebra $\K_G(X \times X)$ is of the form $\K(X^s)_{\rho}$ for some semisimple $s \in G$ and an irreducible $\rho \in \operatorname{Irrep} (Z_{G}(s))$. 

$(b)$ Modules $K(X^s)_{\rho}$, $K(X^{s'})_{{\rho'}}$ are isomorphic iff $(s,\rho)$, $(s',\rho')$ lie in the same conjugacy class.
\end{prop}

Note  that similar results for $G$ being a finite group go back to Lusztig. 

\begin{proof}
\textbf{Step 1}. Set $Z:=Z_G(s)$. Note, first of all, that there exists a natural morphism $\phi\colon \K_Z(X^s) \to \K_Z(X^s)$ defined analogously to the formula (\ref{twist_action}). One can now consider a composition of $\phi$ and the forgetful morphism; we will call this map $\Phi$. 

One has $\Phi\colon \K_Z(X^s) \to \K(X^s)$; it is easy to see that it induces a morphism $\K_Z(X^s)_s \to \K(X^s)$. 

Let $\lambda(s)$ be a $Z$-equivariant Thom class of a normal bundle to $s$-fixed points inside $X$; let $\lambda_s$ be $\Phi(\lambda(s)^{-1})$. 

Let $\gamma_s \subset G$ be the conjugacy class of $s$. One can consider a morphism 
$$
\phi'\colon \K_G(X \times X)_{\gamma_s} \to \K_Z(X^s \times X^s)_s \xrightarrow{\phi_{12}} \K_Z(X^s \times X^s)_1 \to \K(X^s \times X^s) \xrightarrow{1 \boxtimes \lambda_s^{-1}} \K(X^s \times X^s).
$$

Here the first arrow is a natural restriction morphism $i^*$ for $i\colon X^s \times X^s \to X \times X$, the second one is analogous to $\phi$, and the third one is induced by $\lambda_s$ via the Kunneth formula.

We claim that $\phi'$ is well-defined and is a homomorphism of algebras.

\textbf{Step 2.} To prove this, it would be enough  to show that the morphism 
\begin{equation*}
\tilde{\phi}\colon \K_G(X \times X) \xrightarrow{i^*} \K_Z(X^s \times X^s) \xrightarrow{1 \boxtimes \lambda(s)^{-1}} \K_Z(X^s \times X^s) \to \K_Z(X^s \times X^s)_s
\end{equation*}
is a well-defined homomorphism of algebras.

This follows, analogously to the proofs of \cite[5.11.7]{CG} and \cite[5.11.10]{CG}, from \cite[Theorem 4.3]{GE1}.

\textbf{Step 3.} Moreover, by (\ref{B15}), one can see that  the image of $\phi'$ is a semisimple subalgebra $\K(X^s \times X^s)^\Gamma$ (where $\Gamma$ is the component group of $Z$), and that $\phi'$ can be identified with the morphism $\K_G(X \times X)_{\gamma_s}=: K \twoheadrightarrow K/\operatorname{Rad}(K)$.
\end{proof}

\subsection{Restriction of equivariance in $K$-theory.}
 

We start with the following standard lemma. 

\begin{lemma}\label{lem_every_emb_is_good}
Let $\iota\colon T \hookrightarrow C$ be an embedding of algebraic tori (over $\mathbb{C}$). Then there exists a collection of characters $\chi_i\colon C \rightarrow (\mathbb{C}^*)$, $i=1,\ldots,\on{dim}C-\on{dim}T$ such that 
\begin{equation}\label{eq_T_int}
T=\bigcap_{i=1}^{\on{dim}C-\on{dim}T} \on{ker}\chi_i
\end{equation} 
and $\bigcap_{i=1}^{k} \on{ker}\chi_i$ is a torus of dimension $\on{dim}C-k$ for every $k = 1,\ldots,\on{dim}C-\on{dim}T$.
\end{lemma}
\begin{proof}
Let $\Omega_C$, $\Omega_T$ be the character lattices of $C$, $T$. The embedding $\iota$ induces the surjection $\iota^*\colon \Omega_{C} \twoheadrightarrow \Omega_T$ (see, for example, \cite[Corollary 22.5.4 (iii)]{TaY}). The kernel of $\iota^*$ is free of rank $\on{dim}C-\on{dim}T$. Let $\chi_{1}, \ldots, \chi_{\on{dim}C-\on{dim}T}$ be the generators of this kernel. It follows from the definitions that the equality (\ref{eq_T_int}) holds. It remains to note that for every $k = 1, \ldots,\on{dim}C-\on{dim}T$, the quotient $\Omega_{C}/\on{Span}_{\BZ}(\chi_1,\ldots,\chi_k)$ is torsion-free of rank $\on{dim}C-k$. Indeed, let $\bar{\nu}_1,\ldots,\bar{\nu}_{\on{dim}T}$ be any generators of the lattice $\Omega_{T}$. Let $\nu_1,\ldots,\nu_{\on{dim}T} \in \Omega_C$ be such that $\iota^*(\nu_i) = \bar{\nu}_i$. It follows from the definitions that classes of  $\nu_1,\ldots,\nu_{\on{dim}T},\chi_{k+1},\ldots,\chi_{\on{dim}C-\on{dim}T}$ freely generate $\Omega_{C}/\on{Span}_{\BZ}(\chi_1,\ldots,\chi_k)$.
\end{proof}

Let $\iota\colon H \hookrightarrow G$ be an embedding of connected reductive algebraic groups with simply connected derived subgroups. 




We would like to discuss the following proposition. (The similar statement was proved in \cite{Merk} under different assumptions.)

\begin{prop}\label{merka}\label{merk}
Assume that $G$ acts on a smooth variety $X$. Then, the natural restriction morphism provides an isomorphism:
\begin{equation}\label{merk1} \K_G(X) \otimes_{\K_G(\operatorname{pt})} \K_H(\operatorname{pt}) \iso \K_H(X).
\end{equation}
\end{prop}

\begin{remark}
    Note that we do \textbf{not} make any assumptions concerning properness of $X$, nor existence of an equivariant affine paving.
\end{remark}

\begin{proof}
 \textbf{Step 1.} Let us  reduce the situation to the case of $H$ and $G$ being tori. 

Namely, let $T \hookrightarrow C$  be a   pair of compatible maximal tori of $H$, $G$.
 
 Note that, by \cite[Proposition 31]{Merk},  one has  the natural restriction isomorphism \begin{equation}\label{torus}\K_C(X) \simeq \K_G(X) \otimes_{\K_G(\operatorname{pt})} \K_C(\operatorname{pt}).\end{equation}
 In particular, for $W_G$ being a Weyl group of $G$, $\K_G(X) = \K_C(X)^{W_G}$ (here we use that $\K_G(\on{pt})=\mathbb{C}[G]^G \iso \mathbb{C}[T]^W$, see, for example, \cite[Theorem 4]{Serr}) and similarly  $\K_H(X) = \K_T(X)^{W_H}$. 

 The assertion of the Proposition in the toric case reads:

 $$
 \K_T(X) \simeq \K_C(X) \otimes_{\K_C(\operatorname{pt})} \K_T(\operatorname{pt});
 $$
 now we deduce:
 \begin{multline*}
 \K_H(X) \simeq \K_T(X)^{W_H} \simeq (\K_C(X) \otimes_{\K_C(\operatorname{pt})} \K_T(\operatorname{pt}))^{W_H} \simeq 
  (\K_G(X) \otimes_{\K_G(\operatorname{pt})} \K_C(\operatorname{pt}) \otimes_{\K_C(\operatorname{pt})} \K_T(\operatorname{pt}))^{W_H} = \\
  = (\K_G(X)\otimes_{\K_G(\operatorname{pt})} \K_T(\operatorname{pt}))^{W_H} \simeq \K_G(X)\otimes_{\K_G(\operatorname{pt})} \K_H(\operatorname{pt}).
 \end{multline*}

 \textbf{Step 2.} Now we can assume that $H$ is a subtorus inside some torus $G$. 
 
 Let us change the notation.  $T:=H$, $C:=G$. Using Lemma \ref{lem_every_emb_is_good} and the induction on $\on{dim}C - \on{dim}T$, we reduce to the case $T = \on{ker}\chi$ for some (primitive) nonzero character $\chi\colon C \rightarrow \mathbb{C}^*$. 

\begin{lemma} The natural restriction map $\operatorname{res}\colon \K_C(X) \otimes_{\K_C(\operatorname{pt})} \K_T(\operatorname{pt}) \to \K_T(X)$ is surjective.
\end{lemma}
 
 \begin{proof} The more powerful statement is proven in \cite[Proposition 26]{Merk}. The proof is as follows.

 In fact, we have to prove that the restriction map $r\colon \K_C(X) \to \K_T(X)$ is surjective.

 Let us rewrite it. Note that the character $\chi$, above turns $\mathbb A^1$ and $\mathbb{C}^*$ into $C$-varieties. Thus, by \cite[Corollaries 5, 12 and Theorem 8]{Merk}, we deduce that $r$ can be rewritten as:

 $$
 \K_C(X) \simeq \K_C(X \times \mathbb A^1) \twoheadrightarrow \K_C(X \times \mathbb{C}^*) = \K_C(X \times C/T) \simeq \K_T(X).
 $$
 The second arrow is a surjective restriction to an open subset (cf. \cite[Corollary 27]{Merk}), hence the lemma follows.
 \end{proof}

\textbf{Step 3.} What remains is to prove injectivity of $\operatorname{res}$. This is contained in \cite[Corollary 27]{Merk}.

More precisely, we can write down the following piece of the long exact sequence in $K$-theory:

$$
\K_C(X) \to \K_C(X \times \mathbb A^1) \to \K_C(X \times \mathbb{C}^*) \to 0.
$$
One has to prove that the leftmost  map can be identified with the multiplication by $1-\chi \in K_T(\operatorname{pt})$.

But the latter is evident from the projection formula: cf.~\textit{loc.~cit.}
\end{proof}




\subsection{Traces in $K$-theory}\label{traces}
Suppose that $X$ is a smooth proper variety with the same assumptions as in Section~\ref{conv}.

Then, by the Kunneth formula and Poincar\'e duality, \begin{equation}\label{iso}
\K(X \times X) \simeq \K(X) \otimes \K(X) \simeq \K(X) \otimes \K(X)^* \simeq \operatorname{End} \K(X).
\end{equation}

This gives a natural action of $\K(X \times X)$ on $\K(X)$. Moreover, it is well-known that the resulting algebra structure on $\K(X \times X)$ coincides with the one defined via the convolution product.

Let $\operatorname{Eu}$ stand for the Euler characteristic.

\begin{prop}\label{tr} For any $\mathcal F \in \K(X \times X)$, $\operatorname{Tr}_{\mathcal F}(\K(X)) = \operatorname{Eu}(\mathcal F \otimes^{\mathbb L} \mathcal O_{\Delta})$ where $\mathcal O_{\Delta}$ stands for the structure sheaf of  diagonal.
\end{prop}

\begin{proof}
Let us denote the linear functional $\mathcal F \mapsto \operatorname{Eu}(\mathcal F \otimes^{\mathbb L} \mathcal O_{\Delta})$ by $\tau$.

Since $\mathcal O_{\Delta}$ corresponds to the identity operator under (\ref{iso}), to show that $\tau = \operatorname{Tr}_{-}(\K(X))$, one has to establish that (cf. \cite{Il}):

1) $\tau(\mathcal F * \mathcal G) = \tau(\mathcal G * \mathcal F)$ where $*$ stands for the convolution;

2) $\tau(\mathcal O_{\Delta}) = \operatorname{dim} \K(X)$.

\textbf{Step 1.} 1) is checked by diagram chase as follows. 

Let $p_{12}$, $p_{23}$, and $p_{13}$ be the projections to the corresponding pairs of the arguments from $X \times X \times X$. Let $i_{\Delta}$ be the embedding of the diagonal into $X \times X$. Let $p$ denote a projection to a point.

Then one has:
\begin{equation}\label{tau}\tau(\mathcal F * \mathcal G) = p_*i_{\Delta}^*p_{13*}(p_{12}^*\mathcal F \otimes^{\mathbb L} p_{23}^* \mathcal G) = p_*(\mathcal F \otimes^{\mathbb L} \sigma^*\mathcal G),\end{equation}
where all of the functors are derived and $\sigma$ is the involution permuting the factors in $X \times X$ (the second equality is explained in more detail in the analogous calculation (\ref{basec})).

Equality 1) is now clear.

\textbf{Step 2.} We have to show that the Euler characteristic of the derived self-intersection of the diagonal copy of $X$ inside $X \times X$ is equal to the dimension of $H^*(X)$.

This follows easily from the Hodge theorem and the fact that $\operatorname{Tor}^i(\mathcal O_{\Delta}, \mathcal O_{\Delta}) = i_{\Delta*}\Omega_X^i$, since $X$ is smooth (for more details see \cite{Bh} and references therein).
\end{proof}

\subsubsection{The map $T$}\label{mapt} Here we  introduce a trace functor $T$ that plays a key role in the last section of the main text, and prove that it is actually 
a trace (commutator) functor.


We consider an action of an arbitrary algebraic group $G$ on a smooth 
variety $X$.

\begin{defn} Consider the diagram $X \times X \leftarrow G \times X \twoheadrightarrow G$, where the first map is the action map, the second one is the projection.

The functor $T\colon D^b(\on{QCoh}^G(X \times X)) \to D^b(\on{QCoh}^G(G))$ is defined to be the map given by this correspondence.
\end{defn}

\begin{lemma}\label{lemma:mapt}
For any ${\mathcal A}, {\mathcal B} \in D^b(\on{Coh}_G(X \times X))$ we have a canonical isomorphism 
$T(\mathcal A * \mathcal B)\simeq T(\mathcal B * \mathcal A)$
where $*$ stands for the convolution product.
\end{lemma}

\begin{proof}
    To fix ideas we first produce an isomorphism of fibers.

The fiber of $T(\mathcal A * \mathcal B)$ at $g \in G$ is precisely $R\Gamma(R\iota_g^*(\pi_{12}^*\mathcal A \otimes \pi_{23}^*\mathcal B))$ for $\iota_g$ being the embedding of the \textit{twisted diagonal} $\{(gx, y, x) \ | \ g \in G\}$ into $X^3$. One has:

$$T(\mathcal A*\mathcal B)_g = R\Gamma(\{(gx, y, x)\},\mathcal A * \mathcal B) = R\Gamma(\{(gx, y, y, x)\}, \mathcal A \boxtimes \mathcal B) = R\Gamma(\{(y, x, gx, y)\}, \mathcal B \boxtimes \mathcal A) = $$
$$=R\Gamma(\{(gy, gx, gx, y)\},\mathcal B \boxtimes \mathcal A) = R\Gamma(\{(gy, t, t, y)\},\mathcal B \boxtimes \mathcal A) =  T(\mathcal B * \mathcal A)_g.
$$

Here we have performed a change of the variable, and have used the $G$-equivariance of $\mathcal B$. It is easy to see that the identification of fibers $\nu_g\colon T(\mathcal A*\mathcal B)_g \iso T(\mathcal B * \mathcal A)_g$ as above is $G$-equivariant.



We now present a modification of the above argument that works in families and yields the isomorphism of objects in  $D^b(\on{QCoh}^G(G))$.

\textbf{Step 1.} Consider the universal twisted diagonal
$$\Delta_G\colon G \times X \times X \to G \times X^3, (g, x, y) \mapsto (g, gx, y, x).$$
Variety $G \times X^2$ is equipped with a natural morphism $\pi_G\colon G \times X^2 \twoheadrightarrow G$. By the definition, 
$$T(\mathcal A * \mathcal B) = \pi_{G*}\Delta_G^*(\pi_{12}^*\mathcal A \otimes \pi_{23}^*\mathcal B) = \pi_{G*}\Delta_G^*(\gamma^*\pi_{12}^*\mathcal A \otimes \sigma^*\pi_{23}^*\mathcal B),$$
where, as usual, all of the functors are derived, $\mathcal A$ and $\mathcal B$ stand (by a slight abuse of notation)  for the corresponding pull-backs under the second projection $G \times X^2 \to X^2$, $\pi_{ij}\colon G \times X \times \ldots \times X \to X \times X$ is the $ij$-th projection, $\gamma\colon (g, x, y) \to (g, gx, y)$, $\sigma\colon (g, x, y) \to (g, y, x)$. (Note the similarity between this calculation and the formula~(\ref{tau}).)

We can rewrite the last isomorphism using that $\otimes$ is the restriction of $\boxtimes$ to the diagonal: for $\Delta\colon G \times X \times X \hookrightarrow G \times X^4$ being the diagonal embedding,

$$T(\mathcal A * \mathcal B) = \pi_{G*}\Delta^*(\gamma^*\pi_{12}^*\mathcal A \boxtimes \sigma^*\pi_{23}^*\mathcal B) = \pi_{G*}\Delta^*(\sigma^*\pi_{23}^*\mathcal B \boxtimes \gamma^*\pi_{12}^*\mathcal A).$$

\textbf{Step 2.} In other words, for $\Delta_{G, 4}: G \times X^2 \to X^4,$ $(g, x, y) \mapsto (g, gx, y, y, x)$, $$T(\mathcal A * \mathcal B) = \pi_{G*}\Delta_{G, 4}^*(\mathcal A \boxtimes \mathcal B) = \pi_{G*}\widetilde{\Delta_{G, 4}}^*(\mathcal B \boxtimes \mathcal A).$$

for $\widetilde{\Delta_{G, 4}}$ being $\{(g, y, x, gx, y)\}$ as in the Step 1.

(We one more time abuse notation by writing $\mathcal A \boxtimes \mathcal B$ for the corresponding pullback from $X^4$ to $G \times X^4$.)


We now proceed as in the fiberwise argument above.

Namely,  let $\delta_{G, 4}$ be the embedding $G \times X^2 \to X^4$, $(g, x, y) \mapsto (g, gy, gx, gx, y)$. By the $G$-equivariance of $\mathcal{B}$, we get that $T(\mathcal A * \mathcal B) = \pi_{G*}\delta_{G, 4}^*(\mathcal B \boxtimes \mathcal A)$.

Moreover, let $\theta$ be the automorphism $(g, a, b, c, d) \to (g, a, g^{-1}b, g^{-1}c, d)$. Since $\theta$ is the $G$-equivariant automorphism ($G$ acts on itself via conjugation), and $\pi_G\theta = \pi_G$, 

$$T(\mathcal{A} * \mathcal{B}) = \pi_{G*}\theta_*\delta_{G, 4}^*(\mathcal B \boxtimes \mathcal A) = \pi_{G*}(\theta^{-1})^*\delta_{G,4}^*(\mathcal B \boxtimes \mathcal A).$$

Since the composition $\theta^{-1}\delta_{G, 4}: (g, x, y) \to (g, gy, x, x, y)$ is just $\Delta_{G, 4}\sigma$, and $\sigma$ is a $G$-equivariant automorphism, so that $\pi\sigma = \pi$, we are done.
\end{proof}

\end{document}